\documentclass[12pt,a4paper]{article}

\usepackage[utf8]{inputenc}
\usepackage[english]{babel}
\usepackage[T1]{fontenc}
\usepackage{amsmath}
\usepackage{amsfonts}
\usepackage{amssymb}
\usepackage{graphicx}
\usepackage{pdfpages}
\usepackage[left=3.5cm,right=3.5cm,top=3.5cm,bottom=3.5cm]{geometry}
\usepackage[all]{xy}
\usepackage{xcolor}
\definecolor{lien}{rgb}{0.7,0,0}
\definecolor{citation}{rgb}{0,0.5,0.5}
\usepackage[colorlinks=true, linkcolor=lien, citecolor= citation]{hyperref} 
\usepackage[width=0.75\textwidth]{caption}
\usepackage{placeins}
\usepackage{tikz-cd}
\usepackage[cal=boondox]{mathalfa}
\usepackage{amsthm}
\usepackage{float}
\usepackage{enumitem}
\theoremstyle{plain}
\newtheorem{thm}{Theorem}
\newtheorem{prop}[thm]{Proposition}
\newtheorem{lemme}[thm]{Lemma}
\newtheorem{coro}[thm]{Corollary}

\usepackage{hyperref}

\usepackage{siunitx} 

\theoremstyle{definition}
\newtheorem{defn}[thm]{Definition}

\def\N{\mathbb{N}}
\def\R{\mathbb{R}}
\def\Q{\mathbb{Q}}
\def\Z{\mathbb{Z}}

\def\E{\mathbb{E}}
\def\H{\mathbb{H}}

\def\di{\displaystyle}

\def\Tau{\mathrm{T}}

\def\eucl{\mathrm{eucl}}

\def\d{\mathrm{d}}

\def\noi{\noindent}

\newcommand{\pullback}[1]{{#1^*\scalprod{\cdot}{\cdot}}}
\newcommand{\scalprod}[2]{\left\langle #1, #2 \right\rangle}
\newcommand{\od}{{\rm Int} D^2}
    
\title{The Hyperbolic Plane in $\E^3$}
\author{
  Vincent Borrelli
  \thanks{Universit\'e Claude Bernard Lyon 1, CNRS, Institut Camille Jordan, F-69622 Villeurbanne, France.}
  \and
  Roland Denis\footnotemark[1] 
  \and
  Francis Lazarus \thanks{G-SCOP/Institut Fourier, Université Grenoble Alpes, France. This author is partially supported by the LabEx PERSYVAL-Lab (ANR-11-LABX-0025-01) funded by the French program Investissement d’avenir.}
  \and
  Mélanie Theillière \thanks{Department of Mathematics, University of Luxembourg. This author was supported by the Luxembourg National Research Fund (FNR) O21/16309996/HypSTER.}
  \and
  Boris Thibert \thanks{Univ. Grenoble Alpes, CNRS, Grenoble INP, LJK, 38000 Grenoble, France. This author is supported by the French ANR project MAGA (ANR-16-CE40-0014).}
}

\begin{document}
\maketitle
\begin{abstract}
  \noi
  We build an explicit $C^1$ isometric embedding $f_{\infty}:\H^2\to\E^3$ of the hyperbolic plane whose image is relatively compact. 
  Its limit set is a closed curve of Hausdorff dimension 1. Given an initial embedding $f_0$, our construction generates iteratively a sequence of maps by adding at each step $k$ a layer of $N_{k}$ corrugations. To understand the behavior of $df_\infty$ we introduce a \emph{formal corrugation process} leading to a \emph{formal analogue} $\Phi_{\infty}:\H^2\to \mathcal{L}(\R^2,\R^3)$. We show a self-similarity structure for $\Phi_{\infty}$. We next prove that $df_\infty$ is close to $\Phi_{\infty}$ up to a precision that depends on the sequence $N_*:= (N_{k})_k$.
  We then introduce the \emph{pattern maps} $\boldsymbol{\nu}_{\infty}^\Phi$  and $\boldsymbol{\nu}_{\infty}$, of  respectively $\Phi_{\infty}$ and $df_\infty$, that together with $df_0$ entirely describe the geometry of the Gauss maps associated to $\Phi_{\infty}$ and $df_\infty$. For well chosen sequences of corrugation numbers, we finally show an asymptotic convergence of $\boldsymbol{\nu}_{\infty}$  towards $\boldsymbol{\nu}_{\infty}^\Phi$ over circles of rational radii.
\end{abstract}

\vspace{2cm}
\noi

\newpage
\tableofcontents

\section{Introduction}
The Hilbert-Efimov theorem asserts that  the hyperbolic plane $\H^2$ does not admit any $C^{2}$ isometric embedding into the Euclidean 3-space $\E^3$~\cite{hilbert,efimov}. In contrast, the $C^1$ embedding theorem of Nash~\cite{n-c1ii-54} as extended by Kuiper~\cite{k-oc1ii-55,kuiperII} shows the existence of infinitely many $C^1$ isometric embeddings of $\H^2$ into $\E^3$. Since $\H^2$ is non-compact, the question of the behavior of such embeddings at infinity arises naturally. Following Kuiper~\cite{kuiperII} and De Lellis~\cite{delellis}, we consider the limit set $\text{Limset}(f)$ of a map $f:\H^2\to\E^3$. This is the set of points in $\E^3$ that are limits of sequences $(f(p_n))_n$, where $(p_n)_n$ is a sequence of points of $\H^2$ converging to infinity in the Alexandroff one point compactification of $\H^2$.
In 1955, Kuiper~\cite{kuiperII} exhibited an isometric embedding of $\H^2$ in $\E^3$ whose image is unbounded and with void limit set.
More than sixty years later, De Lellis~\cite{delellis} was able to extend this result in codimension two for any Riemannian $n$-dimensional manifold by prescribing the limit set. In the case of $\H^2$ the existence of a nonempty limit set in $\E^4$ implies the following counter-intuitive fact: any point of $\text{Limset}(f)$ is at infinite distance from every other point of $f(\H^2)$ for the metric induced by $\E^4$. Indeed, any path joining a point of $\H^2$ to a point on the boundary at infinity $\partial_{\infty}\H^2$ has infinite length as well as its image in $\E^4$.
\bigskip

\noi
In this paper we consider isometric embeddings of $\H^2$ in codimension one with nonempty limit set. We construct maps that naturally extend to the boundary at infinity $\partial_{\infty}\H^2$ so that their limit sets are images of $\partial_{\infty}\H^2$ by the extensions. We thus obtain maps defined over the compact domain $\overline{\H^2}=\H^2\cup \partial_{\infty}\H^2$ so that we may now study the regularity of the extensions transversely to the boundary.  We shall work with the Poincaré disk model $\H^2=(\od, h)$, where $D^2$ is the closed unit disk of the Euclidean plane $\E^2$ and $h=4\frac{dx^2+dy^2}{(1-x^2-y^2)^2}$ is the hyperbolic metric. We obtain the following results.

\begin{thm}\label{thm:H2}
There exists a map $f_\infty: D^2\to \E^3$ which is $\beta$-Hölder for any $0<\beta<1$ and whose restriction to the interior is a $C^1$-isometric embedding of the hyperbolic plane $\H^2.$ Its limit set is a closed curve of Hausdorff dimension~1.
\end{thm}

\begin{figure}[ht]
    \centering
	\includegraphics[width = 1\textwidth]{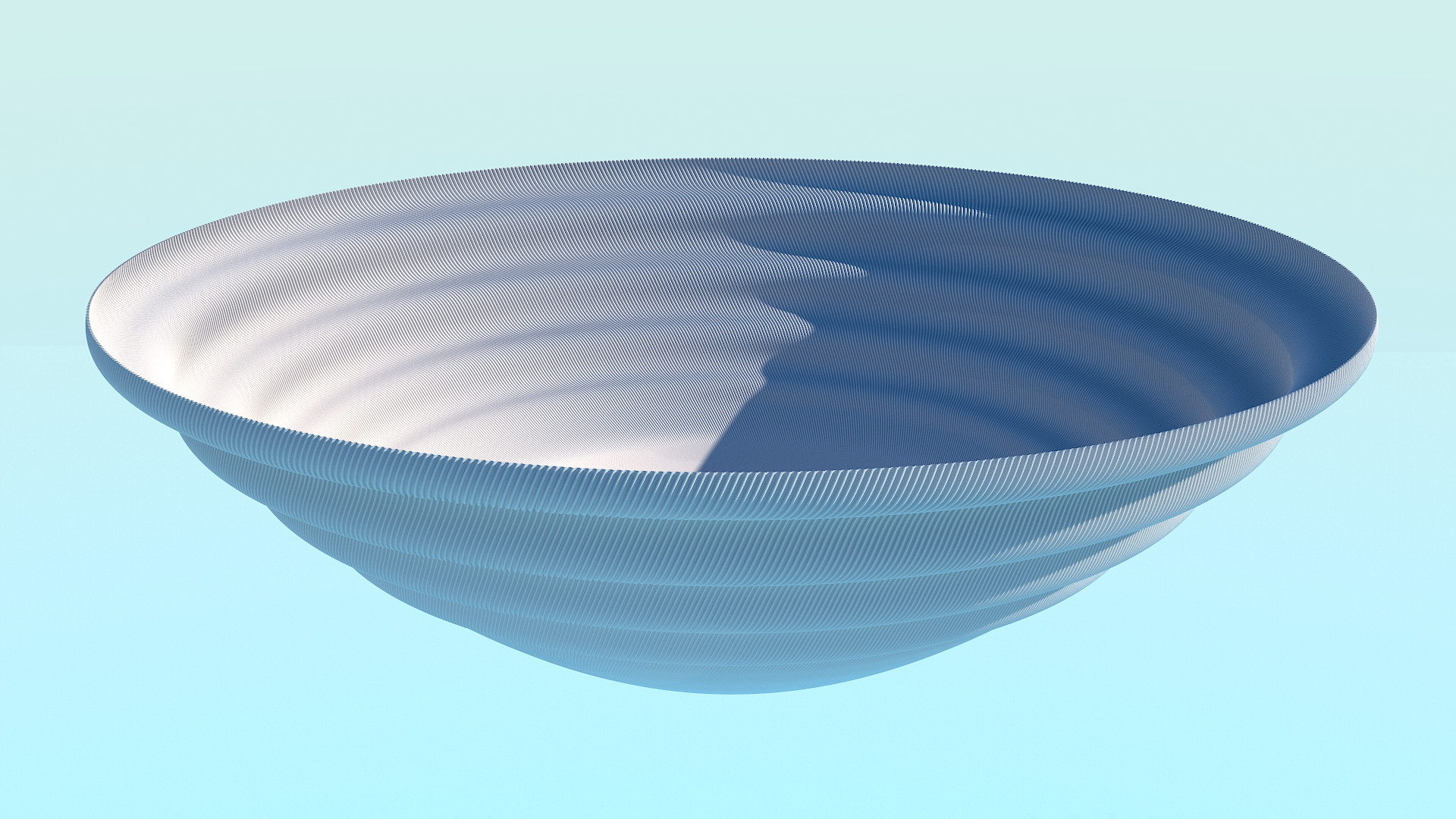}\\
	\caption{Global view of a 3-corrugated $C^1$-isometric embedding of $\H^2$ whose limit set is a closed curve of Hausdorff dimension 1. All graphics renderings in this article are obtained with the corrugation numbers 
          $\num{10}, \num{100}, \num{1 000}, \num{20 000}, \num{2 000 000}, \num{240 000 000}$. }
        \label{fig:global-view}
\end{figure}
\begin{figure}[ht]
    \centering
	\includegraphics[width = 1\textwidth]{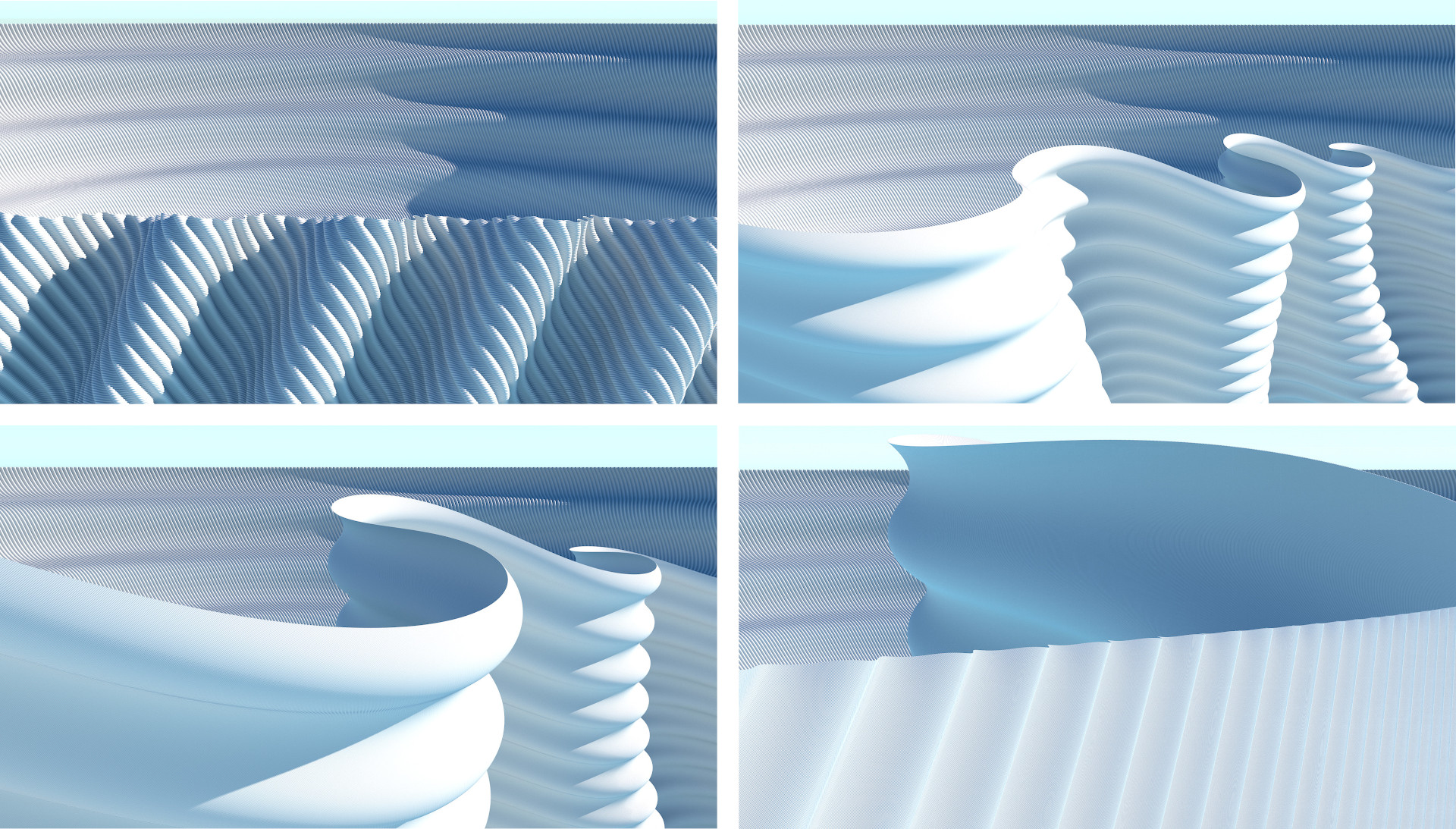}
	\caption{Zoom in on the limit set: at every scale a new layer of corrugations appears.}
        \label{fig:zoom-in}
      \end{figure}
      \noi
See Figures~\ref{fig:global-view} and~\ref{fig:zoom-in} for a graphic rendering of such an embedding.
The $\beta$-Hölder regularity in Theorem~\ref{thm:H2} is the best we can hope for: in any embedding with $\beta=1$ -- that is, of Lipschitz regularity -- the image of a radius of $D^2$ would have finite length which would be in contradiction with the fact that a curve going to infinity in the hyperbolic plane must have infinite length.

\paragraph{Embedding vs immersion.} We build our isometric embeddings by first choosing an initial embedding $f_0: D^2\to \E^3$ such that the induced pullback metric $g_0:= \pullback{f_0}$ is \emph{strictly short} over $\od$, i.e. $g_0< h$. We next choose a sequence of metrics $(g_k)_k$ defined on $D^2$ and converging to the hyperbolic metric $h$ on $\od$. We then construct a sequence of maps $(f_k)_k$ where $f_{k}$ is obtained from $f_{k-1}$ by adding waves with appropriate directions and frequencies. These \emph{corrugations} increase the lengths in such a way that $f_{k}$ is approximately $g_{k}$-isometric. If the convergence of the metrics is fast enough then the sequence $(f_k)_k$ converges to an $h$-isometric limit map $f_\infty$ of $C^1$ regularity.
The rate of convergence of the sequence $(g_k)_k$ has a strong impact on the properties of $f_\infty$, including the $C^1$ regularity and the embedded character. This rate forces the relative increase of lengths at each corrugation. When this increase is too large, the successive corrugations intersect each other as in Figure~\ref{fig:corrugations1-2}.
\begin{figure}[ht]
    \centering
    \includegraphics[width=.8\linewidth]{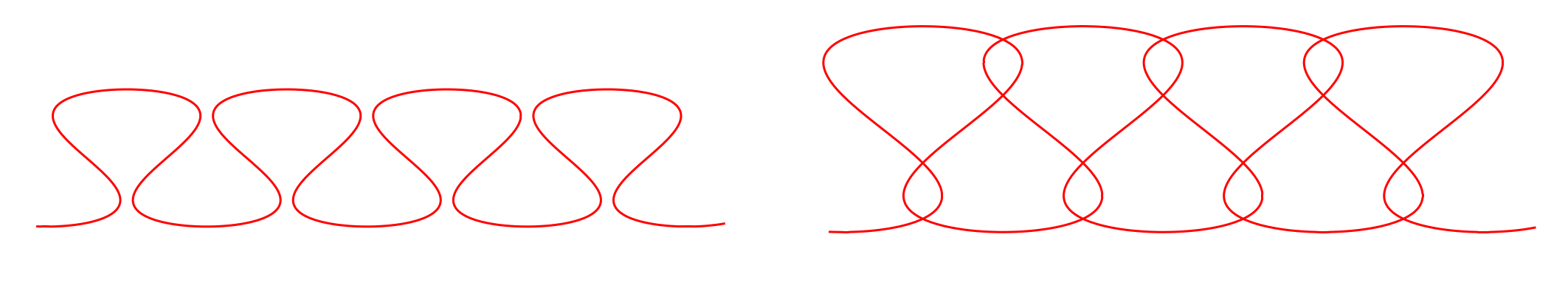}
    \caption{Increasing the length creates self-intersections.}
    \label{fig:corrugations1-2}
\end{figure}
Note that increasing the corrugation frequencies replaces the local behavior by an approximately homothetic figure, hence does not remove the self-intersections. This phenomenon is reminiscent of sawtooth curves of fixed length for which increasing the number of teeth replaces teeth by homothetic ones with the same slope. Although the hyperbolic metric explodes on the boundary, we manage to build a sequence $(g_k)_k$ of metrics whose rate of increase is bounded, allowing us to ensure that the limit surface is embedded. 

\paragraph{Fractal behavior and 3-corrugated embeddings.} A connection between $C^1$ isometric embeddings and fractal behavior has been observed for the construction of isometric embeddings of the flat torus and of the reduced sphere~\cite{PNAS,FOCM}. The self-similarity behavior arises from a specific construction that iteratively deforms a surface in a constant number of given directions. This leads us to introduce the notion of $3$-corrugated embedding.\\ 

\noi In the Nash and Kuiper approach, the map $f_k$ is built from $f_{k-1}$ by applying several corrugation steps. The number of steps and directions of the corrugations depend on the isometric default of $f_{k-1}$ with respect to $g_k$:

\begin{eqnarray*}
    D_k & := & g_{k}-\pullback{f_{k-1}}.
\end{eqnarray*}
This default is a positive definite symmetric bilinear form that can be expressed as a finite linear combination of squares of linear forms $\ell_{k,1},\cdots \ell_{k,n_k}$:
\begin{eqnarray}
\label{eq:decomposition_Dk_general}
D_k=\sum_{i=1}^{n_k}\eta_{k,i}\ell_{k,i}\otimes\ell_{k,i}
\end{eqnarray}
where each $\eta_{k,i}$ is a smooth positive function defined on the compact domain~$D^2$ (see \cite{n-c1ii-54}). 
Then, a sequence of intermediary maps 
$$f_{k-1}=f_{k,0},\; f_{k,1},...,\; f_{k,n_k}=f_k$$
is iteratively constructed by adding corrugations whose wavefronts have directions determined by $\ker \ell_{k,i}$. For every $i\in \{1,\ldots, n_k\}$, the amplitude and frequency of the corrugation are chosen so that $f_{k,i}$ is approximately isometric for the metric $\pullback{f_{k-1}} + \sum_{j=1}^i \eta_{k,j} \ell_{k,j} \otimes \ell_{k,j}$. In particular, $f_{k,n_k}$ is approximately $g_k$-isometric.\\

\noi
In our approach, we use an isometric default $D_{k,i} = g_k -\pullback{f_{k,i-1}}$ updated at each step $(k,i)$. This allows to construct $f_{k,i}$ using only the data of $f_{k,i-1}$ and $g_k$. We also manage to choose an initial map and three fixed linear forms $\ell_1,\ell_2,\ell_3$ such that~\eqref{eq:decomposition_Dk_general} hold for all $k$. This is in contrast with the Nash and Kuiper approach where the linear forms $\ell_{k,i}$ and their number $n_{k}$ depend on~$k$. Such a dependence prevents the appearance of any form of self-similarity in the limit.  
This motivates the introduction of the following notion, see Section~\ref{subsec:3-corrugated}. 
\begin{defn}
We say that the sequence $(f_{k,i})_{k,i}$ and its limit $f_{\infty}$ are \mbox{\emph{$n$-corrugated}} whenever $n_k=n$ and these linear forms are independent of $k$.
\end{defn}

\begin{prop}\label{prop:3-corrugated}
The embedding of Theorem~\ref{thm:H2} can be chosen $3$-corrugated.
\end{prop}
\noi
As illustrative examples, the embeddings of the square flat torus and of the reduced sphere in~\cite{FOCM,PNAS} are $3$-corrugated and have been shown to exhibit a fractal behavior. In the sequel, we will always consider 3-corrugated sequences.

\paragraph{Corrugation Process.}  There are several methods to construct $f_{k,i}$ from $f_{k,i-1}$, for instance by using Kuiper corrugations~\cite[Formula (6.7)]{k-oc1ii-55} or the convex integration formula of Gromov~\cite{g-pdr-86} and Spring~\cite[Formula (3.3) p. 35]{Spring-book}. All are based on the idea of corrugations  and introduce a free parameter $N_{k,i}\in\N^*$ called \textit{number of corrugations}. Here, we choose the corrugation process of~\cite{MT, massot2022formalising} to generate, at each step $(k,i)$, a map $f_{k,i}$ satisfying
$$\mu_{k,i}=f_{k,i}^*\langle\cdot,\cdot\rangle+O\left(\frac{1}{N_{k,i}}\right)$$
where $\mu_{k,i}:=\pullback{f_{k,i-1}}+\eta_{k,i}\ell_{k,i}\otimes\ell_{k,i}$. In our case of a 3-corrugated process, the  $\ell_{k,i}$ are given by a fixed set of three 1-forms $\ell_i$ with $i\in\{1,2,3\}$, and $\eta_{k,i}$ is the coefficient of $\ell_i\otimes\ell_i$ in the decomposition of $g_k-\pullback{f_{k,i-1}}$ in the basis $(\ell_j\otimes\ell_j)_{j\in\{1,2,3\}}.$ We then write
\begin{align}
   f_{k,i}=CP_i(f_{k,i-1},g_k,N_{k,i})\label{eq:CP_3-corrugated}
\end{align}
the map obtained from $f_{k,i-1}$ by a corrugation process with parameters $g_k$ and $N_{k,i}$.
Each maps $f_k=f_{k,3}$ is then approximately $g_k$-isometric in the sense that
\[\|\pullback{f_k}-g_k\|_{\infty}=O\left(\frac{1}{N_{k,1}}\right)+O\left(\frac{1}{N_{k,2}}\right)+O\left(\frac{1}{N_{k,3}}\right).\]
If the corrugation numbers are chosen large enough and if the convergence toward $h$ of the metrics $g_k$ is fast enough, namely if the series
\begin{eqnarray}
\label{eq:cv-metrics}
  \sum_k \|g_{k+1}-g_k\|_{K,\infty}^{1/2}<+\infty
\end{eqnarray}
converges on any compact set $K\subset \od$, then the sequence $(f_{k,i})_{k,i}$ converges toward a $C^1$ maps $f_{\infty}$ on $\od$ which is $h$-isometric, see Section~\ref{subsec:sequence_metrics_maps}.

\paragraph{Formal Corrugation Process.} We introduce in this work the notion of \emph{formal corrugation process}.
For a point $p\in D^2$, there are infinitely many isometric linear maps $L: (\R^2,\mu_{k,i}(p))\to \E^3$, i.e. that satisfy $\pullback{L}=\mu_{k,i}(p)$. By extension, there are infinitely many sections $L: D^2\to Mono(\R^2,\E^3)$ that are formal solutions of the $\mu_{k,i}$-isometric relation. In Section~\ref{subsec:target_differential}, we show that there exists a $\mu_{k,i}$-isometric section $L_{k,i}$ such that
\begin{align}
df_{k,i}=L_{k,i}+O(1/N_{k,i}). \label{eq:C1-proximity}
\end{align}
Moreover $L_{k,i}$ is given by a \emph{pointwise} formula of the form
\begin{eqnarray}
\label{eq:def-L_ki}
    L_{k,i} := df_{k,i-1} + \mathbf{z}_{k,i} \otimes \ell_{k,i}.
\end{eqnarray}
for some $\mathbf{z}_{k,i}:D^2 \rightarrow \E^3$ depending on $\eta_{k,i}$, $\ell_{k,i}$ and $N_{k,i}$. Note that there is no reason for the section $L_{k,i}$ to be holonomic\footnote{Recall that $L$ is holonomic if there exists a map $F: D^2\to \E^3$ such that $dF=L$.}. We say that $L_{k,i}$ is obtained by a \emph{formal corrugation process} and we write, analogously to~\eqref{eq:CP_3-corrugated}, 
\[L_{k,i}=FCP_i(df_{k,i-1},g_k,N_{k,i}).\]
The adjective \emph{formal} refers to the fact that $L_{k,i}$ is a formal solution of the $\mu_{k,i}$-isometric relation.

\paragraph{Formal analogues.}We now introduce the notion of formal analogues that will appear to encode the asymptotic behavior of the differential $df_\infty$. 
By iterating the formal corrugation process we obtain, in parallel to the 3-corrugated sequence $(f_{k,i})_{(k,i)}$,  a sequence $(\Phi_{k,i})_{(k,i)}$ of formal solutions given by
\[ \Phi_0=df_0 \quad\mbox{and}\quad \Phi_{k,i} = FCP_i (\Phi_{k,i-1},g_k,N_{k,i}).
\]
Under condition~\eqref{eq:cv-metrics} the sequence $(\Phi_{k,i})_{(k,i)}$ converges toward a $C^0$~map $\Phi_{\infty}$ on $\od$ (see Lemma~\ref{lem:sequence-Phi_ki-well-defined}). 
The map $\Phi_{\infty}$ is a formal solution for the $h$-isometric relation, that is
\[
    \pullback{\Phi_{\infty}}=h.
\]
If the corrugation process~\eqref{eq:CP_3-corrugated} were exact, that is, if $\pullback{f_{k,i}}=\mu_{k,i}$ for all $(k,i)$, then we would have $\Phi_{k,i} = df_{k,i}$. For that reason, we call $\Phi_{k,i}$ and $\Phi_{\infty}$ the \emph{formal analogues} of $df_{k,i}$ and $df_{\infty}$. In fact, the difference between the differential and its formal analogue depends on the corrugation numbers, see Section~\ref{subsec:comparing_Phi_df}. Theorem~\ref{thm:C1-density} below states that $df_{\infty}$ and $\Phi_{\infty}$ can be made arbitrarily close by choosing sufficiently large corrugation numbers. 


\begin{thm}\label{thm:C1-density}
Let $K\subset Int\,D^2$ be a compact set. For every $\varepsilon>0$ there exists a sequence of corrugation numbers $N_*=(N_{k,i})_{k,i}$ such that
\[\|\Phi_{\infty}-df_{\infty}\|_{K,\infty}\leq \varepsilon.
  \]
\end{thm}
\noi
In the spirit of $h$-principle, one may interpret Theorem~\ref{thm:C1-density} as a holonomic approximation result between the formal solution $\Phi_\infty$ and the holonomic section $df_\infty$~\cite{em-ihp-02}.
For a given initial embedding $f_0$ and a sequence of metrics $(g_k)_k$ the limit map $\Phi_\infty$ is well defined for every choice of the corrugation numbers $(N_{k,i})_{k,i}$. The theorem implies that for large enough corrugation numbers, the corresponding $\Phi_\infty$ are realized by holonomic sections up to $\varepsilon$. A notable observation is that $\Phi_k$ only depends on $\Phi_0$, on the three linear forms $\ell_1,\ell_2,\ell_3$, on the corrugation numbers $N_{1,1},...,N_{k,3}$  and on the values of $g_1,\dots, g_k$ at the considered point. This is not the case for $df_k$; even if the corrugation process~\eqref{eq:CP_3-corrugated} is pointwise, the derivatives of $f_k$ involve the derivatives of $g_k$. Hence, studying $\Phi_k$ and its limit $\Phi_\infty$ greatly simplifies the understanding of the geometry of $f_k$ and its limit $f_\infty$. Our numerical experiments actually show a remarkable similarity between $\Phi_k$ and $df_k$. See Figure~\ref{fig:regularite-rho}, where $K$ is a circle $\{(\rho,\varphi)\mid \rho = \mathrm{const}\}$.
\begin{figure}[h t]
    \centering
    \captionsetup{width=.9\linewidth}
	\includegraphics[width = 1\textwidth]{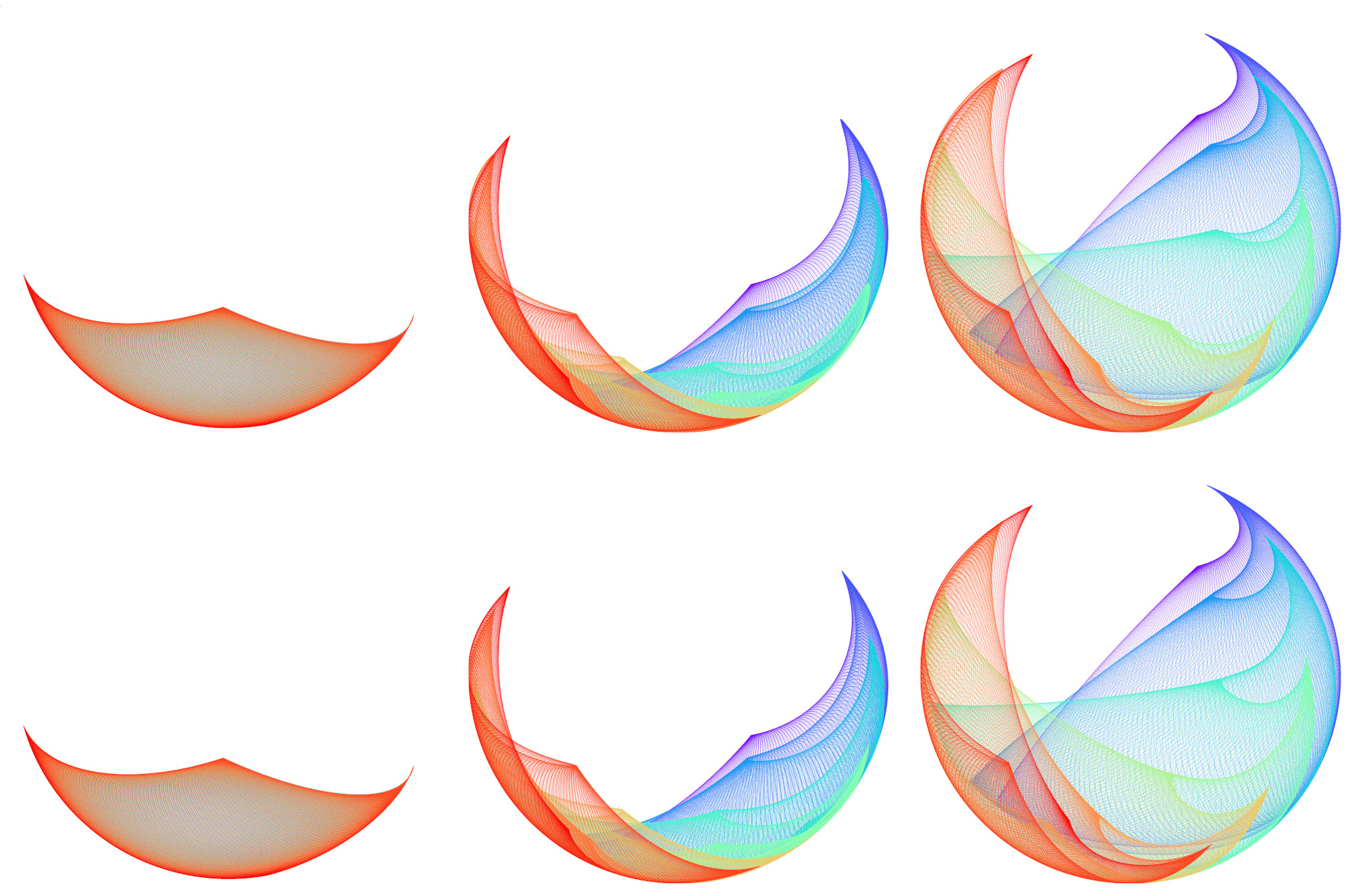}
	\caption{These pictures show the images into the 2-sphere of small circle arcs by the normal maps of the embedding $f_{2,2}$ and of its formal analogue $\Phi_{2,2}$. 
Top row, images by the normal map $\mathbf{n}_{2,2}$ of the embedding $f_{2,2}$ of arcs of amplitude $\frac{\pi}{700}$~of the circles $\{\rho=0.5\}$ (left), $\{\rho=0.7\}$ (center) and $\{\rho=0.9\}$ (right). Bottom row, images of the same arcs by the normal map $\mathbf{n}^{\Phi}_{2,2}$ of the formal analogue $\Phi_{2,2}$. The arcs are colored from red to blue. As their images are highly self-intersecting, the colors blend together.}
\label{fig:regularite-rho}	
\end{figure}

\paragraph{Formal normal map.} The formal analogue $\Phi_{\infty}$ gives a key to understand the behavior of the normal map $\mathbf{n}_{\infty}$ of~$f_{\infty}$. At each point $p=(\rho,\varphi)\in\od$, the formal analogue defines an oriented plane $Im\,(\Phi_{\infty}(p))$ and thus a unit \emph{formal normal}~$\mathbf{n}^\Phi_{\infty}(p)$. It follows by Theorem~\ref{thm:C1-density} that $\mathbf{n}^\Phi_{\infty}$ can be arbitrarily close to $\mathbf{n}_{\infty}$ if the corrugation numbers are conveniently chosen. We thus mainly focus on~$\mathbf{n}^\Phi_{\infty}.$ An obvious observation is that~$\mathbf{n}^\Phi_{\infty}$ is not $SO(2)$-rotationally symmetric despite the fact that the initial application has rotational symmetry and that all the metrics $g_k$ depend only on $\rho$. This is due to the accumulation of corrugations in three different directions which has the effect of destroying any rotational symmetry. However
the destruction of symmetry is not total: a rotational symmetry of finite order persists for 
$\mathbf{n}^\Phi_{\infty}$.
To understand the origin of these symmetries, it is convenient to introduce a variant of the normal map. To do so, we
complete the normal~$\mathbf{n}_{\infty}$ to an orthonormal basis $\mathbf{F}_{\infty}=(\mathbf{t}_{\infty},\mathbf{w}_{\infty},\mathbf{n}_{\infty})$ of $\E^3$. Here, $\mathbf{w}_{\infty}$ is the normalized derivative of $f_{\infty}$ in the direction $\ker \ell_1$ with the induced orientation of the oriented half-plane $\{ \ell_1 > 0\}$, and $\mathbf{t}_{\infty}$ is chosen to form a direct orthonormal basis. 
Similarly, we define $\mathbf{F}^{\Phi}_{\infty}$ and $\mathbf{F}_{0}$ by considering $\Phi_{\infty}$ and $df_0$ instead of~$df_{\infty}$. 
We now define the \emph{pattern maps} $\boldsymbol{\nu}_{\infty}$ and  $\boldsymbol{\nu}^\Phi_{\infty}$ as the coordinates in $\R^3$ of $\mathbf{n}_{\infty}$ and $\mathbf{n}^\Phi_{\infty}$ in the basis $\mathbf{F}_0$, so that 
\begin{equation}
\label{eq:normal-normal_pattern}
\mathbf{n}_{\infty}=  \mathbf{F}_{0} \cdot \boldsymbol{\nu}_{\infty}
\quad\mbox{and}\quad
\mathbf{n}^\Phi_{\infty}=  \mathbf{F}_{0} \cdot \boldsymbol{\nu}^{\Phi}_{\infty}.
\end{equation}
Obviously, the behavior of these pattern maps depend on the corrugation sequence $N_*=(N_{k,i})_{k,i}$ and more specifically on the greatest common divisor of the subsequences $(N_{k,2})_{k\in\N^*}$ and $(N_{k,3})_{k\in\N^*}$ of $N_*$, which we denote by
\[\Omega=\Omega[N_*].
  \]
We show in Lemma~\ref{lem:periodicity_formal_pattern} that the formal pattern map
\[\varphi\longmapsto\boldsymbol{\nu}^{\Phi}_{\infty}[N_*]\left(\rho,\varphi\right) \]
is $\frac{2\pi}{7\Omega}$-periodic. Here we have used the notation $[N_*]$ to  emphasize the dependency on the corrugation numbers. 
Since $\mathbf{F}_{0}$ has rotational symmetry, the following proposition follows:
\begin{prop}
  \label{prop:formal-normal-periodic}
 The formal normal
$\mathbf{n}_{\infty}^{\Phi}$ has rotational symmetry of order $7\Omega$.
\end{prop}

\paragraph{Self-similarity behavior of $\mathbf{n}^\Phi_{\infty}$.}  As stated in Theorem~\ref{thm:C1-density}, the difference between the formal normal $\mathbf{n}^{\Phi}_\infty$ and the Gauss map $\mathbf{n}_\infty$  of $f_\infty$ can be made arbitrarily small, see Figure~\ref{fig:regularite-rho}. This motivates the study of the formal normal. We exhibit below a self-similarity behavior of $\mathbf{n}^{\Phi}_\infty$ along the circle $\Gamma=\{\rho=\mathrm{const}\}$:  the image  $\mathbf{n}^{\Phi}_{\infty}(\Gamma)$ approximately decomposes into a finite union of rotated copies of $\mathbf{n}_{\infty}^{\Phi}(\Gamma_1^0)$, where $\Gamma_1^0$ is the arc $\{(\rho,\varphi)\,|\, 0\leq\varphi\leq\frac{2\pi}{7\Omega}\}$. This decomposition occurs at finer scales as stated in Proposition~\ref{prop:fractal-normal}  below.
For any integer $\ell\geq1$, we introduce similarly as above
\begin{itemize}
\item  $\mathbf{F}^{\Phi}_{\ell}$, the frame $(\mathbf{t}^{\Phi}_{\ell},\mathbf{w}^\Phi_{\ell},\mathbf{n}^{\Phi}_{\ell})$ obtained by considering $\Phi_{\ell,3}$ instead of $df_\infty$, 
  \item $\Omega_\ell=\Omega_\ell[N_*]$ the greatest common divisor of the subsequences $(N_{k,2})_{k\geq \ell}$ and $(N_{k,3})_{k\geq \ell}$. 
\end{itemize}
Observe that  $\Omega_1=\Omega$. For any $j\in\{0,\cdots,7\Omega_\ell-1\}$, we also introduce the following notations.
\begin{itemize}
\item $\mathrm{Lip}_\rho(\mathbf{F}_\ell^{\Phi})$ for the Lipschitz constant of $\varphi\mapsto\mathbf{F}^\Phi_{\ell}(\rho,\varphi) $,
\item $\mathrm{rot}_\ell^j$ for the rotation mapping $\mathbf{F}_{\ell-1}^{\Phi}(\rho,0)$ to $\mathbf{F}_{\ell-1}^{\Phi}(\rho,\frac{2j\pi}{7\Omega_{\ell}})$, and
  \item $\Gamma_\ell^j$  for the arc $\{(\rho,\varphi)\,|\, \frac{2j\pi}{7\Omega_\ell}\leq\varphi\leq\frac{2(j+1)\pi}{7\Omega_\ell}\}$.
\end{itemize}
\begin{prop}
\label{prop:fractal-normal}
Let $\rho\in\, ]0,1[$, $\ell\geq 1$, $j\in\{0,\cdots,7\Omega_{\ell}-1\}$. Denoting by $\d_{Hauss}$ the Hausdorff distance, we have
\begin{equation}
    \label{eq:dist_hausdorff}
  \d_{Hauss}\left( 
\mathbf{n}^{\Phi}_{\infty}(\Gamma^j_{\ell}) ,\mathrm{rot}_{\ell}^j \circ \mathbf{n}^{\Phi}_{\infty}(\Gamma_{\ell}^0)
\right) 
\leq \frac{4\pi \mathrm{Lip}_\rho(\mathbf{F}_{\ell-1}^{\Phi})}{7\Omega_{\ell}}.  
\end{equation}
In particular, we have 
\[
\d_{Hauss}\left( 
\mathbf{n}^{\Phi}_{\infty}(\Gamma) ,\bigcup_{j=0}^{7\Omega_{\ell}-1}\mathrm{rot}_{\ell}^j \circ \mathbf{n}^{\Phi}_{\infty}(\Gamma_{\ell}^0)
\right) 
\leq \frac{4\pi \mathrm{Lip}_\rho(\mathbf{F}_{\ell-1}^{\Phi})}{7\Omega_{\ell}}.
\]
\end{prop}

\noi
The upper bound of Inequation~\eqref{eq:dist_hausdorff} can be made arbitrarily small by a convenient choice of the sequence $(N_{k,i})_{k,i}.$ Indeed, the Lipschitz constant $\mathrm{Lip}_\rho(\mathbf{F}_{\ell-1}^{\Phi})$ depends on the corrugation numbers $N_{k,i}$ with $k\leq \ell-1$ while $\Omega_{\ell}$ only depends on $N_{k,2}$ and $N_{k,3}$ with $k\geq \ell.$ 
The appearance of the number~7 in the denominator of Inequation~\eqref{eq:dist_hausdorff}  results from a choice in the construction of $\Phi_{\infty}$ and has no special meaning, see Lemma~\ref{lem:parameter-a}.
Corollary~\ref{coro:fractal-normal2} below shows that this decomposition applies recursively. Namely, the image $\mathbf{n}^{\Phi}_{\infty}(\Gamma^j_\ell)$  approximately decomposes  into a union of $\Omega_{\ell+1}/\Omega_\ell$ copies of  $\mathbf{n}^{\Phi}_{\infty}(\Gamma_{\ell+1}^0)$.
\begin{coro}
\label{coro:fractal-normal2}
Let $\ell \geq 1$ and  $j\in\{0,\cdots,7\Omega_\ell-1\}$, we have
\[
\d_{Hauss}\left( 
\mathbf{n}^{\Phi}_{\infty}(\Gamma^j_\ell)\ ,\bigcup_{j'=j\frac{\Omega_{\ell+1}}{\Omega_\ell}}^{(j+1)\frac{\Omega_{\ell+1}}{\Omega_\ell}-1}
\mathrm{rot}_{\ell+1}^{j'} \circ \mathbf{n}^{\Phi}_{\infty}(\Gamma_{\ell+1}^0)
\right) 
\leq \frac{4\pi \mathrm{Lip}_\rho(\mathbf{F}_{\ell}^{\Phi})}{7\Omega_{\ell+1}}.
\]
\end{coro}

\begin{figure}[H] 
    \centering
    \captionsetup{width=.9\linewidth}
	\includegraphics[width = 0.8\textwidth]{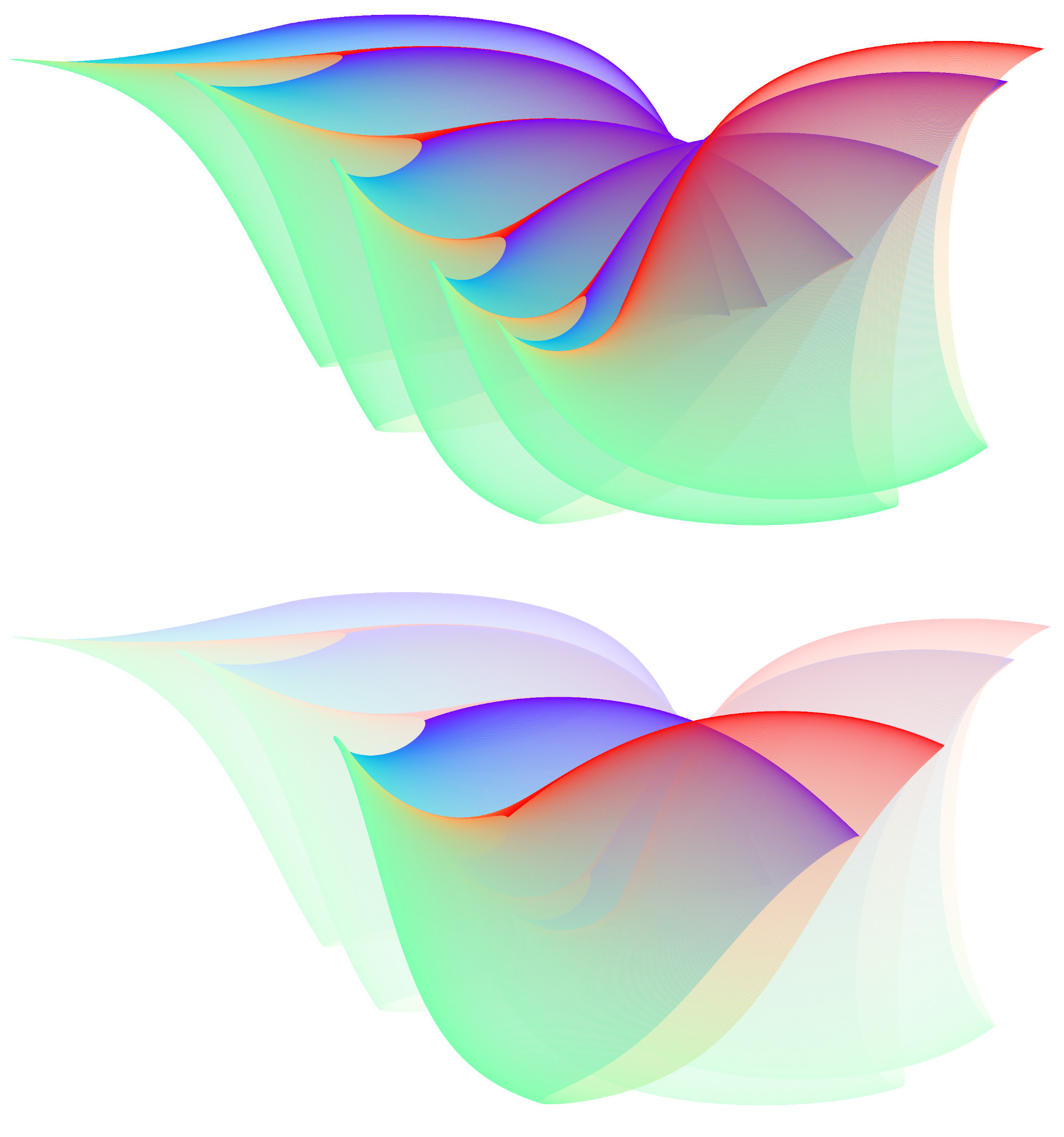}
	\caption{For any $j$, the image by $\mathbf{n}^{\Phi}_{2,2}$ of the arc $\Gamma_{1,2}^j$ of amplitude $\frac{2\pi}{7\Omega_{1,2}}=\frac{2\pi}{7\Omega}$ is made of $\Omega_{1,3}/\Omega_{1,2}=10$ approximated copies of  $\mathbf{n}^{\Phi}_{2,2}(\Gamma_{1,3}^0)$. (Refer to  Section~\ref{sec:similarity-scale} for the notations.) For the readability of the figure, only the image of half $\Gamma_{1,2}^0$,  an arc of amplitude $\frac{\pi}{7\Omega}$, with 5 copies is shown. Indeed, the image of the other half almost overlays the first half, giving the illusion of a single image traveled back and forth. Here $\rho=0.99$ and each copy, colored from red to blue, is represented in a domain of a spherical coordinate system. In the bottom figure, only one copy is highlighted while the other four are faded out.
	}
\label{fig:subpatterns}	
\end{figure}

\noi
In Figures~\ref{fig:regularite-rho}, \ref{fig:subpatterns} and \ref{fig:subsubpatterns}, we visualize $\mathbf{n}^{\Phi}_{2,2}$ and $\mathbf{n}^{\Phi}_{2,3}$ instead of $\mathbf{n}^{\Phi}_\infty$ for obvious numerical limitations. However, we show in Section~\ref{sec:similarity-scale} (Proposition~\ref{prop:similarity-any-scale}) that $\mathbf{n}_{k,i}^{\Phi}$ also exhibits a self-similarity structure in the sense that
the image by $\mathbf{n}_{k,i}^{\Phi}$ of each subarc of some regular subdivisions of $\Gamma$ is approximately made of rotated copies of the image of an even smaller arc. 

\begin{figure}[H] 
    \centering
    \captionsetup{width=.9\linewidth}
	\includegraphics[width = 0.8\textwidth]{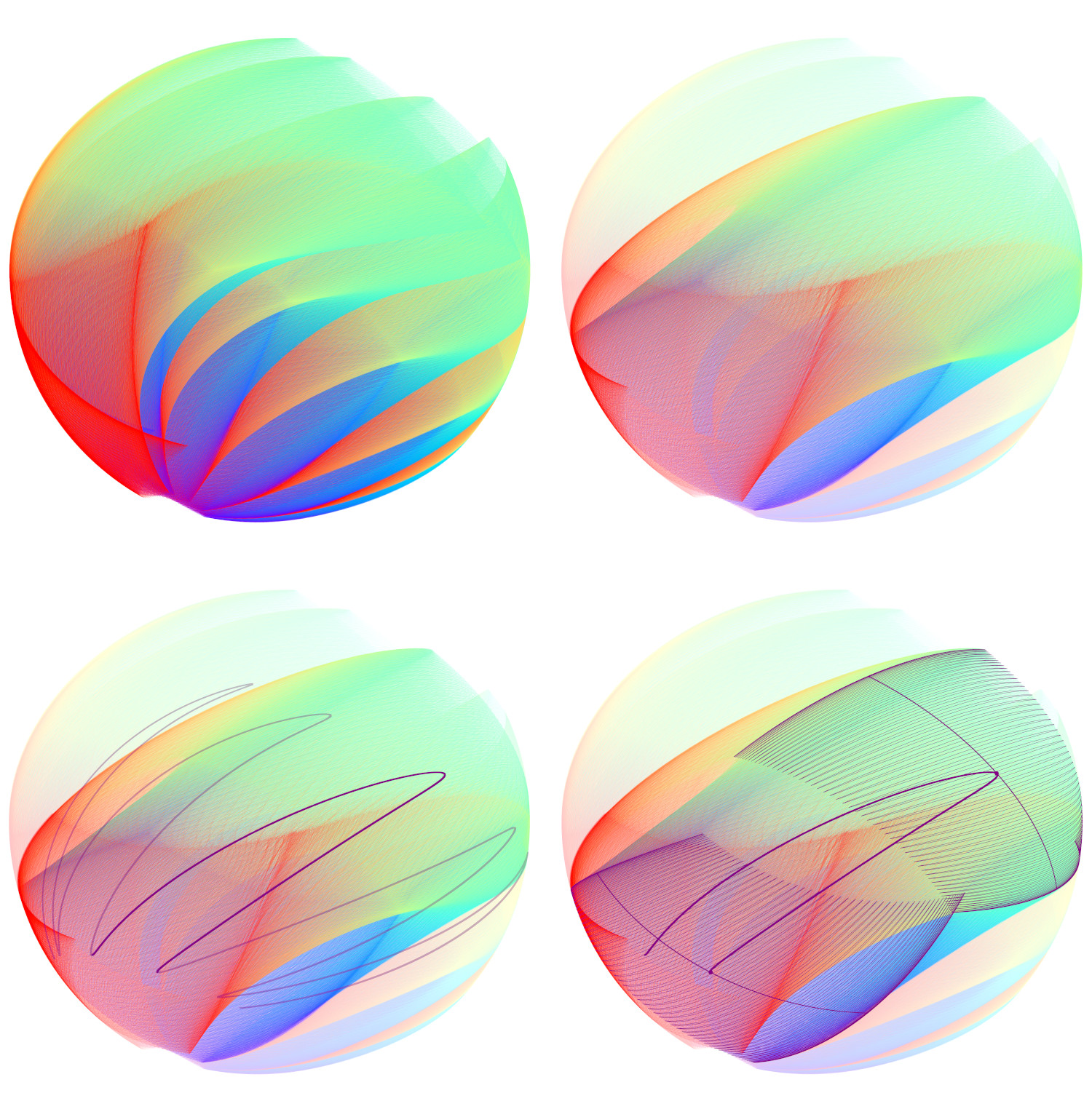}
	\caption{Self-similarity structure of the formal normal along the circle of radius $\rho=0.7$. 
\textbf{Top left}: image by $\mathbf{n}^{\Phi}_{2,3}$ of the first half $\Gamma_{1,2}^0$, an arc of amplitude $\frac{2\pi}{14\Omega_{1,2}} = \frac{\pi}{700}$. \textbf{Top right}: only the image by $\mathbf{n}^{\Phi}_{2,3}$ of $\Gamma_{1,3}^{3}\subset \Gamma_{1,2}^0$ is highlighted.
	\textbf{Bottom left}: An overlay is added in purple that corresponds to the image by $\mathbf{n}^{\Phi}_{2,1}$ of the first half of $\Gamma_{1,2}^0$. The thick part of this purple curve is the image  of $\Gamma_{1,3}^{3}$.
	\textbf{Bottom right}: Two overlays are finally added, also in purple. They correspond to the images by $\mathbf{n}^{\Phi}_{2,3}$ of  the arcs $\Gamma_{2,1}^{0}$ and $\Gamma_{2,1}^{1000}$ of amplitude $\frac{2\pi}{7\Omega_{2,1}} = \frac{\pi}{\num{7000000}}$. The image $\mathbf{n}_{2,3}(\Gamma_{1,3}^{3})$ is approximately composed of $\Omega_{2,1}/\Omega_{1,3}=2000$ rotated copies of  $\mathbf{n}^{\Phi}_{2,3}(\Gamma_{2,1}^0)$ distributed along the thick curve $\mathbf{n}^{\Phi}_{2,1}(\Gamma_{1,3}^{3})$. }
\label{fig:subsubpatterns}	
\end{figure}
\noi
\paragraph{Regularity of $\mathbf{n}^\Phi_{\infty}$.} We then focus on the regularity of $\mathbf{n}^\Phi_{\infty}$. Proposition~\ref{cor:regularity-n} states that this regularity is the same at every point of a circle centered at the origin. Now, if the normal~$\mathbf{n}^\Phi_{\infty}$ were of class $C^1$, we could extract the extrinsic curvatures from the shape operator $-{\rm d} \mathbf{n}^\Phi_{\infty}$.  Here, the normal $\mathbf{n}^\Phi_{\infty}$ is not $C^1$ but we may consider the modulus of continuity of $\mathbf{n}^\Phi_{\infty}$.  Proposition~\ref{prop:formal-Weierstrass} states that this modulus, when
restricted to the circles $\{\rho=Cte\}$, is related to the modulus of a Weierstrass-like function:
\[W_{\rho}(\varphi):=\sum_{k=1}^{\infty}\left( \sum_{i=1}^3\alpha_{k,i}^{\Phi}(\rho)\cos(b_{k,i}\varphi+c_{k,i}(\rho))\right).
\]
In this expression, the coefficient $\alpha_{k,i}^{\Phi}(\rho)$ can be made explicite, see  Lemma~\ref{lem:alpha-beta}, and has polynomial growth in $\rho^{k}$.
We thus expect that the regularity of 
\[\varphi\longmapsto \mathbf{n}^\Phi_{\infty}(\rho,\varphi)\]
decreases as $\rho$ tends towards 1. We have observed this phenomenon numerically, as illustrated in Figure~\ref{fig:regularite-rho}.

\paragraph{Asymptotic behavior.}
In Theorem~\ref{thm:C1-density}, the approximation of $df_\infty$ by its formal analogue $\Phi_\infty$ depends on the sequence $N_*$ of corrugation numbers:
\[\|\Phi_{\infty}[N_*]-df_{\infty}[N_*]\|_{K,\infty}\leq \varepsilon.
\]
Since $N_*$ itself depends on $\varepsilon$, this prevents us from transferring geometric properties of a given $\Phi_{\infty}[N_*]$ to $df_{\infty}[N_*]$ with arbitrary precision.
However, this arbitrary precision can be achieved over the circles of rational radius $\Gamma=\{\rho=\frac{p}{q}\}$. In Theorem~\ref{thm:image_absolute_normal_circle}, we exhibit a sequence of isometric embeddings whose normal patterns converge over those circles towards the formal normal of $\Phi_{\infty}[N_*]$ for any given $N_*$. The key step is to replace the sequence $N_*$ by its multiple
\[(1+q)N_*:=((1+q)N_{k,i})_{k,i}.
\]
This has the effect of composing $\boldsymbol{\nu}_{\infty}^\Phi[N_*]$ by the reparametrization $(\rho,\varphi)\mapsto (\rho,(1+q)\varphi).$ As a consequence,  the images  $\boldsymbol{\nu}_{\infty}^{\Phi}[N_*](\Gamma)$ and $\boldsymbol{\nu}_{\infty}^{\Phi}[(1+q)N_*](\Gamma)$ are equal, see Lemma~\ref{lem:invariance_absolute_normal}. This crucial fact, combined with Theorem~\ref{thm:C1-density}, leads to an asymptotic connection between the images
of $\Gamma$ by the formal pattern  maps $\boldsymbol{\nu}_{\infty}^{\Phi}[N_*]$ and $\boldsymbol{\nu}_{\infty}[(1+n!)N_*]$ for any $n\geq q$. 

\begin{thm}[Asymptotic behavior]
    \label{thm:image_absolute_normal_circle}
Let $\Gamma$ be the circle 
$\{\rho=\rho_0\}$ with $\rho_0\in \Q\,\cap\,]0,1[$. We have
\[
\lim_{n\to +\infty}\boldsymbol{\nu}_{\infty}[(1+n!)N_*](\Gamma)=\boldsymbol{\nu}^\Phi_{\infty}[N_*](\Gamma).
\]
\end{thm}

\noi
In this theorem, it is understood that the limit applies to those elements $\boldsymbol{\nu}_{\infty}[(1+n!)N_*](\Gamma)$ that are well-defined, that is for which the sequences $(1+n!)N_*$ lead to a  well-defined map $f_\infty$.
Whether Theorem~\ref{thm:image_absolute_normal_circle} can be extended to  circles with non-rational radii remains open. 
Theorem~\ref{thm:image_absolute_normal_circle} and its proof suggest that the arithmetic properties of the sequence of corrugation numbers plays a role in the asymptotic geometry of the normal map.\\

The approach followed here is in a way the reverse of the usual one which consists of building a holonomic section starting from a formal solution. In parallel to the holonomic solution, we build a formal analogue whose geometry is controlled and which allows to infer properties on the holonomic solutions. This approach could certainly be applied to any isometric embeddings built by convex integration. 

\section{The general strategy}
Here, we introduce the necessary ingredients of our construction of an isometric embedding of the hyperbolic plane. Our construction relies on a sequence of metrics $(g_k)_k$ converging to the hyperbolic metric, from which we build a sequence of maps $(f_k)_k$ converging to an isometric embedding. The $C^1$ convergence of $(f_k)_k$ can be reduced to the validity of four properties~\ref{it:P1}--\ref{it:P4} involving $(f_k)_k$ and $(g_k)_k$. We prove in Section~\ref{sec:proof-thm-1} that these properties can be fulfilled. Assuming~\ref{it:P1}--\ref{it:P4}, we show in Section~\ref{subsec:Holder} that our construction has maximum Hölder regularity on the limit set.

\subsection{Working on $\overline{\H^2}$}
\label{subsec:coordinate-system}
We use polar coordinates on $D^2$ and therefore introduce the cylinder  
$$ C_0 :=\{(\rho,\varphi) \, |\, \rho\in ]0,1],\;\varphi\in\R/(2\pi\Z)\}$$
and its universal covering $\widetilde{C}_0=\,]0,1]\,\times \R$. We orient $C_0$ and $\widetilde{C}_0$ by requiring $(\partial_{\rho},\partial_{\varphi})$ to be direct and we endow $Int\, C_0$ with the metric
\begin{eqnarray*}
    h := 4 \frac{ d\rho^2 + \rho^2 d\varphi^2 }{ (1-\rho^2)^2 }
\end{eqnarray*}
to obtain a Riemannian surface $(Int\, C_0,h)$  isometric to the punctured Poincaré disk $\H^2\setminus \{O\}.$ From an initial immersion/embedding $f_0$ defined on~$C_0$, we will iteratively apply a corrugation process to produce a sequence of immersions/embedings $(f_q)_q$ that will be $C^0$-converging toward a limit map $f_{\infty}$ defined on $C_0$. Moreover this sequence will be $C^1$-converging on the interior of $C_0$ and the limit map $f_{\infty}$ will be isometric between $(Int\,C_0, h)$ and $\H^2\setminus\{O\}$.\\

\noi
The extension at the origin of the constructed maps $(f_q)_q$ can be realized by modifying iteratively a sequence of lifts $(\widetilde{f}_q)_q$ on disks centered at the origin and of arbitrarily small radius of $\widetilde{C}_0$ (see \cite{Spring-book} p.199, Complement 9.28, for a general result, or \cite{FOCM} for an explicit construction of an extension from an equatorial ribbon to a whole 2-sphere). This part of the construction will be skipped here because it only perturbs $f_{\infty}$ on a compact domain of $Int\, C_0$. We therefore focus on constructing $f_{\infty}$ on the compact annulus
\begin{eqnarray*}
    C := C_0\setminus Int\,D(\rho_0)
\end{eqnarray*}
obtained by removing from $C_0$ an open disk centered at $O$ and of radius $\rho_0$ for some $0<\rho_0<1.$ Accordingly, we denote $\widetilde{C} := \,]\rho_0,1]\,\times \R$.

\subsection{A Nash \& Kuiper-like approach}
\label{subsec:sequence_metrics_maps}
We briefly summarize our approach that relies on the classical Nash and Kuiper construction of isometric maps, see~\cite{n-c1ii-54,k-oc1ii-55,g-pdr-86,Spring-book,ContiEtAl-2012,delellis,ls-fnkti-23}. In our context, it starts with a short initial embedding $f_0:C\to\E^3$ and a sequence of increasing metrics $(g_k)_{k\in\N}$ defined on~$C$ such that $g_0=\pullback{f_0}$ and $\lim_{k\to +\infty}g_k=h$. The hyperbolic metric  $h$ is not defined on $\{(1,\varphi)\}$ so this last limit is only required over $C^*:=C\setminus(\{1\}\times \R/2\pi\Z)$. Let $(\tau_k)_{k\in\N^*}$ be a decreasing sequence of positive numbers such that
\begin{eqnarray*} 
    \Tau:=\sum_{k=1}^{+\infty} \tau_k<+\infty.
\end{eqnarray*}
From the initial embedding $f_0$, we will build a sequence $(f_k)_{k\in\N^*}$  of maps defined on $C$  satisfying the following properties at every point $p\in C$:
\begin{enumerate}[label=($P_{\arabic*}$)]
\item\label{it:P1} $\|g_k -\pullback{f_k} \| \leq  \|g_{k+1}-g_{k}\|$ 
\item\label{it:P2} $\|f_{k}-f_{k-1}\| \leq \tau_k$
\item\label{it:P3} $\|df_{k}-df_{k-1}\|\leq \tau_k +{A}\, \|g_k-\pullback{f_{k-1}}\|^{1/2}$
\item\label{it:P4}  $\sum_k \|g_k-g_{k-1}\|_{K,\infty}^{1/2}<+\infty$ for any compact set $K\subset C^*$.
\end{enumerate}
In \ref{it:P3}, the factor ${A}$ is a constant that does not depend on $k$. Here and thereafter, we use the operator norm for linear maps and for a symmetric bilinear form $b$ we use the norm
\[
    \|b\| = \sup_{v\in\R^2}\frac{|b(v,v)|}{\|v\|^2}.
\]
We also denote by $\|\cdot \|_{C,\infty}$ the supremum norm taken over $C$.

\begin{prop}[Nash-Kuiper]\label{prop:Nash-Kuiper}
If the sequences $(f_k)_{k}$ and $(g_k)_{k}$ satisfy Properties \ref{it:P1} to \ref{it:P4} then $(f_k)_{k}$ converges  toward a  map $f_\infty$ which is continuous over $C$.  Moreover, $f_\infty$ is an $h$-isometric immersion of class $C^1$ over $ C^*$.
 \end{prop}
\begin{proof}
We provide the proof for the sake of completeness. 
Property $(P_2)$ ensures the $C^0$ convergence over $C$ of the sequence $(f_k)_k$ toward a continuous map $f_{\infty}$ such that
$\|f_{\infty}-f_0\|_{C,\infty}\leq \Tau$. 
Properties $(P_1)$ leads to the inequality
\[\begin{array}{lll}
    \|g_k-\pullback{f_{k-1}}\|^{1/2} & \leq &  \|g_k-g_{k-1}\|^{1/2}+\|g_{k-1} -\pullback{f_{k-1}}\|^{1/2}\\
    & \leq  & 2\|g_k-g_{k-1}\|^{1/2}.
\end{array}\]
The convergence of the series in $(P_4)$ thus implies the convergence of the series $\sum_k \|g_k -\pullback{f_{k-1}} \|^{1/2}_{K,\infty}$. Together with the convergence of $\sum_k \tau_k$, this implies by $(P_3)$ that $\sum_k \|df_{k}-df_{k-1}\|_{K,\infty}$ is convergent for every compact $K\subset C^*$. It follows that $f_\infty$ is $C^1$ over $C^*$. Then property $(P_1)$ ensures that $f_{\infty}$ is $h$-isometric over $ C^*$.
\end{proof}

\subsection{Hölder regularity}\label{subsec:Holder}
The limit map $f_{\infty}$ is $C^1$ everywhere except on the boundary $\partial D^2.$ At those points, its regularity can be controlled by the sequence of metrics $(g_k)_k$ together with the sequence $(\tau_k)_k$ introduced in~\ref{subsec:sequence_metrics_maps}.

\begin{prop}\label{prop:Holder-general} 
Let $0<\beta<1.$ Under the assumptions \ref{it:P1}-\ref{it:P4}, if
$$\begin{array}{ll}
    (i) & \mbox{there exists }k_0\in\N  \mbox{ such that, for all } k\geq k_0,\, \tau_k\leq  \|g_k-g_{k-1}\|^{1/2}_{C,\infty} \\
    (ii)  & \sum_k \tau_k^{1-\beta}\|g_{k}-g_{k-1}\|^{\beta/2}_{C,\infty}<+\infty
  \end{array}$$
then the limit map $f_{\infty}$ is $\beta$-Hölder on $C.$
\end{prop}
\begin{proof}
The proof relies on the following classical interpolation inequality:
$$\|F\|_{C^{0,\beta}}\leq 2^{1-\beta} \|dF\|_{\infty}^{\beta}\|F\|_{\infty}^{1-\beta}$$
where $F$ is a $C^1$ map and $\|F\|_{C^{0,\beta}}=\sup_{x\neq y}\frac{|F(y)-F(x)|}{|y-x|^\beta}$ denotes the Hölder norm. From Properties $(P_1), (P_2)$ and $(P_3)$ stated in~\ref{subsec:sequence_metrics_maps} we have
\[
    \left\{\begin{array}{l}
    \|f_k-f_{k-1}\|_{C,\infty} \leq \tau_k\vspace*{1mm}\\
    \|df_k-df_{k-1}\|_{C,\infty}\leq \tau_k+{A}\|g_k-g_{k-1}\|^{1/2}_{C,\infty}.
    \end{array}\right.
\]
If the series $\sum \|g_k-g_{k-1}\|^{1/2}_{C,\infty}$ was convergent, then the previous inequalities would imply that the limit map $f_{\infty}$ exists and is $C^1$ on $C$.
Otherwise, the series $\sum \|g_k-g_{k-1}\|^{1/2}_{C,\infty}$ is divergent and it follows from Assumption $(i)$ of the lemma that for $k\geq k_0$ we have
\[
    \|df_k-df_{k-1}\|_{C,\infty}\leq (1+{A})\|g_k-g_{k-1}\|^{1/2}_{C,\infty}.
\]
The interpolation inequality leads to
\[
    \|f_k-f_{k-1}\|_{C^{0,\beta}}\leq 2^{1-\beta}(1+{A})^{\beta} \tau_k^{1-\beta}\|g_k-g_{k-1}\|^{\beta/2}_{C,\infty}.
\]
Thanks to Assumption $(ii)$ and the fact that $C^{0,\beta}(D^2)$ is a Banach space, we can deduce that the limit map $f_{\infty}$ is $\beta$-Hölder on $C.$
\end{proof}

\section{The corrugation process}\label{sec:corrugation process}

\subsection{The target differential}\label{subsec:target_differential}
Let $f:\widetilde{C} \mapsto \E^3$ be a smooth immersion. We fix an affine projection $\varpi:\widetilde{C}\to\R$ and a family of loops $\gamma:\widetilde{C}\times (\R/\Z)\to \E^3$ that is periodic in the angular coordinate, i.e. that satisfies
\[
    \gamma((\rho, \varphi+2\pi), s) = \gamma((\rho, \varphi), s).
\]
\begin{defn}[Corrugation process~\cite{MT}]\label{def:CP_general}
We denote by $F:\widetilde C\to\E^3$ the map defined by
\begin{eqnarray}
	\forall p\in\widetilde C,\qquad F(p) := f(p) + \frac{1}{N} \int_0^{N\varpi(p)} (\gamma(p,s)- \overline{\gamma}(p) )\d s, \label{eq:CP_general}
\end{eqnarray}
where $\overline{\gamma}(p):=\int_0^1\gamma(p,t)\d t$ denotes the average of the loop $\gamma(p,\cdot)$ and $N\in\N^*$ is any integer. We say that $F$ is obtained from $f$ by a \emph{corrugation process}.
\end{defn}
\noi
Observe that, for all $x\in\R,$
\begin{align}
    \int_x^{x+1} (\gamma(p,s)- \overline{\gamma}(p) )\d s =0. \label{eq:periodicity-int}
\end{align}
The differential of $F$ has the following expression
\begin{eqnarray}\label{eq:dF}
dF(p)=df(p)+\left(\gamma(p, N\varpi(p)) -\overline{\gamma}(p)\right)\otimes d\varpi+ \frac{1}{N} \int_0^{N\varpi(p)} d(\gamma(p,s)-\overline{\gamma}(p)  )\d s. 
\end{eqnarray}
The last term can be made arbitrarily small by increasing the number $N$. We therefore introduce a definition for the remaining terms. They appear to contain important geometric information. See Sections~\ref{sec:formal_process} and~\ref{sec:Gauss_map}.
\begin{defn}\label{def:target_differential}[Target differential]
We denote by $L$ the map given by
\begin{eqnarray*}
    L(p):=df(p)+\left(\gamma(p, N\varpi(p)) -\overline{\gamma}(p)\right)\otimes d\varpi
\end{eqnarray*}
that we call the \emph{target differential}.
\end{defn}
\noi
The defining formula~\eqref{eq:CP_general} for $F$ and the expression~\eqref{eq:dF} of its differential imply for all $p \in \widetilde{C}$:
\begin{itemize}\label{eq:properties_CP}
    \item[(i)] $\|F(p) - f(p)\| = O(1/N)$
    \item[(ii)] $\|dF(p) - L(p)\| = O(1/N)$.
\end{itemize}
Indeed, the integrand is 1-periodic and has vanishing average, see \eqref{eq:periodicity-int}. Thus, Formula~\eqref{eq:CP_general} allows to build a map $F$ that is arbitrarily close to $f$ and whose differential is the target differential $L$, up to $O(1/N)$.

\subsection{The corrugation frame}\label{subsec:corrugation-frame}
For each point $p\in \widetilde{C}$, the pair $(f,\varpi)$ defines a line
\begin{eqnarray*}
    W(p):= df(p)( \ker\,d\varpi)
\end{eqnarray*}
of the tangent space $df(p)(T\widetilde{C}).$ Note that $W$ is tangent to the corrugation wavefront $f(\{\varpi={\rm Cst}\})$. We denote by $W^{\perp}$ the orthogonal complement of $W$ in $df(p)(T\widetilde{C})$ and by $V=[V]^W+[V]^{W^{\perp}}$ the components of any tangent vector $V$ in the orthogonal direct sum $W\oplus W^{\perp}$. Let $(v,w)$ be a direct basis of $\R^2$ such that $d\varpi(p)(v)>0$ and $w\in \ker\,d\varpi(p)$.
We put
\begin{align}
    \mathbf{w}(p) &:= \dfrac{df(p)(w)}{\|df(p)(w)\|} \nonumber \\
    \mathbf{n}(p) &:=\frac{df(p)(v)\wedge df(p)(w)}{\|df(p)(v)\wedge df(p)(w)\|}\quad\mbox{and} \label{eq:def-t-n}\\
    \mathbf{t}(p)&:=\mathbf{w}(p)\wedge \mathbf{n}(p). \nonumber
\end{align}
Observe that $(\mathbf{t}(p),\mathbf{w}(p))$ is a direct orthonormal basis of $df(p)(T_p\widetilde{C})$ because $f$ is assumed to be an immersion, and that $\mathbf{n}(p)$ spans its normal space.
We call the orthonormal frame $\mathbf{F}=(\mathbf{t},\mathbf{w},\mathbf{n})$ the \emph{corrugation frame}. It only depends on the pair $(f,\varpi)$.\\

\noi
We also introduce a vector $u$, depending on $p$, to be such that $df(p)(u)$ is collinear to $\mathbf{t}(p)$ and $d\varpi(u)=1.$ This choice allows a simpler writing of the various formulas which occur in this article. Remark that 
\begin{align}
    [df]^{W^\perp} = df(u)\otimes d\varpi. \label{eq:df_P_perp}
\end{align}

\subsection{The corrugation process on $\widetilde{C}$}
We choose the family of loops $\gamma:\widetilde{C}\times (\R/\Z)\to \R^3$ defined by
\begin{eqnarray}\label{eq:def_gamma}
    \gamma(\cdot,s):= \di  r\left(\cos\theta\,\mathbf{t}+\sin\theta\,\mathbf{n}\right) \quad \mbox{with} \quad \theta := \alpha\cos(2\pi s)
\end{eqnarray}
and where $r$ and $\alpha$ are functions determined below. The image of each loop is an arc of circle of amplitude $2\alpha$ and radius $r$. Its average is
\begin{eqnarray*}
    \overline{\gamma} = r\left(\int_0^1\cos(\alpha\cos(2\pi s))\d s\right)\mathbf{t} = rJ_0(\alpha)\mathbf{t}
\end{eqnarray*}
where $J_0$ denote the Bessel function of the first kind and of order 0. Recall the definition of $u$ introduced in Section~\ref{subsec:corrugation-frame}. We choose $r$ and $\alpha$ such that $\overline{\gamma} = df(u)$, i.e.,
\begin{eqnarray}\label{eq:condition_r_alpha}
    rJ_0(\alpha)\mathbf{t}=df(u).
\end{eqnarray}
In absolute value, the Bessel function $J_0$ is lower than 1. Hence Formula~\eqref{eq:condition_r_alpha} is satisfied if and only if
\begin{eqnarray}\label{eq:condition-on-r-alpha}
    r\geq \left\|df(u) \right\| \quad\text{and}\quad \alpha:=J_0^{-1}\left(\frac{1}{r}\|df(u)\|\right),
\end{eqnarray}
where $J_0^{-1}$ denotes the inverse of the restriction $J_0|_{[0,\kappa_0[}$ of the Bessel function to the positive numbers less than its first zero $\kappa_0=2.404...$

\begin{lemme}\label{lem:expression-target-differential}
If the radius $r$ and the amplitude $\alpha$ satisfy~\eqref{eq:condition-on-r-alpha}  then
the target differential $L$ has the following expression
\begin{eqnarray*}\label{eq:target-differential}
    L = [df]^W + r(\cos\theta\,\mathbf{t} + \sin\theta\,\mathbf{n}) \otimes d\varpi
\end{eqnarray*}
and 
$$\pullback{L} = \pullback{f}+\left( r^2-\left\|df(u)\right\|^2 \right) d\varpi \otimes d\varpi.$$
\end{lemme}

\begin{proof}
From Definition~\ref{def:target_differential} and the value of the average $\overline{\gamma}$, we have
\begin{eqnarray}
\label{eq:L-df}
    L & = & df-df(u)\otimes d\varpi+r(\cos\theta\,\mathbf{t} + \sin\theta\,\mathbf{n})\otimes d\varpi
\end{eqnarray}
It then remains to observe that $[df]^W=df-df(u)\otimes d\varpi$ by~\eqref{eq:df_P_perp} to obtain the first equality. The second equality is easily checked over the basis $(u,w)$.
\end{proof}

\noi
For any smooth map $\eta:\widetilde{C}\to\R_{\geq 0}$ we consider the metric on $\widetilde{C}$
\begin{eqnarray*}
	\mu := \pullback{f}+ \eta d\varpi \otimes d\varpi.
\end{eqnarray*}

\begin{coro}\label{cor:L-mu-isometric} 
Let $\alpha$ be given by~\eqref{eq:condition_r_alpha} with
\begin{eqnarray*}
    r := \sqrt{ \eta + \left\| df(u) \right\|^2}
\end{eqnarray*}
then the map $L$ is $\mu$-isometric, i.e., $\pullback{L} = \mu$. In particular, at every point $p\in\widetilde{C}$, the linear map $L(p)$ is a monomorphism.
\end{coro} 
\noi
Since $\gamma$ depends only on $f$ and $\eta$, we introduce the following notations.
\paragraph{Notations.}
Let $f$ be an immersion, $\varpi$ a projection, $N>0$ an integer and $\eta> 0$ a function. The map obtained by the corrugation process in Definition~\ref{def:CP_general} is denoted by
\begin{eqnarray*}
  F = CP(f,\eta, \varpi, N)
\end{eqnarray*}
where we choose the family of loops \eqref{eq:def_gamma} with $r$ and $\alpha$ as in Corollary~\ref{cor:L-mu-isometric}. Beware that $f$ should be an immersion in order to have a well defined corrugation frame, allowing to define $\gamma$ as in~\eqref{eq:def_gamma}.
We also denote by 
\begin{eqnarray*}
  L(f,\eta,\varpi,N)
\end{eqnarray*}
the target differential $L$ of Definition~\ref{def:target_differential}, where we again choose the family of loops \eqref{eq:def_gamma} with $r$ and $\alpha$ as in Corollary~\ref{cor:L-mu-isometric}. 

\subsection{Properties of the corrugation process}
In this section, we fix an immersion $f$,  a projection $\varpi$  and a function $\eta> 0$.
\begin{lemme}\label{lem:propriete_CP}
For all point $p\in \widetilde{C}$, the map $F = CP(f,\eta, \varpi,N)$ satisfies 
\begin{eqnarray*}
\left\{\begin{array}{l} 
  \|F(p)-f(p)\| = O(\frac{1}{N}),\vspace*{1mm}\\ \|dF(p)-L(p)\| = O(\frac{1}{N}),\vspace*{1mm}\\
  \|\pullback{F}(p)-\mu(p)\| = O(\frac{1}{N}),\vspace*{1mm}\\
  \|dF(p)-df(p)\| \leq  \|dF(p)-L(p)\| + \sqrt{7\eta(p)}\|d\varpi\|
  \end{array}\right.
\end{eqnarray*}
\end{lemme}

\begin{proof}
The first two equalities follow from properties $(i)$ and $(ii)$ of Section~\ref{subsec:target_differential}, the third from $(ii)$ and the fact that $\pullback{L} = \mu.$ For the last inequality we have, by the triangle inequality, 
$$\|dF(p)-df(p)\|\leq\|dF(p)-L(p)\|+\|L(p)-df(p)\|.$$
Equation~\eqref{eq:L-df} shows that the difference $L-df$ reduces to a tensor product of the form $X\otimes d\varpi$  where 
$$\begin{array}{lll}
 \|X\|^2 & = & \left(r\cos\theta -\left\|df(u)\right\|\right)^2+r^2\sin^2\theta.\vspace*{1mm}\\
 \end{array}$$
By using Equation~\eqref{eq:condition_r_alpha}, we obtain
$$\|X\|^2=r^2\left(1 + J_0^2(\alpha) -2 J_0(\alpha)\cos\theta\right).$$
Since the first positive root of $J_0$ is lower than $\pi$ we have $\pi>\alpha>\theta$ and
$$\|X\|^2\leq r^2\left(1 + J_0^2(\alpha) -2 J_0(\alpha)\cos\alpha\right).$$
We then use the following inequality from~\cite[Sublemma 5]{ensaios}:
\begin{eqnarray*}
	1 + J_0^2(\alpha) - 2J_0(\alpha)\cos\alpha \leq 7(1-J_0^2(\alpha))
\end{eqnarray*}
that holds for every $\alpha$ between zero and the first positive root of $J_0$. We finally obtain by Corollary~\ref{cor:L-mu-isometric}  
$$\|X\|^2\leq 7\left(r^2-\left\|df(u) \right\|^2 \right)=7\eta.$$
\end{proof}
\noi
For every linear map $\Psi:\R^2\to\E^3$ we set
\begin{align}
    \lambda(\Psi):=\inf_{v\in\R^2\setminus\{0\}}\frac{\|\Psi(v)\|}{\|v\|}. \label{eq:lambda}
\end{align}
So, $\Psi$ is a monomorphism if and only if $\lambda(\Psi)>0.$

\begin{lemme}\label{lem:F-immersion}
Let $F=CP(f,\eta,\varpi,N)$. For all $p\in\widetilde{C}$, we have
\begin{eqnarray*} 
    \lambda(dF(p))\geq \lambda(df(p))-\|dF(p)-L(p)\|.
\end{eqnarray*}
Hence, if $\lambda(df(p))> \|dF(p)-L(p)\|$ for all $p\in \widetilde{C}$ then the map $F$ is an immersion.
\end{lemme}

\begin{proof}
For every vector $v\in \R^2$, we have by the reverse triangle inequality:
\[
    \|dF(p)(v)\| \geq \|L(p)(v)\| - \|dF(p)(v)-L(p)(v)\|.
\]
Since $L$ is $\mu$-isometric, we also have
\[
    \|L(p)(v)\| = \sqrt{\mu(p)(v,v)} \geq \|df(p)(v)\|.
\]
Putting together the two inequalities, we easily deduce the inequality in the lemma. 
\end{proof}

\begin{lemme}\label{lem:F-embedding}
Let $F=CP(f,\eta,\varpi,N)$. If $f$ is an embedding, $F$ is an immersion and if, on some compact set $K$, the amplitude $\alpha$ is strictly lower than $\pi/2$ then for $N$ large enough the restriction of $F$ on $K$ is an embedding.
\end{lemme}

\begin{proof}
Let $p\in K$ and let $w,u $ be as in Section~\ref{subsec:corrugation-frame}.
From \eqref {eq:L-df} and Lemma~\ref{lem:propriete_CP}, we have
\begin{eqnarray*}
    dF(p)(u) &=& r(p)(\cos\theta(p)\, \mathbf{t}(p) + \sin\theta(p)\, \mathbf{n}(p)) + O(1/N)\\
    dF(p)(w) &=& df(p)(w) + O(1/N)
\end{eqnarray*}
with $\theta(p) = \alpha(p)\cos(2\pi N\varpi(p))$. It follows that the angle between the normal $\mathbf{n}_F(p)$ of $F$ at $p$ and the normal $\mathbf{n}(p) = \mathbf{n}_f(p)$ is less than $\alpha(p) + O(1/N)$. By the hypothesis on $\alpha$, we deduce that there exists a radius $\delta(p)>0$ and a corrugation number $N(p)$ such that for all $q\in D(p,\delta(p))$ and for all $N\geq N(p)$ the angle between $\mathbf{n}_F(q)$ and $\mathbf{n}(p)$ is strictly less than $\pi/2$. Since $K$ is compact, we easily deduce that there exists $\delta_K>0$ and a corrugation number $N_K$ such that for all $p,q\in K$ with $\|p-q\|<\delta_K$ the angle between $\mathbf{n}_F(q)$ and $\mathbf{n}(p)$ is strictly less than $\pi/2$. Thus, over each  $D^2(p,\delta_K)$ the immersion $F$ is a graph over $df(p)(\R^2)$, hence its restriction to $D^2(p,\delta_K)$ is an embedding. The crucial point of this approach is that $\delta_K$ does not depend on $N\geq N_K.$ We now consider the two distances
$$d_f(p_1,p_2):=\|f(p_2)-f(p_1)\|\quad\mbox{and}\quad d_F(p_1,p_2):=\|F(p_2)-F(p_1)\|$$
and the following neighborhood of the diagonal of $K\times K$ :
$$V(\delta_K)=\{(p_1,p_2)\in K\times K\, |\, \exists p\mbox{ such that } p_1,p_2\in D^2(p,\delta_K)\}.$$
Since the complement $K\times K\setminus V(\delta_K)$ is relatively compact and $f$ is an embedding, we have on this complement
$$d_f(p_1,p_2)\geq d_{min}>0.$$
From Lemma~\ref{lem:propriete_CP}, we know that $\|F-f\|_{K,\infty}=O(N^{-1})$ and thus there exists $N'_K\geq N_K$ such that for all $N\geq N'_K$, $\|F-f\|_{K,\infty}<d_{min}/3$. It follows that
$$\forall (p_1,p_2)\in K\times K,\quad d_F(p_1,p_2)\geq d_f(p_1,p_2)-\frac{2}{3}d_{min}.$$ 
This implies that the function $d_F$ never vanishes on $K\times K\setminus V(\delta_K)$.  Since the restriction to $D^2(p,\delta_K)$ of $F$ is an embedding, the distance $d_F$ can not vanish on $V(\delta_k)$, except at the points of the diagonal. This shows that $F$ is an embedding. 
\end{proof}

\paragraph{Descending to the quotient $C$.}\label{subsec:descending-to-quotient}
In general, the affine projection $\varpi:\widetilde{C}\to \R$ does not descend to the quotient~$C.$ However, its differential $d\varpi: \widetilde{C} \to \mathcal{L}(\R^2,\R)$ does, since it is constant. If the immersion $f:\widetilde{C} \to\E^3$ and the map $\eta: \widetilde{C} \to \R_{>0}$ descend to the quotient, the metric $\mu = \pullback{f} + \eta d\varpi\otimes d\varpi$ also descends to the quotient. The lemma below easily follows from Definition~\ref{def:CP_general} and the 1-periodicity observed in Equation~\eqref{eq:periodicity-int}.

\begin{lemme}\label{lem:CP-quotient}
If $f$ and $\eta$ descend to the quotient~$C$ and if $\varpi$ satisfies
$$\forall (\rho,\varphi)\in\widetilde{C},\quad \varpi(\rho,\varphi+2\pi)-\varpi(\rho,\varphi)\in \Z$$
then the map
$F=CP(f,\mu,\varpi,N)$ descends to the quotient~$C.$  
\end{lemme}

\section{Isometric 3-corrugated immersions}

In this section, we show under very general assumptions on the sequence of metrics $(g_k)_k$ that it is possible to simplify the Nash-Kuiper construction. This simplification consists, at each step $k$, in fixing the number $n_k$ of linear forms to 3 and in considering throughout the same linear forms $\ell_1,\ell_2,\ell_3$. This allows us to express the variable $\eta$, involved in the Corrugation Process, as a function of the target metric and the map on which the Corrugation Process is applied. The result is an explicit  iterative process~\eqref{eq:CP_i} defining a sequence of applications converging to a $C^1$ isometry, see Proposition~\ref{prop:3-corrugated-process-general}. We say that the limit map is $3$-corrugated. Note that Proposition~\ref{prop:3-corrugated-process-general} is stated in the particular context of the construction of an isometric embedding of the Poincaré disk into $\E^3$ but it could be applied, \textit{mutatis mutandis}, to explicitly construct codimension~1  isometric immersions in $\E^{n+1}$. In this case, the number of linear forms to consider would be $s_n:=\frac{n(n+1)}{2}$ and the limit map would be $s_n$-corrugated.

\subsection{Primitive basis of the cone of metrics}
Let $\mathcal{S}_2(\R^2)$ be the vector space of symmetric bilinear forms of $\R^2$ and let $(\ell_i\otimes\ell_i)_{i\in\{1,2,3\}}$ be a basis of $\mathcal{S}_2(\R^2)$ where $\ell_1$, $\ell_2$ and $\ell_3$ are three linear forms on $\R^2$. We denote by $(H_i)_{i\in\{1,2,3\}}$ the dual basis of the primitive basis $(\ell_i\otimes\ell_i)_{i\in\{1,2,3\}}$. So, each $H_i:\mathcal{S}_2(\R^2) \to\R$ is a linear form and for any symmetric bilinear form
$$B=\sum_{i=1}^3\eta_i\ell_i\otimes\ell_i$$
we have $\eta_i=H_i(B).$ Let
\begin{align}
  h_{max}:=\max\{\|H_1\|,\|H_2\|,\|H_3\|\} \label{eq:hmax}
\end{align}
be the maximum of the norms\footnote{Recall that we use the operator norms.} of the three linear forms $H_i$. We thus have 
\begin{eqnarray*}
    |H_i(B)|\leq h_{max}\|B\|.
\end{eqnarray*}
Let $D:C\to \mathcal{S}_2(\R^2)$ be of class $C^{\infty}$. We also introduce 
\begin{align}H_{min}(D)(p):=\min_{i\in\{1,2,3\}} H_i(D)(p) \quad \text{ and }\quad
  H_{min}(D):=\inf_{p\in C} H_{min}(D)(p). \label{eq:Hmin}
\end{align}
We thus have
\begin{eqnarray*}
    H_{min}(D)> 0\quad\iff\quad   D\in C^{\infty}(C,\mathcal{C})
\end{eqnarray*}
where $\mathcal{C}$ is the positive cone
\begin{align}
    \mathcal{C} := \{ \eta_1\ell_1\otimes\ell_1+ \eta_2\ell_2\otimes\ell_2+ \eta_3\ell_3\otimes\ell_3\, |\, \eta_1> 0,\eta_2> 0,\eta_3> 0\}.\label{eq:cone-C}
\end{align}

\subsection{Definition of the 3-corrugated process}\label{subsec:3-corrugated}
Let $(\varpi_i)_{i\in\{1,2,3\}}$ be three affine projections satisfying the condition of Lemma~\ref{lem:CP-quotient}. We set $\ell_i=d\varpi_i$ and assume that  $(\ell_i\otimes\ell_i)_{i\in\{1,2,3\}}$ is a basis of $\mathcal{S}_2(\R^2)$. 
In general the coefficient $\eta_{k,i}=H_i(D_{k,i})$ of the decomposition of the difference 
\begin{eqnarray}\label{eq:isometric_default_k_i}
    D_{k,i}:=g_k-\pullback{f_{k,i-1}}
\end{eqnarray}
on this basis has no reason to be positive.
If at each step $(k,i)$ this coefficient $\eta_{k,i}$ is positive then the corrugation process is well-defined and can be used to build 
a sequence $(f_{k,i})_{k,i}$ of corrugated maps. In that case, we write
\begin{eqnarray}
\label{eq:CP_i}
f_{k,i}=CP_i(f_{k,i-1},g_k,N_{k,i}),\qquad i\in\{1,2,3\},
\end{eqnarray} 
for $CP(f_{k,i-1},\eta_{k,i},\varpi_i,N_{k,i})$. Indeed, the $\varpi_i$ being given once for all, the  coefficient $\eta_{k,i}$ can be deduced from $g_k$ and $f_{k,i-1}$. The affine projection used in the corrugation process is indicated by the subscript $i$ in $CP_i$.
We use the convention $f_{k-1}=f_{k,0}=f_{k-1,3}$ and set $f_{k}:=f_{k,3}.$ 

\begin{defn}\label{def:3-corrugated-map}
When the maps $f_k$ are constructed as above by iterating the corrugation process \eqref{eq:CP_i} involving the same three maps $\varpi_i$, $i\in\{1,2,3\}$, we say that the sequence $(f_k)_k$ is obtained by a \emph{3-corrugated process} and that the limit map $f_{\infty}$ (if it exists) is \emph{3-corrugated}.
\end{defn}

\subsection{Properties of 3-corrugated limit maps}
The proof of the existence of a 3-corrugated process is not immediate. Not only must we ensure that the isometric default $g_1-\pullback{f_0}$ of the initial map $f_0$ lies in the positive cone $\mathcal{C}$ generated by the $(\ell_i\otimes\ell_i)_i$ but we also need to make sure that this property remains true for the successive isometric defaults $D_{k,i}$: precisely at each step $(k,i)$ the $\ell_i\otimes\ell_i$-component of $D_{k,i}$ must be positive. To deal with this problem we will consider only sequences of metrics satisfying the following property
\begin{equation*}\label{it:P5}\tag{$P_5$}
\forall k\in\N^*,\forall p\in C, \quad g_{k}(p)-g_{k-1}(p)\in \mathcal{C} \nonumber
\end{equation*} 

\begin{prop}\label{prop:3-corrugated-process-general}
Let $f_0:C\to\E^3$ be an immersion and $(g_k)_k\uparrow h$ be an increasing sequence of metrics defined on $C$ and converging toward~$h$ on $C^*=C\setminus(\{1\}\times\R/2\pi\Z)$.  If 
\begin{itemize}
\setlength{\itemindent}{2em}
    \item[$(i)$] the sequence of metrics $(g_k)_k$ satisfies Properties~\ref{it:P4} and~\eqref{it:P5}
    \item[$(ii)$] the corrugation numbers $(N_{k,i})_{k,i}$ are chosen large enough
\end{itemize}
then the sequence $(f_{k,i})_{k,i}$ iteratively defined by~\eqref{eq:CP_i} $C^0$ converges on $C$ toward a 3-corrugated map $f_{\infty}$. The sequence is also $C^1$ converging on $C^*$ and, on that set, $f_{\infty}$ is a $h$-isometric immersion. 
\end{prop}

\noi
The proof of this proposition is given in the two next sections. The first one shows that the sequence $(f_k)_k$ is well-defined and the second one that  Properties ($P_1$), ($P_2$) and ($P_3$) hold. Proposition~\ref{prop:3-corrugated-process-general} then follows by applying Proposition~\ref{prop:Nash-Kuiper}.

\subsection{Proof of the existence part}\label{subsec:3-corrugated-process}
The map $f_{k,i}$ constructed via formula~\eqref{eq:CP_i} is well-defined provided that the map $f_{k,i-1}$ is an immersion and that the coefficient $\eta_{k,i} = H_i(D_{k,i})$ is positive. For all $i\in\{1,2,3\}$ we put
\[
    \mu_{k,i}:= \pullback{f_{k,i-1}}+\eta_{k,i}\ell_{i}\otimes\ell_{i}\quad\mbox{and}\quad Err_{k,i}:= \mu_{k,i}-\pullback{f_{k,i}}.
\]
We also set $err_{k,i} := \|Err_{k,i}\|_{C,\infty}.$
By Lemma~\ref{lem:propriete_CP}, $err_{k,i}=O(N_{k,i}^{-1})$. The following Lemmas~\ref{lem:Err},~\ref{lem:minoration_eta} and~\ref{lem:D_k4} show how to control the coefficients $\eta_{k,i}$ and the isometric default $D_{k,4}:=g_k-\pullback{f_k}$ of the map $f_k=f_{k,3}$ in terms of the $err_{k,i}$. Then, Lemma~\ref{lem:F_ki_well_defined} gives a sufficient condition on the choice of corrugation numbers to obtain a well defined sequence of maps $(f_{k,i})_{k\in\N^*, i\in\{1,2,3\}}.$

\begin{lemme}\label{lem:Err} 
For every $i\in\{1,2,3\}$, recalling that $D_{k,i} = g_k-\pullback{f_{k,i-1}}$, we have
$$Err_{k,i} = H_i(D_{k,i}) \ell_i\otimes \ell_i - D_{k,i} + D_{k,i+1}.$$
In particular
\begin{eqnarray}\label{eq:H_i-Err_i}
    \forall j\neq i,\qquad H_j(Err_{k,i})= H_j(D_{k,i+1}-D_{k,i}).
\end{eqnarray}
\end{lemme}

\begin{proof}
It is enough to decompose $Err_{k,i}$ as follows:
\begin{eqnarray*}
	Err_{k,i} &=& \mu_{k,i} -\pullback{f_{k,i}}\\
	&=& (\mu_{k,i} -\pullback{f_{k,i-1}}) -(g_k-\pullback{f_{k,i-1}}) + (g_k-\pullback{f_{k,i}})\\
	&=& H_i(D_{k,i})\ell_i\otimes \ell_i -D_{k,i} +D_{k,i+1},
\end{eqnarray*}
where the last line comes from the definition of $\mu_{k,i}$.
\end{proof}

\begin{lemme}\label{lem:minoration_eta}
With $H_{min}$ and $h_{max}$ as in~\eqref{eq:Hmin} and~\eqref{eq:hmax}, we have for all $p\in C$
\begin{eqnarray*}
\begin{array}{lll}
    H_1(D_{k,1}(p)) & \geq & H_{min}(D_{k,1})\vspace*{1mm}\\
    H_2(D_{k,2}(p)) & \geq & H_{min}(D_{k,1})-h_{max}. err_{k,1}\vspace*{1mm}\\
    H_3(D_{k,3}(p))& \geq & H_{min}(D_{k,1})-h_{max}.(err_{k,1}+err_{k,2}).
\end{array}
\end{eqnarray*}
\end{lemme}

\begin{proof}
The first inequality is trivial. For the second one, by definition of $h_{max}$ we have
$$\|H_2(Err_{k,1})\|_{C,\infty} \leq h_{max}.err_{k,1}.$$
By using~\eqref{eq:H_i-Err_i} of Lemma~\ref{lem:Err} with $j=2$ and $i=1$ we deduce for all $p\in C$
\begin{eqnarray*}
\begin{array}{lll} 
  H_2(D_{k,2}(p))  & \geq  & H_2(D_{k,1}(p)) - h_{max}. err_{k,1}\vspace*{1mm}\\
  & \geq & H_{min}(D_{k,1}) - h_{max}. err_{k,1}.
  \end{array}
\end{eqnarray*}
For the $\ell_3\otimes\ell_3$-component of $D_{k,3}$, we similarly have
\begin{eqnarray*}
H_3(Err_{k,1}) = H_3(D_{k,2}-D_{k,1}) \mbox{ and }
H_3(Err_{k,2}) = H_3(D_{k,3}-D_{k,2})
\end{eqnarray*}
and thus
\begin{eqnarray*}
    H_3(D_{k,3}(p)) &=& H_3(D_{k,1}(p))+ H_3(Err_{k,1}(p))+H_3(Err_{k,2}(p))\\
    & \geq & H_{3}(D_{k,1}(p)) - h_{max}.(err_{k,1}+err_{k,2})\\
    & \geq & H_{min}(D_{k,1}) - h_{max}.(err_{k,1}+err_{k,2})
\end{eqnarray*}
which is the third inequality of the lemma.
\end{proof}
\noi
The following lemma is needed to ensure that the difference $D_{k,4}=g_k-\pullback{f_{k,3}}$ is small enough so that the map $f_{k,3}$ is short for the next metric $g_{k+1}.$ 
\begin{lemme}
\label{lem:D_k4}
Let $C_H:=1 + h_{max}\|\ell_2^2\|+ h_{max}\|\ell_3^2\|.$ We have
$$\|D_{k,4}\|_{C,\infty} \leq C_H.(err_{k,1}+err_{k,2}+err_{k,3}).$$
\end{lemme}

\begin{proof}
It follows from~\eqref{eq:H_i-Err_i} of Lemma~\ref{lem:Err} that
$$\left\{\begin{array}{lll}
 \di |H_2(D_{k,1}-D_{k,2})| & \leq & \di h_{max}.err_{k,1}\vspace*{1mm}\\ 
 \di |H_3(D_{k,1}-D_{k,3})| & \leq  & \di h_{max}.(err_{k,1}+err_{k,2}).
 \end{array}\right.$$
By Lemma \ref{lem:Err} we have for every $i\in\{1,2,3\}$
$$D_{k,i+1} = D_{k,i} + Err_{k,i} + H_i(D_{k,i}) \ell_i\otimes \ell_i$$
which implies that
$$D_{k,4} = D_{k,1}+ \sum_{i=1}^3 Err_{k,i} - \sum_{i=1}^3 H_i(D_{k,i})\ell_i\otimes \ell_i$$
hence
\begin{eqnarray*}
	\|D_{k,4}\|_{C,\infty}
	&=& \| \sum_{i=1}^3 Err_{k,i} + \sum_{i=1}^3 H_i(D_{k,1} -D_{k,i})\ell_i\otimes \ell_i \|_{C,\infty}\\
	&\leq & err_{k,1}+err_{k,2}+err_{k,3}\vspace{1mm}\\
	& & +h_{max}err_{k,1}\|\ell_{2}^2\|+h_{max} (err_{k,1}+err_{k,2})\|\ell_3^2\|\\
	&\leq & \left(1 + h_{max}\|\ell_2^2\|+ h_{max}\|\ell_3^2\| \right) (err_{k,1}+err_{k,2}+err_{k,3}).
\end{eqnarray*}
\end{proof}

\noi
Recall from Section~\ref{subsec:sequence_metrics_maps} that we have introduced a decreasing sequence of positive numbers $(\tau_k)_{k\in\N^*}$ with finite sum $\Tau$. These $\tau_k$ are helpful to guide the choice of the corrugation numbers. For a reason that will become clear in the sequel, we further assume that 
\begin{eqnarray}\label{eq:tau_k-bounded-lambda-f_0}
    \Tau\leq\frac{1}{2}\lambda_C(df_0) \quad \text{ where }\quad  \lambda_C(df_0):=\inf_{p\in C}\lambda(df_0(p)).
\end{eqnarray}
Recall that $f_0$ is the initial map and that $\lambda(df_0(p))$ is given by~\eqref{eq:lambda}. Note also that $f$ is an immersion if and only if $\lambda_C(df)>0$.

\begin{lemme}\label{lem:F_ki_well_defined} 
Under Assumption $(a)$ of Proposition~\ref{prop:3-corrugated-process-general}, we can choose at each step $(k,i)$ the corrugation number $N_{k,i}$ so that 
\begin{eqnarray}\label{eq:LC1}
    err_{k,i}\leq \di\min\left(\frac{H_{min}(D_{k,1})}{4h_{max}},\frac{H_{min} (g_{k+1} - g_k)}{6C_Hh_{max}}\right)
\end{eqnarray}
and 
\begin{eqnarray}\label{eq:LC3}
    \|df_{k,i}-L_{k,i}\|_{C,\infty} \leq\frac{\tau_k}{3},
\end{eqnarray}
where $L_{k,i} = L(f_{k,i},\eta_{k,i},\varpi_i,N_{k,i})$ is the target differential of Definition~\ref{def:target_differential}. For such a choice the sequence $(f_{k,i})_{k,i}$ is well-defined. 
\end{lemme}

\begin{proof}
We first assume $k=1.$ By assumption $(a)$ of Proposition~\ref{prop:3-corrugated-process-general}, we have $g_1-\pullback{f_0}\in C^{\infty}(C,\mathcal{C})$ which is equivalent to the fact that $H_{min}(D_{1,1})>0$. Hence $\eta_{1,1}=H_1(D_{1,1})$ is positive and since $f_0$ is an immersion, the  map $f_{1,1}$ is well-defined for any choice of $N_{1,1}$. Nevertheless, to apply the next corrugation process to $f_{1,1}$, we need to ensure that $f_{1,1}$ is an immersion. By Lemma~\ref{lem:propriete_CP}, we can choose $N_{1,1}$ such that~\eqref{eq:LC3} holds. Then by Lemma~\ref{lem:F-immersion}
$$\lambda_C(df_{1,1})\geq \lambda_C(df_{0})-\frac{\tau_1}{3}$$
and by~\eqref{eq:tau_k-bounded-lambda-f_0} we deduce that 
$f_{1,1}$ is an immersion. Moreover, still by Lemma~\ref{lem:propriete_CP}, we can choose $N_{1,1}$ such that~\eqref{eq:LC1} holds. Then by Lemma~\ref{lem:minoration_eta}, $H_2(D_{1,2})>0$. Thus $f_{1,2}$ is well-defined. We choose $N_{1,2}$ so that~\eqref{eq:LC3} holds. Then 
$$\lambda_C(df_{1,2})\geq \lambda_C(df_{1,1})-\frac{\tau_1}{3}\geq \lambda_C(df_{0})-\frac{2\tau_1}{3}$$
and by~\eqref{eq:tau_k-bounded-lambda-f_0},
$f_{1,2}$ is an immersion. If moreover $N_{1,2}$ is chosen so that~\eqref{eq:LC1} holds then $f_{1,3}$ is well-defined and $\lambda_C(df_{1,3})\geq \lambda_C(df_{0})-\tau_1>0$ by~\eqref{eq:tau_k-bounded-lambda-f_0}, thus $f_1:=f_{1,3}$ is an immersion. To continue the iteration, we need to ensure that $g_2-\pullback{f_1}\in C^{\infty}(C,\mathcal{C})$ which is equivalent to  $H_{min}(D_{2,1})>0$. If $N_{1,3}$ is chosen according to~\eqref{eq:LC1} then by Lemma~\ref{lem:D_k4}
\begin{eqnarray*}
    \|D_{1,4}\|_{C,\infty}\leq\frac{H_{min}(g_{2}-g_{1})}{2h_{max}}.
\end{eqnarray*}
Since $|H_i(D_{1,4}(p))|\leq h_{max} \|D_{1,4}(p)\| \leq h_{max} \|D_{1,4}\|_{C,\infty}$ for all $i\in\{1,2,3\}$ and all $p\in C$ we deduce
$$|H_i(D_{1,4})| \leq \frac{1}{2}H_{min}(g_{2}-g_{1}).$$
From the fact that $D_{2,1}=g_2-\pullback{f_{1,3}}=(g_2-g_1)+(g_1-\pullback{f_{1,3}} )$ we now have
\begin{eqnarray*}
	H_i(D_{2,1}) &=& H_i(g_{2}-g_{1}) + H_i(D_{1,4})\geq \frac{1}{2}H_{min}(g_{2}-g_{1}).
\end{eqnarray*}
From Assumption $(a)$ of Proposition~\ref{prop:3-corrugated-process-general}, we have $H_{min}(g_{2}-g_{1})>0$ which shows that $H_{min}(D_{2,1})>0.$ By an easy induction, we conclude that the $f_{k,i}$ are iteratively well defined, each of them being an immersion satisfying
\begin{eqnarray}
\label{eq:lambda-df_ki}
\lambda_C(df_{k,i})\geq \lambda_C(df_0)-\sum_{j=1}^k\tau_k\geq \frac{1}{2}\lambda_C(df_0).
\end{eqnarray}
\end{proof}

\subsection{End of proof of Proposition~\ref{prop:3-corrugated-process-general}}
\label{subsec:C0-C1-convergences}

\begin{lemme}[Property $P_1$]\label{lem:P1} 
Under Assumptions $(a)$ of Proposition~\ref{prop:3-corrugated-process-general}, if the corrugation numbers are chosen to satisfy~\eqref{eq:LC1} and~\eqref{eq:LC3} 
then Property \ref{it:P1} holds:
\[
    \forall k\in\N^*, \forall p\in C: \qquad\|g_k(p)-\pullback{f_k}(p)\|\leq\|g_{k+1}(p)-g_k(p)\|.
\]
\end{lemme}

\begin{proof}
We use  Lemma~\ref{lem:D_k4} and the fact that the corrugation numbers are chosen accordingly to~\eqref{eq:LC1} to write for all $p$
\[
    \|D_{k,4}(p)\|\leq\frac{H_{min}(g_{k+1}-g_{k})(p)}{2h_{max}} \leq\frac{1}{2}\|g_{k+1}(p)-g_{k}(p)\|
  \]
where $H_{min}$ is defined in~\eqref{eq:Hmin}. Since $D_{k,4} = g_k-\pullback{f_k}$, the above inequality proves the lemma.
\end{proof}

\begin{lemme}[Property $P_3$]\label{lem:P3} 
Under Assumption $(a)$ and $(b)$ of Proposition~\ref{prop:3-corrugated-process-general}, if
in addition the corrugations numbers $N_{k,i}$ are chosen to satisfy~\eqref{eq:LC1} and \eqref{eq:LC3} then Property~\ref{it:P3} holds:
\[
    \forall k\in\N^*, \forall p\in C: 
    \|df_{k}(p)-df_{k-1}(p)\| \leq  \tau_k +{A}\, \|g_k(p)-\pullback{f_{k-1}}(p)\|^{1/2}.
\]
\end{lemme}

\begin{proof}
From Lemma~\ref{lem:propriete_CP} and condition~\eqref{eq:LC3} we deduce for all $p\in C$:
$$\|df_{k}(p)-df_{k-1}(p)\| \leq \sum_{i=1}^3 \|df_{k,i}(p)-df_{k,i-1}(p)\| \leq \tau_k+\sqrt{7}\sum_{i=1}^3\|\sqrt{H_i(D_{k,i}(p))}\ell_i\|.$$
By Lemma~\ref{lem:Err} we have, omitting the variable $p$,
$$\left\{\begin{array}{lll}
 H_2(D_{k,2}) & = & H_2(Err_{k,1})+H_2(D_{k,1})\vspace*{1mm}\\
 H_3(D_{k,3}) & = & H_3(Err_{k,2})+H_3(Err_{k,1})+H_3(D_{k,1})\\
  \end{array}\right.$$
thus
$$\left\{\begin{array}{lll}
 |H_2(D_{k,2})| & \leq & h_{max}err_{k,1}+h_{max}\|D_{k,1}\|\vspace*{1mm}\\
 |H_3(D_{k,3})| & \leq & h_{max}(err_{k,2}+err_{k,1})+h_{max}\|D_{k,1}\|.\\
  \end{array}\right.$$
By condition~\eqref{eq:LC1} we obtain
$$|H_2(D_{k,2})|\leq \frac{1}{4}H_{min}(D_{k,1})+h_{max}\|D_{k,1}\| \leq\frac{5}{4}h_{max}\|D_{k,1}\|$$
and similarly $|H_3(D_{k,3})|\leq \frac{3}{2}h_{max}\|D_{k,1}\|.$ Finally,
$$\sum_{i=1}^3\sqrt{7}\|\sqrt{H_i(D_{k,i})}\ell_i\|\leq {A}\|D_{k,1}\|^{1/2}$$
for some constant ${A}$ that only depends on $h_{max}$ and the norms $\|\ell_i\|$ of the linear forms $(\ell_i)_{i\in\{1,2,3\}}$. This concludes the proof of the lemma. 
\end{proof}

\begin{proof}[Proof of Proposition~\ref{prop:3-corrugated-process-general}]
Lemma~\ref{lem:F_ki_well_defined} shows the existence part.  Lemmas~\ref{lem:P1} and~\ref{lem:P3} show that Properties~\ref{it:P1} and~\ref{it:P3} hold provided that the corrugation numbers are chosen to satisfy~\eqref{eq:LC1}, \eqref{eq:LC3}.  By Lemma~\ref{lem:propriete_CP}, we can further choose the corrugation numbers $N_{k,i}$ so that 
\begin{eqnarray}\label{eq:LC2}
    \|f_{k,i}-f_{k,i-1}\|_{C,\infty} \leq\frac{\tau_k}{3}
\end{eqnarray}
and such a choice ensures Property $(P_2)$. It then remains to apply Proposition~\ref{prop:Nash-Kuiper} to conclude.
\end{proof}

\section{Proofs of Theorem~\ref{thm:H2} and Proposition~\ref{prop:3-corrugated}}
\label{sec:proof-thm-1}
In this section, we choose the three affine projections $\varpi_i$, $i\in\{1,2,3\}$, the initial embedding $f_0$ and the sequence of metrics $(g_k)_k$ to apply Proposition~\ref{prop:3-corrugated-process-general} in order to obtain a $3$-corrugated embedding $f_{\infty}$ satisfying the statement of Theorem~\ref{thm:H2}.

\subsection{The wavefront forms}\label{subsec:corrugation_forms}
Let $a\in\frac{1}{2\pi}\Z^*$. We consider the three projections $\varpi_i:\widetilde{C}\to\R$ defined by
\begin{eqnarray}\label{eq:def-varpi}
\varpi_1(\rho,\varphi) := -\rho,\quad
\varpi_2(\rho,\varphi) := \rho-a\varphi \quad \mbox{and} \quad
\varpi_3(\rho,\varphi) := \rho+a\varphi
\end{eqnarray}
We use the circular convention $\varpi_0=\varpi_3$. Observe that the condition $a\in\frac{1}{2\pi}\Z^*$ is necessary for these projections to pass to the quotient $C$, see Lemma~\ref{lem:CP-quotient}.
\begin{figure}[ht]
\centering
	\includegraphics[width = 0.75\textwidth]{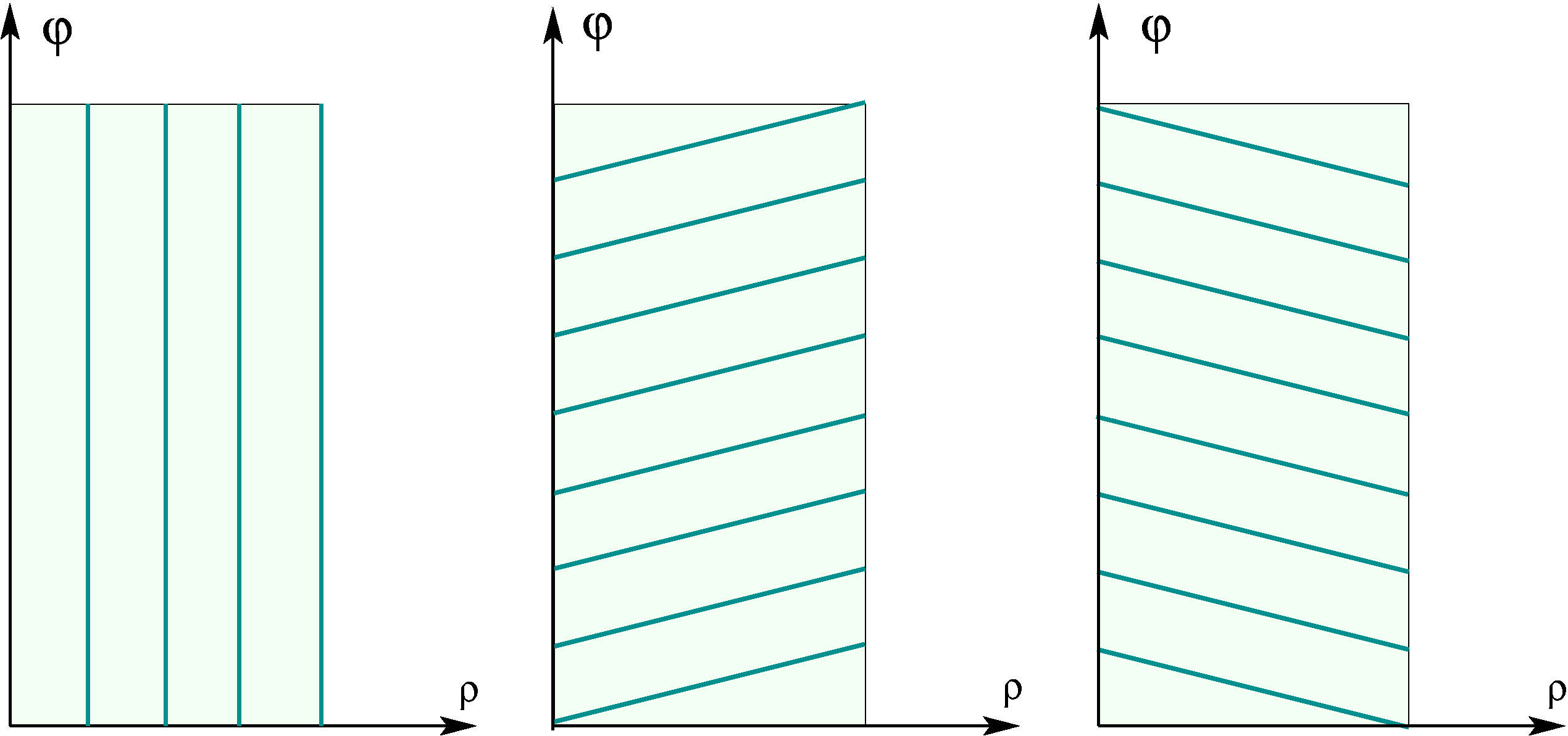}
	\caption{Level lines of the projections $\varpi_i$ in $C_0$ with $i=1$ (left), $i=2$ (middle) and $i=3$ (right).}
\end{figure}
We denote by $\ell_i=d\varpi_i$ the linear forms: 
\begin{eqnarray*}
	\ell_1 := -d\rho, \quad \ell_2 := d\rho - a d\varphi, \quad	\ell_0 = \ell_3 := d\rho + a d\varphi.
\end{eqnarray*}
\begin{figure}[ht]
    \centering
	\includegraphics[width = 0.8\textwidth]{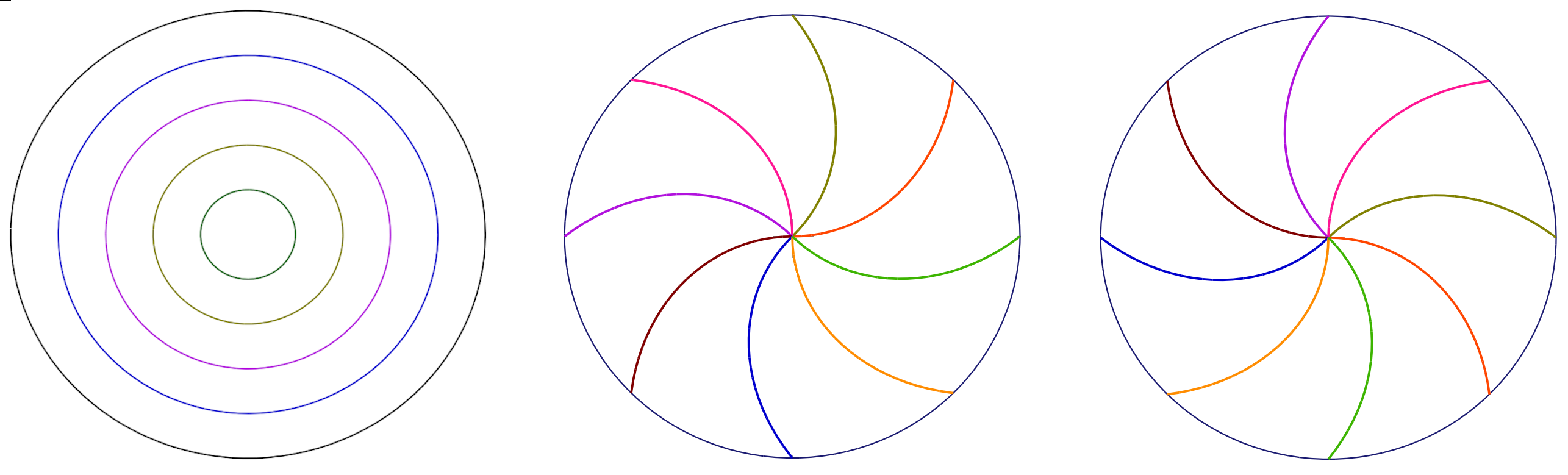}
	\caption{Wavefront curves $\{p\in D^2,\varpi_i(p)=Cte\}$ for $i=1$ (left), $i=2$ (middle) and $i=3$ (right). Their images by $f_{\infty}$ correspond to the wavefronts of the different layers of corrugations, see Figure 1. }
\end{figure}

\subsection{The initial embedding.}\label{subsec:initial-map}
We choose as initial embedding the map $f_0:C\to\E^3$ defined by
\begin{eqnarray}
\label{eq:f_0}
    f_0(\rho,\varphi) = 2\left( \rho\cos\varphi, \rho\sin\varphi, \frac{\sqrt{2}}{2} \rho^2 \right).
\end{eqnarray}
The analytic expression of its isometric default $\Delta$ to the Poincaré metric $h$ is given by
$$\Delta:=h - \pullback{f_0}=4\left(\frac{1}{(1-\rho^2)^2}-1-2\rho^2\right)d\rho^2+4\rho^2\left(\frac{1}{(1-\rho^2)^2}-1\right)d\varphi^2.$$
It is readily checked that $\Delta$ is a metric on every point of $C^*$. This shows that $h>\pullback{f_0}$, i.e. $f_0$ is a strictly short embedding. 
\begin{lemme}
\label{lem:parameter-a}
Let $a=\di\frac{n}{2\pi}$ where $n\geq 7$ is an integer. Then $\Delta\in C^{\infty}(C^*,\mathcal{C})$, where $\mathcal{C}$  is the positive cone defined in~\eqref{eq:cone-C}.
\end{lemme}

\begin{proof}
Let
$$B=Ed\rho^2+F(d\rho\otimes d\varphi + d\varphi \otimes d\rho) + Gd\varphi^2\in\mathcal{S}_2(\R^2)$$ 
be a symmetric bilinear form and let $H_1(B)$, $H_2(B)$ and $H_3(B)$ its coefficients in the basis $(\ell_i\otimes\ell_i)_{i\in\{1,2,3\}}$. A straightforward computation leads to 
\begin{align}
  H_1(B)=E - \frac{G}{a^2},\quad H_2(B)=\frac{G}{2a^2} - \frac{F}{2a}\quad\mbox{ and }\quad H_3(B)=\frac{G}{2a^2} + \frac{F}{2a}. \label{eq:coeff_B}
\end{align}
In particular, the values of the three linear forms $H_i:\mathcal{S}_2(\R^2)\to\R$, $i\in\{1,2,3\},$ at the isometric default $\Delta$ are
$$H_2(\Delta)=H_3(\Delta)=\frac{2\rho^2}{a^2}\left(\frac{1}{(1-\rho^2)^2}-1\right)>0$$
and
$$H_1(\Delta)=\frac{a^2(12\rho^4-8\rho^6)-8\rho^4+4\rho^6}{a^2(1-\rho^2)^2}.$$
The  number $H_1(\Delta)$ is positive if and only if its numerator is positive that is
$$a^2>\frac{2-\rho^2}{3-2\rho^2}.$$
The function $\rho\mapsto \frac{2-\rho^2}{3-2\rho^2}$ is increasing for $\rho\in\, ]0,1]$ and its maximum is 1. We deduce that $H_1(\Delta)>0$ for every $\rho\in\, ]0,1]$ if and only if $a>1.$ 
\end{proof}

\paragraph{Choice of the parameter $a$.} We choose $a=\frac{7}{2\pi}$. 

\subsection{The sequence of metrics}\label{subsec:sequence-of-metrics-gk}
Just like the metric $h$, the isometric default $\Delta$ blows up at $\rho=1$. To build an increasing sequence of metrics $(g_k)_k$ converging toward $h$ while remaining bounded on $C$, we consider the Taylor series of $\Delta$  for the variable $\rho$. We then add truncations of this series to the metric $\pullback{f_0}$. The coefficients of the resulting metrics being polynomial in $\rho$, they extend to the boundary $\rho=1.$ 
In more details, the Taylor series of the $E$ and $G$ coefficients of the isometric default $\Delta$ are
$$E(\Delta)=4\sum_{n=1}^{\infty}(n+2)\rho^{2(n+1)}\quad\mbox{and}\quad G(\Delta) =   4\sum_{n=1}^{\infty}(n+1)\rho^{2(n+1)}.$$ 
For every $k\in\N^*$ we consider the truncations
$$\delta_k^E(\rho):=   4\sum_{n=1}^{k}(n+2)\rho^{2(n+1)}\quad\mbox{and}\quad \delta_k^G(\rho)  :=   4\sum_{n=1}^{k}(n+1)\rho^{2(n+1)}$$
and define a sequence of metrics by setting 
\begin{eqnarray}
\label{eq:def-g_k}
g_k:=\pullback{f_0}+\Delta_k\quad\mbox{where}\quad \Delta_k:=\delta_k^E(\rho)d\rho^2+\delta_k^G(\rho)d\varphi^2.
\end{eqnarray}
Note that $\Delta_0=0$ and $g_0=\pullback{f_0}.$ 
Obviously $g_k\uparrow h$ and each metric $g_k$ is bounded on $C$. 

\begin{lemme}[Property $P_5$]
\label{lem:difference-gk} 
The sequence $(g_k)_k$ of metrics defined by~\eqref{eq:def-g_k} satisfies 
\[\forall k\in\N^*,\qquad g_{k}-g_{k-1}\in C^{\infty}(C,\mathcal{C}).\] 
\end{lemme}

\begin{proof} 
A direct computation using~\eqref{eq:coeff_B} leads to
$$H_2(g_{k}-g_{k-1})=H_3(g_{k}-g_{k-1})=4\rho^{2(k+1)}\frac{k+1}{2a^2}>0$$
and
$$H_1(g_{k}-g_{k-1})=4\rho^{2(k+1)}\left(k+2-\frac{k+1}{a^2}\right)>0$$
because $a>1.$ 
\end{proof}

\noi
For a metric $g$ we denote by $\|g\|_{X,\infty}$ the supremum of $\|g(\rho,\varphi)\|$ on $X\subset C.$ The following lemma will be needed later to deduce that the sequence $(f_k)_k$ is $C^1$-converging over each compact set of $C^*$ that is over $C^*.$

\begin{lemme}[Property $P_4$]\label{lem:asymptotic-behavior-gk}
The sequence $(g_k)_k$ of metrics defined by~\eqref{eq:def-g_k} satisfies
\[\sum_{k=1}^{+\infty} \| g_k - g_{k-1} \|_{K,\infty}^{1/2} < +\infty.
\]
for all compact set $K\subset C^*.$
\end{lemme}

\begin{proof}
Let $b<1$ be the radius of a disk centered at the origin that contains $K$ and let $\eucl:=d\rho\otimes d\rho+d\varphi\otimes d\varphi$ be the Euclidean metric on $C.$ We have 
\begin{eqnarray*} 
\|g_k-g_{k-1}\|_{K,\infty}\leq 4b^{2k+2}(k+2)\|\eucl\|
\end{eqnarray*}
and the resultat follows from the fact that 
$\sum\sqrt{k+2}\,b^{k+1}<+\infty.$
\end{proof}

\subsection{Existence and regularity of the limit map}

\begin{prop}
\label{prop:limit-map}
Let $f_0$ be the embedding defined by~\eqref{eq:f_0}, $\varpi_{1},\varpi_{2},\varpi_3$ be the affine projections defined by~\eqref{eq:def-varpi} and $(g_k)_k$ be the sequence of metrics defined by~\eqref{eq:def-g_k}. If the corrugation numbers $N_{k,i}$ are large enough then the 3-corrugated process
$$f_{k,i}=CP_i(f_{k,i},g_k,N_{k,i}),\qquad i\in\{1,2,3\},$$
is well-defined and its limit map $f_{\infty}$ is continuous on $C$ and a $C^1$ $h$-isometric immersion on $C^*.$
\end{prop}

\begin{proof} Lemmas~\ref{lem:asymptotic-behavior-gk} and~\ref{lem:difference-gk} show that \ref{it:P4} and \eqref{it:P5} hold. It is then enough to apply Proposition~\ref{prop:3-corrugated-process-general} to obtain Proposition~\ref{prop:limit-map}. 
\end{proof}

\begin{lemme}[Hölder regularity]
\label{lem:holder-limit-map}
If, in addition to the assumptions of Proposition~\ref{prop:limit-map}, the sequence $(\tau_k)_k$ is chosen such that, from some $k_0>0$, we have
$$\forall k\geq k_0,\qquad\tau_k\leq e^{-k}$$
then the limit map $f_{\infty}$ is $\beta$-Hölder for any $0<\beta<1.$
\end{lemme}

\begin{proof}
From the definition of the metrics $g_k$ we have 
\[4(k+1)\|\eucl\|\leq \|g_k-g_{k-1}\|_{C,\infty}\leq 4(k+2)\|\eucl\|.
\]
In particular, there exists $k_0\in\N$ such that for all $k\geq k_0$, we have $\tau_k\leq\|g_k-g_{k-1}\|_{C,\infty}^{1/2}$
and Condition~$(i)$ of Proposition~\ref{prop:Holder-general} is fulfilled. Regarding Condition~$(ii)$, it is easily seen that the series
$$\sum (k+2)^{\frac{\beta}{2}}e^{-(1-\beta)k}$$
is convergent for every $0<\beta<1.$ 
\end{proof}

\subsection{Embedded nature of the limit map}
\begin{lemme}
\label{lem:encadrement-pullback-1}
If the corrugations numbers are chosen so that~\eqref{eq:LC1}, \eqref{eq:LC3} and  \eqref{eq:LC2} hold then we have
$$g_{k-1}-\left(|D_{k-1,4}|+\sum_{j=1}^i|Err_{k,j}|\right)\leq\pullback{f_{k,i}}\leq g_k+|D_{k,4}|+\sum_{j=i+1}^3 |Err_{k,j}|$$
for all $i\in\{0,1,2,3\}$.
\end{lemme}

\begin{proof} 
We first prove the left hand side inequality. From the definition of $D_{k-1,4}$, we have
$$
\pullback{f_{k,0}} = \pullback{f_{k-1,3}} = g_{k-1} - D_{k-1,4} \geq g_{k-1}-|D_{k-1,4}|,
$$
which proves the inequality for $i=0$. 
From Lemma~\ref{lem:Err} we have
\begin{eqnarray}
\label{eq:Err-fki-fki-1}
Err_{k,i} = H_i(D_{k,i}) \ell_i\otimes \ell_i +\pullback{f_{k,i-1}} - \pullback{f_{k,i}}.
\end{eqnarray} 
Since $H_i(D_{k,i})\geq 0$ it follows
\begin{eqnarray}
\label{eq:f_ki-f_ki-1}
\pullback{f_{k,i}}\geq \pullback{f_{k,i-1}}-Err_{k,i}.
\end{eqnarray}
If $i=1$, and by definition of $D_{k-1,4}=g_{k-1}-\pullback{f_{k-1,3}}$, we deduce
$$\pullback{f_{k,1}}\geq \pullback{f_{k,0}} - Err_{k,1} \geq g_{k-1}-D_{k-1,4}-Err_{k,1} $$
Using inductively~\eqref{eq:f_ki-f_ki-1}, we then obtain
$$\pullback{f_{k,i}}\geq g_{k-1}-D_{k-1,4}-\sum_{j=1}^iErr_{k,j}.$$
Regarding the other inequality, by definition of $D_{k,4}$ we have 
$$\pullback{f_{k,3}} = g_k - D_{k,4} \leq g_k+|D_{k,4}|.$$
Using~\eqref{eq:f_ki-f_ki-1} and by induction we find for all $i\in\{0,1,2,3\},$
$$\pullback{f_{k,i}}\leq g_k+|D_{k,4}|+\sum_{j=i+1}^3 Err_{k,j}.$$
\end{proof}

\begin{lemme}
\label{lem:encadrement-pullback-2}
Let $\lambda>1.$ If, in addition to~\eqref{eq:LC1}, \eqref{eq:LC3} and  \eqref{eq:LC2} the corrugations numbers $N_{k,i}$ are chosen to satisfy 
\begin{eqnarray}
\label{eq:LC4}
err_{k,i}\leq \frac{1}{6\lambda C_H} \min\left(\min_{p\in C}\|g_{k+1}(p)-g_{k}(p)\|,\min_{p\in C}\|g_{k}(p)-g_{k-1}(p)\|\right)
\end{eqnarray}
then for all $k\in\N^*$ and for all $i\in\{0,1,2,3\}$ we have:
$$g_{k-1}-\frac{1}{\lambda}(g_{k}-g_{k-1})\leq  \pullback{f_{k,i}}\leq g_{k}+\frac{1}{\lambda}(g_{k}-g_{k-1}).$$
\end{lemme}

\begin{proof}
By Lemma~\ref{lem:D_k4} and condition~\eqref{eq:LC4} of the lemma, we have 
$$\|D_{k-1,4}\|\leq C_H\sum_{i=1}^3err_{k-1,i}\leq \frac{1}{2\lambda}\min_{p\in C}\|g_{k}(p)-g_{k-1}(p)\|$$
and, since $C_H\geq 1,$ we also have
$$\sum_{j=1}^3err_{k,i}\leq \frac{1}{2\lambda}\min_{p\in C}\|g_{k}(p)-g_{k-1}(p)\|.$$
We deduce from Lemma~\ref{lem:encadrement-pullback-1} that $g_{k-1}-\frac{1}{\lambda}(g_{k}-g_{k-1})\leq  \pullback{f_{k,i}}$. We also have
$$\|D_{k,4}\|\leq C_H\sum_{i=1}^3err_{k,i}\leq \frac{1}{2\lambda}\min_{p\in C}\|g_{k}(p)-g_{k-1}(p)\|$$
and from Lemma~\ref{lem:encadrement-pullback-1} it follows that $\pullback{f_{k,i}}\leq g_{k}+\frac{1}{\lambda}(g_{k}-g_{k-1}).$
\end{proof}

\begin{lemme}\label{lem:emb-fki}
If the corrugation numbers $(N_{k,i})_{k,i}$ are chosen large enough then each map $f_{k,i}$ is an embedding.
\end{lemme}

\begin{proof} To apply Lemma~\ref{lem:F-embedding}, we need to show that each $\alpha_{k,i}$ is strictly less than $\pi/2.$ We assume that the corrugation numbers are chosen to satisfy~\eqref{eq:LC1}, \eqref{eq:LC3}, \eqref{eq:LC2} and~\eqref{eq:LC4}. 
From ~\eqref{eq:condition_r_alpha} and Corollary~\ref{cor:L-mu-isometric}, we know that 
$$
\alpha_{k,i} = J_0^{-1}\left(
\frac{\|df_{k,i-1}(u_i)\|}{\sqrt{\eta_{k,i}+\|df_{k,i-1}(u_i)\|^2}} 
\right)
=J_0^{-1}(\psi(X_{k,i})),
$$
where $X_{k,i}:=\eta_{k,i}/{\|df_{k,i-1}(u_i)\|^2}$ and $\psi:\R^+\to [0,1]$ is defined by $\psi(x)=(1+x)^{-1/2}$. Since the function $J_0^{-1}\circ \psi$ is increasing, vanishes at $0$ and satisfies $J_0^{-1}(\psi(\xi))=\pi/2$ for $\xi=3.488629...$, it is sufficient to show that $X_{k,i}<\xi.$ 
From~\eqref{eq:Err-fki-fki-1} we have
$$H_i(D_{k,i})\ell_i\otimes\ell_i=Err_{k,i}+\pullback{f_{k,i}}-\pullback{f_{k,i-1}}$$
and thus by using Lemma~\ref{lem:encadrement-pullback-2} 
$$\eta_{k,i}=\eta_{k,i}\ell_i^2(u_i)\leq |Err_{k,i}(u_i,u_i)|+(1+2\lambda^{-1})(g_k-g_{k-1})(u_i,u_i).$$
Still using Lemma~\ref{lem:encadrement-pullback-2}, we then have
$$X_{k,i}= \frac{\eta_{k,i}}{\|df_{k,i-1}(u_i)\|^2}\leq (1+2\lambda^{-1})\frac{(g_{k}-g_{k-1})(u_i,u_i)}{(g_{k-1}-\frac{1}{\lambda}(g_{k}-g_{k-1}))(u_i,u_i)}+\frac{|Err_{k,i}(u_i,u_i)|}{\|df_{k,i-1}(u_i)\|^2}
.$$
In this inequality, the last term can be made arbitrarily small by chosing $N_{k,i}$ large enough, so the problem reduces to show that
$$A_k(\lambda):=(1+2\lambda^{-1})\frac{(g_{k}-g_{k-1})(u_i,u_i)}{(g_{k-1}-\frac{1}{\lambda}(g_{k}-g_{k-1}))(u_i,u_i)}<\xi$$
for every $k\geq 1$. From~\eqref{eq:def-g_k} we have
$$g_k-g_{k-1}\leq 4 (k+2)\rho^{2(k+1)}\eucl\quad\mbox{and}\quad g_{k-1}\geq 4\left(\sum_{n=0}^{k-1}(n+1)\rho^{2(n+1)}\right)\eucl.$$
which implies that 
$$A_k(\lambda)\leq \frac{(1+2\lambda^{-1})(k+2)\rho^{2(k+1)}}{\left(\sum_{n=0}^{k-1}(n+1)\rho^{2(n+1)}\right)-\frac{1}{\lambda}(k+2)\rho^{2(k+1)}}.$$
A direct calculation shows that the right hand side is lower than $\xi$ for every $k\geq 1$ if $\lambda$ is chosen large enough (for instance $\lambda\geq 100$).
\end{proof}

\begin{lemme}
\label{lem:emb-limit-map}
If the corrugation numbers $(N_{k,i})_{k,i}$ are chosen large enough then the limit map $f_{\infty}$ is an embedding.
\end{lemme}

\begin{proof}
The argument of~\cite{k-oc1ii-55} \textsection 10  applies and shows that the limit map $f_{\infty}$ is an embedding provided that the $N_{k,i}$ are large enough. In the reasoning, it is important to keep in mind that the maps $f_{k,i}$ are defined at each step on the compact set $C$. 
\end{proof}

\noi
Theorem~\ref{thm:H2} and Proposition~\ref{prop:3-corrugated} follow directly from Proposition~\ref{prop:limit-map} and Lemmas~\ref{lem:holder-limit-map} and~\ref{lem:emb-limit-map}.

\section{Formal Corrugation Process}\label{sec:formal_process}

\subsection{The sequence $(\Phi_{k,i})_{k,i}$}
\begin{defn}\label{def:FCP}
Let $\Phi:C\to\mbox{Mono}\,(\R^{2},\E^3)$, $\eta:C\to\R_{\geq 0}$, $\ell=d\varpi\in\mathcal{L}(\R^2,\R)$, $W=\Phi(\ker\ell)$ and $N\in\N^*$. Consider the formal corrugation frame
\[
    \mathbf{w}=\frac{\Phi(w)}{\|\Phi(w)\|}, \;
    \mathbf{n}= \di\frac{\Phi(v)\wedge\Phi(w)}{\|\Phi(v)\wedge\Phi(w)\|} \quad{and}\quad
    \mathbf{t}=\mathbf{w}\wedge\mathbf{n}
 \]
where $v$ is any vector such that $\ell(v)>0$ and $w\in\ker\ell$ is such that $(v,w)$ is a direct basis.
We define the \emph{formal corrugation process} of $\Phi$ to be 
\[
    \Phi^c:=[\Phi]^W+\mathbf{z}\otimes\ell\quad\mbox{with}\quad \mathbf{z}:=r(\cos\theta\,\mathbf{t}+\sin\theta\, \mathbf{n}) 
\]
where 
\begin{equation}\label{eq:def-Phi^c}
\left\{ \begin{array}{lll}
  r & = &  \di \sqrt{\eta+\left\|\Phi(u) \right\|^2}\vspace*{1mm}\\
   \theta & = &  \di\alpha\cos (2\pi N\varpi)\vspace*{1mm}\\
   \alpha & = &  \di J_0^{-1}\left( \frac{1}{r}\|\Phi(u)\|
   \right)\vspace*{1mm}
  \end{array}\right.
\end{equation}
and where $u$ is the unique vector such that $\Phi(u)$ is collinear to $\mathbf{t}$ and $\ell(u)=1$. Observing that $[\Phi]^W=\Phi - \Phi(u)\otimes\ell$, we obtain
\begin{align}
    \Phi^c = \Phi + (\mathbf{z} - \Phi(u))\otimes\ell. \label{eq:Phi-c}
\end{align}
Since the data $(\Phi,\eta, \varpi,N)$ completely defines $\Phi^c$, we write
\begin{eqnarray}\label{eq:FCP}
    \Phi^c=FCP(\Phi,\eta,\varpi,N).
\end{eqnarray}
\end{defn}

\noi
{\bf Remark.} If $\Phi=df$ for some immersion $f:C\to\E^3$ then $\Phi^c$ is the target differential $L(f,\eta,\varpi,N)$, see Definition~\ref{def:target_differential}. Moreover, from Corollary~\ref{cor:L-mu-isometric}, $\Phi^c$ is $\mu$-isometric for $\mu:=\pullback{f}+\eta d\varpi\otimes d\varpi$.\\

\noi
Let $\varpi_1$, $\varpi_2$ and $\varpi_3$ be three affine projections and let $f_0:C\to\E^3$ be an immersion. We assume given a sequence of metrics $(g_k)_k$ defined on $C$ satisfying the following hypotheses:
\begin{itemize}
\item $g_0=\pullback{f_0}$,
    \item $(g_k)_k$ is increasing,
    \item $(g_k)_k$ is converging on $C^*$ toward $h$,
    \item the difference $g_k-g_{k-1}\in C^{\infty}(C,\mathcal{C})$.
\end{itemize}
 From such a sequence of metrics and any sequence of positive integers $(N_{k,i})_{k,i}$ we define iteratively a sequence of maps $\Phi_{k,i}:D^2\to Mono(\R^2,\E^3)$ by setting
\begin{eqnarray}
\label{eq:def-Phi_ki}
\Phi_0:=df_0\quad\mbox{and}\quad \Phi_{k,i}:=FCP_i(\Phi_{k,i-1},g_k,N_{k,i})
\end{eqnarray}
where $FCP_i(\Phi_{k,i-1},g_k,N_{k,i}):=FCP(\Phi_{k,i-1},\eta_{k,i},\varpi_i,N_{k,i})$ and
\begin{align}
  \eta_{k,i}:=H_i(g_k-\pullback{\Phi_{k,i-1}}). \label{eq:eta_k_i}
\end{align}

\begin{lemme}\label{lem:sequence-Phi_ki-well-defined}
Given a sequence of metrics  $(g_k)_k$ satisfying the above assumptions and given any sequence of positive integers $(N_{k,i})_{k,i}$, the sequence $(\Phi_{k,i})_{k,i}$ is well defined and each $\Phi_{k,i}$ is $\mu_{k,i}$-isometric for
\begin{eqnarray}\label{eq:def_mu_ki_Phi}
    \mu_{k,i}^\Phi:= g_{k-1} + \sum_{j=1}^i H_j(g_k-g_{k-1}) \ell_j\otimes \ell_j.
\end{eqnarray}
Moreover, for every $(k,i)$ we have $\eta_{k,i}=H_i(g_k-g_{k-1})$.
\end{lemme}

\begin{proof}
The sequence $(\Phi_{k,i})_{k,i}$ is well defined if, at each step $(k,i)$, we have $\eta_{k,i}\geq 0$ and if $\Phi_{k,i-1}$ is a monomorphism. By definition $\Phi_{1,0}=df_0$ is a monomorphism and $\pullback{\Phi_{1,0}} = g_0$. We also have $\eta_{1,1} = H_1(g_1-g_0)>0$ by~\eqref{eq:eta_k_i}. We observe that $[\Phi_{1,1}]^W = \Phi_{1,1} - \Phi_{1,1}(u_{1,1})\ell_1$, where $u_{1,1}$ stands for $u$ as in Definition~\ref{def:FCP}, with respect to $\Phi_{1,0}$ and $\ell_1$. From the definition of the formal corrugation process, we thus have $\Phi_{1,1} = \Phi_{1,0} + (\mathbf{z}_{1,1} - \Phi_{1,0}(u_{1,1}))\otimes \ell_1$.
We deduce that $\Phi_{1,1}(u_{1,1}) = \mathbf{z}_{1,1}$ and $\Phi_{1,1}(w_1) =  \Phi_{1,0}(w_1)$ for any $w_1\in \ker\ell_1$. By testing over $(u_{1,1},w_1)$ we easily check that
\[
    \pullback{\Phi_{1,1}} = \pullback{\Phi_{1,0}} + H_1(g_1-g_0) \ell_1\otimes \ell_1 
\]
proving that $\Phi_{1,1}$ is $\mu_{1,1}^\Phi$-isometric. In particular, $\mu_{1,1}^\Phi$ being a metric, $\Phi_{1,1}$ is a monomorphism. We now compute
\[
    \eta_{1,2} = H_2(g_1 - \pullback{\Phi_{1,1}} ) = H_2(g_1 - \mu_{1,1}^\Phi)
\]
Observe that $H_2(\mu_{1,1}^\Phi) = H_2(g_0)$, whence $\eta_{1,2} = H_2(g_1-g_0)>0$. It follows that $\Phi_{1,2}$ is well defined and we have $\Phi_{1,2} = \Phi_{1,1} + (\mathbf{z}_{1,2} - \Phi_{1,1}(u_{1,2}))\otimes \ell_2$.  Similarly to the previous computation, we check that
\[
    \pullback{\Phi_{1,2}} = \pullback{\Phi_{1,1}} + H_2(g_1-g_0) \ell_2\otimes \ell_2
\]
so that $\pullback{\Phi_{1,2}} = \mu_{1,2}^\Phi$. An easy induction then shows that for every $(k,i)$ the map $\Phi_{k,i}$ is a well defined $\mu_{k,i}^\Phi$-isometric monomorphism with $\eta_{k,i}=H_i(g_k - g_{k-1})$.
\end{proof}

\begin{lemme}\label{lem:convergence_Phi_k_i}
Let $K\subset C^*$ be a compact set. If the series $\sum \|g_k-g_{k-1}\|_{K,\infty}^{\frac{1}{2}}$ converges,  then the series $\sum \|\Phi_{k,i}-\Phi_{k,i-1}\|_{K,\infty}$ is convergent. 
As a consequence, if $\sum \|g_k-g_{k-1}\|_{K,\infty}^{\frac{1}{2}}$ converges  on any compact set of $C^*$, then
  the sequence $(\Phi_{k,i})_{k,i}$ converges on $C^*$ toward a continuous limit map $\Phi_{\infty}$.
\end{lemme}
\begin{defn}
The map $\Phi_{\infty}$ can be interpreted as a target differential for the limit map $f_\infty$ of the 3-corrugated process $CP_i(\cdot,g_k,N_{k,i})$.
We call $\Phi_{\infty}$ the \emph{formal analogue of $df_\infty$}.
\end{defn}
\begin{proof}[Proof of Lemma~\ref{lem:convergence_Phi_k_i}]
Arguments similar to those used in the proof of Lemma~\ref{lem:propriete_CP} show that, at each step $(k,i)$,
$$\|\Phi_{k,i}-\Phi_{k,i-1}\|_{K,\infty}\leq \|\sqrt{7\eta_{k,i}}\ell_i\|_{K,\infty}.$$
From Lemma~\ref{lem:sequence-Phi_ki-well-defined}, we know that
\[\eta_{k,i}=H_i(g_k-g_{k-1})\leq h_{max}\|g_k-g_{k-1}\|
\]
and we obtain
\begin{eqnarray}
\label{eq:convergence-Phi_ki}
\|\Phi_{k,i}-\Phi_{k,i-1}\|_{K,\infty}\leq\sqrt{7h_{max}\|\ell_i\|}\|g_k-g_{k-1}\|^{\frac{1}{2}}_{K,\infty}.
\end{eqnarray}
It is then straightforward to deduce the convergence result of the lemma.
\end{proof}

\subsection{The map $\Phi\mapsto\Phi^c$}
Since the formal corrugation process is defined by a pointwise formula~\eqref{eq:FCP}, it induces a map $\phi\mapsto \phi^c$ from (a subspace of) $Mono(\R^2,\E^3)$ to itself. Precisely, an index $i\in\{1,2,3\}$, an inner product $g$ on $\R^2$ and a corrugation number $N$ being given, then the map $\phi\mapsto \phi^c=FCP_i(\phi,g,N)$
is well defined on  
\[\mathcal{D}(g,i):=\{\phi\in Mono(\R^2,\E^3)\,|\, H_i(g-\pullback{\phi})\geq 0 \}.
  \]
Observe that the subspace $\mathcal{D}(g,i)$ is not compact because $Mono(\R^2,\E^3)$ is open in $\mathcal{L}(\R^2,\E^3)$. 
For any monomorphism $\phi:\R^2\to\E^3$, we set
$$\|\phi\|:=\sup_{v\in\R^2\setminus\{0\}}\frac{\|\phi(v)\|}{\|v\|}\quad\mbox{and}\quad \lambda(\phi):=\inf_{v\in\R^2\setminus\{0\}}\frac{\|\phi(v)\|}{\|v\|}>0.$$
Given $0<\lambda\leq\Lambda$, we consider the compact subspace $\mathcal{K}(\lambda,\Lambda,g,i)$ of $Mono(\R^2,\E^3)$ defined by
$$\mathcal{K}(\lambda,\Lambda,g,i):=\mathcal{D}(g,i)\cap\{\phi\in Mono(\R^2,\E^3)\,|\, \lambda\leq\lambda(\phi)\mbox{ and } \|\phi\|\leq\Lambda\}.$$

\begin{lemme}
\label{lem:FCP-Holder-ponctuel}
Let $0<\lambda\leq\Lambda$, $i$ and $g$ be fixed and let $\mathcal{K}=\mathcal{K}(\lambda,\Lambda,g,i)$. There exists a constant $C=C(\lambda,\Lambda,g)>0$ such that
$$\forall \phi_1,\phi_2\in \mathcal{K},\qquad\|\phi_2^c-\phi_1^c\|\leq C\|\phi_2-\phi_1\|^{\frac{1}{2}}.$$
In other words, the map $\phi\mapsto\phi^c$ is $\frac{1}{2}$-Hölder on $\mathcal{K}$.
\end{lemme}

\begin{proof}
In Formula~\eqref{eq:FCP} defining the formal corrugation process, everything depends smoothly on $\phi$ except the amplitude $\alpha$ that involves the inverse function $J_0^{-1}$ and the radius $r$ that involves a square root. It is readily checked that the term under the square root never vanishes on $\mathcal{K}$. However, the argument in the inverse function $J_0^{-1}$ can be equal to one (when $\eta=0$) and for this value the inverse function $J_0^{-1}$ in not differentiable. This prevents $\phi\mapsto\phi^c$ to be Lipschitz on $\mathcal{K}$.  Nevertheless, Lemma~\ref{lem:J_0-holder} below shows that $J_0^{-1}$ is $\frac{1}{2}$-Hölder. From this, it is straightforward to obtain the result of the lemma. Regarding the constant $C$, since the number of corrugations $N$ only appears in the definition of the angle $\theta=\alpha\cos(2\pi N\varpi)$, it disappears when writing an upper bound of the difference, indeed
$$|\theta(\phi_2)-\theta(\phi_1)|\leq |\alpha(\phi_2)-\alpha(\phi_1)|.$$
Therefore, the constant $C(\lambda,\Lambda,g,i)$ is independent of $N$. By taking the maximum when $i\in\{1,2,3\}$ this constant can also be taken independent of $i$.
\end{proof}
\begin{lemme}
\label{lem:J_0-holder}
The inverse  $J_0^{-1}$ of the restriction $J_0|_{[0,\kappa_0[}$, where $\kappa_0=2.404...$ is the first positive zero of the Bessel function $J_0$, is $\frac{1}{2}$-Hölder.
\end{lemme}
\begin{proof}
  We have the classical series expansion
  $J'_0(x) =-J_1(x)= \sum_{k=0}^\infty (-1)^{k+1} a_k(x)$ with $a_k(x)=\frac{x^{2k+1}}{2^{2k+1}k!(k+1)!}$. We compute for $0\leq x\leq \kappa_0$  and $k\geq 0$:
  \[\frac{a_{k+1}(x)}{a_k(x)} = \frac{x^2}{4(k+1)(k+2)} \leq \frac{\kappa_0^2}{8}< 1
    \]
 It follows that the series of $J'_0$
 satisfies the alternating series test for every $x\in [0,\kappa_0]$. Truncating the series after the second term, we thus get
 \[J'_0(x)\leq P(x) \quad \text{ with } \quad P(x)=-\frac{x}{2} +\frac{x^3}{16}.
 \]
 We easily check that $P(x) + x/8$ is non positive over $[0,\kappa_0]$, whence $J'_0(x) \leq -x/8$. 
 By integrating, we deduce for $0\leq x<y\leq \kappa_0$:
 \[ J_0(x) - J_0(y) \geq \frac{1}{16}(y^2 - x^2) \geq  \frac{1}{16}(y-x)^2
 \]
 We conclude that for all $u,v\in [0,1]$:
  \[ |J_0^{-1}(u)-J_0^{-1}(v)| \leq 4|u-v|^{1/2}.
    \]
\end{proof}

\noi
We denote by $\Gamma Mono(\R^2,\E^3)$ the space of monomorphism fields $\Phi:C\to Mono(\R^2,\E^3)$ over $C$. For any compact set $K\subset C^*$, we also set
$$\|\Phi\|_{K,\infty}:=\sup_{p\in K}\|\Phi(p)\|\quad\mbox{and}\quad \lambda_K(\Phi):=\inf_{p\in K}\lambda(\Phi(p)).$$
Given a metric $g$ on $C$, the map $\Phi\mapsto\Phi^c=FCP_i(\Phi,g,N)$ is well defined on
$$\Gamma\mathcal{D}(g,i,K):=\{\Phi\in \Gamma Mono(\R^2,\E^3)\,|\, \forall p\in K, H_i(g_p-\pullback{\Phi(p)})\geq 0 \}.$$
Similarly as above, $0<\lambda\leq\Lambda$ being given, we consider the compact subspace $\Gamma\mathcal{K}=\Gamma\mathcal{K}(\lambda,\Lambda,g,i,K)$ defined by
$$\Gamma\mathcal{K}:=\Gamma\mathcal{D}(g,i,K)\cap\{\Phi\in \Gamma Mono(\R^2,\E^3)\,|\, \lambda\leq\lambda_K(\Phi)\mbox{ and } \|\Phi\|_{K,\infty}\leq\Lambda\}.$$

\begin{lemme}
\label{lem:FCP-Holder-K}
Let $0<\lambda\leq\Lambda,$ $i$, $g$ and let $K\subset C^*$ be a compact set. There exists a constant $C_K=C_K(\lambda,\Lambda,g)>0$ such that
$$\forall \Phi_1,\Phi_2\in \Gamma\mathcal{K},\qquad\|\Phi_2^c-\Phi_1^c\|_{K,\infty}\leq C_K\|\Phi_2-\Phi_1\|^{\frac{1}{2}}_{K,\infty}.$$
In other words, the map $\Phi\mapsto\Phi^c$ is $\frac{1}{2}$-Hölder on $\Gamma\mathcal{K}$.
\end{lemme}

\begin{proof} The result stated in this lemma is a corollary of Lemma~\ref{lem:FCP-Holder-ponctuel}. The constant $C_K$ is given by $C_K(\lambda,\Lambda,g)=\sup_{p\in K}C(\lambda,\Lambda,g_p).$
\end{proof}

\subsection{Comparing $\Phi_{k,i}$ to $df_{k,i}$}
\label{subsec:comparing_Phi_df}
\noi
In the following, we consider a sequence $(f_{k,i})_{k,i}$ defined by the 3-corrugated process~\eqref{eq:CP_i} and its formal analogue $(\Phi_{k,i})_{k,i}$. Recall from Section~\ref{subsec:sequence_metrics_maps} that we introduced a decreasing sequence of positive numbers $(\tau_k)_k$ guiding the choice of the corrugation numbers. We assume that~\eqref{eq:tau_k-bounded-lambda-f_0} holds and that $\tau_1<1.$ Given any compact $K\subset C^*$, we introduce a sequence $(C_{k}(K))_{k}$ defined by
$$C_k(K):=C_K\left(\frac{\lambda_K(df_0)}{2},\|g_{k+1}\|_{K,\infty}^{\frac{1}{2}},g_{k}\right)$$
where $C_K$  appears in Lemma~\ref{lem:FCP-Holder-K}.

\begin{lemme}
\label{lem:Phi_ki-f_ki-lambda-Lambda} 
Let $K\subset C^*$ a compact set. For all $k\in\N^*$, $i\in\{1,2,3\},$ we have
\begin{eqnarray*}
    \|\Phi_{k,i}-(df_{k,i})^c\|_{K,\infty} \leq C_k(K) \|\Phi_{k,i-1}-df_{k,i-1}\|_{K,\infty}^{\frac{1}{2}}.
\end{eqnarray*}
\end{lemme}

\begin{proof}
By Lemma~\ref{lem:sequence-Phi_ki-well-defined}, the map $\Phi_{k,i}$ is isometric for $\mu_{k,i}^\Phi$. Therefore,
$$\|\Phi_{k,i}(p)\|=\sup_{v\in\R^2\setminus\{0\}}\frac{\|\Phi_{k,i}(p)(v)\|}{\|v\|}=\sup_{v\in\R^2\setminus\{0\}}\frac{\sqrt{\mu_{k,i}^\Phi(p)(v,v)}}{\|v\|}.$$
From Lemma~\ref{lem:sequence-Phi_ki-well-defined} we have  $g_{k-1}\leq\mu_{k,i}^\Phi\leq g_k$
and thus 
$$\|\Phi_{k,i}(p)\|\leq\sup_{v\in\R^2\setminus\{0\}}\frac{\sqrt{g_{k}(p)(v,v)}}{\|v\|}=\|g_k(p)\|^{\frac{1}{2}}.$$
It follows that $\|\Phi_{k,i}\|_{K,\infty}\leq \|g_{k}\|^{\frac{1}{2}}_{K,\infty}$. Similarly, using the fact that $\Phi_{k,i}$ is $\mu_{k,i}^\Phi$-isometric, we obtain
$\lambda_K(\Phi_{k,i})\geq \lambda_K(df_0).$
By construction, the map $df_{k,i}$ is short for $g_{k+1}$ and thus $\|df_{k,i}\|_{K,\infty}\leq \|g_{k+1}\|^{\frac{1}{2}}_{K,\infty}$. From~\eqref{eq:lambda-df_ki}, we also have $\lambda_K(df_{k,i})\geq\frac{1}{2}\lambda_K(df_{0}).$ Finally, we have obtained that $\Phi_{k,i}$ and $df_{k,i}$ both belong to $\Gamma\mathcal{K}(\frac{1}{2}\lambda_K(df_0),\|g_{k+1}\|_{K,\infty}^{\frac{1}{2}},g_k,i,K).$ We then apply Lemma~\ref{lem:FCP-Holder-K} 
 to conclude.
\end{proof}

\noi
Let $M_{k,i}(K)$ be the sequence of constants defined inductively by
\[M_{1,1}(K):=\frac{1}{3}\quad\mbox{and}\quad M_{k,i}(K):= C_k(K)M_{k,i-1}^{1/2}(K) + \frac{1}{3}.
  \]

\begin{lemme}
\label{lem:comparison-monomorphism-immersion-finite}
Assume that $\tau_1<1$. For all $k\in\N^*$ and $i\in\{1,2,3\}$, we have
$$\|\Phi_{k,i}-df_{k,i}\|_{K,\infty}\leq M_{k,i}(K)\tau_1^{\kappa_{k,i}}\quad\mbox{where}\quad \kappa_{k,i}:=2^{-(3k+i-4)}.$$
\end{lemme}
\noi

\begin{proof}
By induction. We have $\Phi_0=df_0$ and $\Phi_{1,1}-df_{1,1} = \di (df_0)^c-df_{1,1}.$ 
Since $(df_0)^c=L_{1,1}$, it follows from~\eqref{eq:LC3} that
\begin{eqnarray*}
\|\Phi_{1,1}-df_{1,1}\|_{K,\infty}\leq \frac{\tau_1}{3}=M_{1,1}(K)\tau_1.
\end{eqnarray*}
Assuming the result of the lemma for $(k,i-1)$, we have at step $(k,i):$
$$\begin{array}{lll}
\Phi_{k,i}-df_{k,i} & = & (\Phi_{k,i}-(df_{k,i-1})^c)+((df_{k,i-1})^c-df_{k,i})\vspace*{1mm}\\
 & = & (\Phi_{k,i-1}^c-(df_{k,i-1})^c)+((df_{k,i-1})^c-df_{k,i})\vspace*{1mm}\\
  \end{array}$$ 
Since $(df_{k,i-1})^c=L_{k,i}$, it follows from~\eqref{eq:LC3} and Lemma~\ref{lem:Phi_ki-f_ki-lambda-Lambda} that 
$$\|\Phi_{k,i}-df_{k,i}\|_{K,\infty} \leq  C_k(K)\|\Phi_{k,i-1}-df_{k,i-1}\|_{K,\infty}^{\frac{1}{2}}+\frac{\tau_{k}}{3}.$$
By the induction hypothesis and since $\tau_k\leq\tau_1\leq\tau_1^{\kappa_{k,i-1}/2}$, we deduce
$$\begin{array}{lll}
\|\Phi_{k,i}-df_{k,i}\|_{K,\infty} 
& \leq & \di C_k(K)M_{k,i-1}(K)^{1/2}\tau_{1}^{\kappa_{k,i-1}/2}+\frac{\tau_{k}}{3}
\vspace*{1mm}\\
& \leq & \di \left(C_k(K)M_{k,i-1}^{1/2}(K)+\frac{1}{3}
\right)\tau_1^{\kappa_{k,i-1}/2}=M_{k,i}(K)\tau_1^{\kappa_{k,i}}.
  \end{array}$$
\end{proof}

\subsection{Proof of Theorem~\ref{thm:C1-density}}

Let $k^*\in\N$. We write the difference $\Phi_{\infty}-df_{\infty}$ as
\[\|\Phi_{\infty}-df_{\infty}\|_{K,\infty}\leq \|\Phi_{\infty}-\Phi_{k^*}\|_{K,\infty}+\|\Phi_{k^*}-df_{k^*}\|_{K,\infty}+\|df_{k^*}-df_{\infty}\|_{K,\infty}\]
where we have used the notation $f_k=f_{k,3}$ and $\Phi_k=\Phi_{k,3}.$
Let $\varepsilon>0$. By Lemma~\ref{lem:convergence_Phi_k_i} and by the proof of Proposition~\ref{prop:Nash-Kuiper}, we can choose $k^*$ so that
\[\|\Phi_{\infty}-\Phi_{k^*}\|_{K,\infty}\leq \frac{\varepsilon}{3} \quad \text{ and }\quad  \|df_{\infty}-df_{k^*}\|_{K,\infty}  \leq \frac{\varepsilon}{3}.
  \]
Choosing $\tau_1\leq (\frac{\varepsilon}{3M_{k,i}(K)})^{1/\kappa_{k,i}}$, we have by Lemma~\ref{lem:comparison-monomorphism-immersion-finite}, that 
$$\|\Phi_{k^*}-df_{k^*}\|_{K,\infty}\leq\frac{\varepsilon}{3}.$$
It follows that $\|\Phi_{\infty}-df_{\infty}\|_{K,\infty}\leq\varepsilon$, which ends the proof of Theorem~\ref{thm:C1-density}.

\section{Gauss map}
\label{sec:Gauss_map}
\subsection{The corrugation matrices}
\label{subsec:corrugation-matrices}
Let $(f_{k,i})_{k,i}$ be a sequence of maps generated by a 3-corrugated process. The data $(f_{k,i-1},\ell_i)$ where $\ell_i=d\varpi_i$ allows to define a field of corrugation frames $\mathbf{F}_{k,i-1}=(\mathbf{t}_{k,i-1},\mathbf{w}_{k,i-1},\mathbf{n}_{k,i-1})$ as in Section~\ref{subsec:corrugation-frame}. More precisely, let $(v_i,w_i)$ be a direct basis of $\R^2$ such that $\ell_i(v_i)>0$ and $w_i\in \ker\ell_i$, we set
\[ \mathbf{w}_{k,i-1} := \frac{df_{k,i-1}(w_i)}{\|df_{k,i-1}(w_i)\|},\qquad \mathbf{n}_{k,i-1} :=\frac{df_{k,i-1}(v_i)\wedge df_{k,i-1}(w_i)}{\|df_{k,i-1}(v_i)\wedge df_{k,i-1}(w_i)\|}, 
  \]
and $\mathbf{t}_{k,i-1}$ is chosen so that $\mathbf{F}_{k,i-1}$ is a direct orthonormal frame.

For each $(k,i)$ there exists a field of orthogonal matrices $\mathcal{M}_{k,i}:C\to SO(3)$ to pass from one frame to the other:
\begin{eqnarray}
\label{eq:corrugation-matrix}
\mathbf{F}_{k,i}=\mathbf{F}_{k,i-1}\cdot\mathcal{M}_{k,i}.
\end{eqnarray}
We call $\mathcal{M}_{k,i}$ a \emph{corrugation matrix}. We introduce an 
 intermediary frame $\mathbf{F}_{k,i-\frac{1}{2}}$ defined by
\begin{eqnarray}
  \label{eq:intermediary_frame}
{\bf w}_{k,i-\frac{1}{2}}:=\frac{df_{k,i}(w_i)}{\|df_{k,i}(w_i)\|}, \quad {\bf n}_{k,i-\frac{1}{2}} := {\bf n}_{k,i} \quad\mbox{and}\quad{\bf t}_{k,i-\frac{1}{2}} := {\bf w}_{k,i-\frac{1}{2}}\wedge \di {\bf n}_{k,i-\frac{1}{2}}.
\end{eqnarray}
Each corrugation matrix thus 
decomposes into a product of two orthogonal matrices
$\mathcal{M}_{k,i}=\mathcal{L}_{k,i}\mathcal{R}_{k,i}$ where $\mathcal{L}_{k,i}$ and $\mathcal{R}_{k,i}$ are defined by
\[\mathbf{F}_{k,i-\frac{1}{2}}=\mathbf{F}_{k,i-1}\cdot \mathcal{L}_{k,i}\quad\mbox{and}\quad\mathbf{F}_{k,i}=\mathbf{F}_{k,i-\frac{1}{2}}\cdot \mathcal{R}_{k,i}.
\]
We have
\[\mathcal{R}_{k,i}=\left(\begin{array}{ccc}
\cos\beta_{k,i} & -\sin\beta_{k,i} & 0\\
\sin\beta_{k,i} & \cos\beta_{k,i} & 0\\
  0 & 0 & 1
\end{array}\right)\]
where $\beta_{k,i}$ is the angle between $df_{k,i}(w_{i})$ and $df_{k,i}(w_{i+1})$. Since $f_{k,i}$ is $C^1$ converging to an $h$-isometric map, this angle converges toward the $h$-angle between  $w_{i}$ and $w_{i+1}$. Regarding $\mathcal{L}_{k,i}$ it was shown in~\cite[Theorem 21]{ensaios} that
\begin{eqnarray}
\label{eq:L_ki}
\mathcal{L}_{k,i}=\left(\begin{array}{ccc}
\cos\theta_{k,i} & 0 & -\sin\theta_{k,i}\\
   0 & 1 & 0\\
\sin\theta_{k,i} & 0 & \cos\theta_{k,i}\\
\end{array}\right)+O\left(\frac{1}{N_{k,i}}\right)
\end{eqnarray}
where $\theta_{k,i}=\alpha_{k,i}\cos(2\pi N_{k,i}\varpi_{i})$. Asymptotically, the corrugation matrix thus looks like a product of two rotations with perpendicular axis, the first one reflecting the effect of the corrugations in the normal direction while the second is changing the direction of the wavefront in preparation for the next corrugation. 
It is readily seen from~\cite[Theorem 23]{ensaios} that the product
\[\mathcal{M}_{\infty}:=\mathcal{M}_{1,1}\mathcal{M}_{1,2}\mathcal{M}_{1,3}\mathcal{M}_{2,1}\cdots=\prod_{k=1}^{\infty}\left(\prod_{i=1}^3 \mathcal{M}_{k,i}\right)\]
is converging toward a continuous map  $\mathcal{M}_{\infty}:C^*\to SO(3)$ (beware of the unusual order of this product). As $k$ tends to infinity, the frame $\mathbf{F}_{k,0}$ converges to a frame $\mathbf{F}_{\infty}=(\mathbf{t}_{\infty},\mathbf{w}_{\infty},\mathbf{n}_{\infty})$ adapted to $f_{\infty}$. Writing $\mathbf{F}_0$ for $\mathbf{F}_{1,0}$, we then have by iterating~\eqref{eq:corrugation-matrix}:
\[\mathbf{F}_{\infty}=\mathbf{F}_0\cdot\mathcal{M}_{\infty}.\]
The normal map $\mathbf{n}_{\infty}$ of $f_{\infty}$ is thus given by
\begin{equation}
\label{eq:F0_M_e3}
\mathbf{n}_{\infty}=\mathbf{F}_0\cdot\mathcal{M}_{\infty}\cdot \mathbf{e}_3
\end{equation}
where $\mathbf{e}_3$ the last vector of the canonical basis of $\E^3$.

\subsection{The formal corrugation matrices}
In analogy with the 3-corrugation process, the sequence $(\Phi_{k,i})_{k,i}$ defines a sequence of frames $(\mathbf{F}^\Phi_{k,i})_{k,i}$ and corrugation matrices $(\mathcal{M}^\Phi_{k,i})_{k,i}$ allowing to express the normal map $\mathbf{n}^\Phi_{\infty}$ of $\Phi_{\infty}$ as an infinite product. Namely, with $(v_i,w_i)$ defined as in Section~\ref{subsec:corrugation-matrices}, we have
\[ \mathbf{w}^\Phi_{k,i-1} := \frac{\Phi_{k,i-1}(w_i)}{\|\Phi_{k,i-1}(w_i)\|},\qquad \mathbf{n}^\Phi_{k,i-1} :=\frac{\Phi_{k,i-1}(v_i)\wedge \Phi_{k,i-1}(w_i)}{\|\Phi_{k,i-1}(v_i)\wedge \Phi_{k,i-1}(w_i)\|}, 
  \]
and $\mathbf{t}^\Phi_{k,i-1}$ is chosen so that $\mathbf{F}^\Phi_{k,i-1}=(\mathbf{t}^\Phi_{k,i-1}, \mathbf{w}^\Phi_{k,i-1}, \mathbf{n}^\Phi_{k,i-1})$ is a direct orthonormal frame. We then write
\begin{equation}
\label{eq:F0_M_e3_Phi}
\mathbf{n}^\Phi_{\infty}=\mathbf{F}_0.\mathcal{M}^\Phi_{\infty}.\mathbf{e}_3 \quad\mbox{with}\quad \mathcal{M}^\Phi_{\infty}:=\prod_{k=1}^{\infty}\left(\prod_{i=1}^3 \mathcal{M}^\Phi_{k,i}\right).
\end{equation}
By considering the intermediary frame $\mathbf{F}^\Phi_{k,i-\frac{1}{2}}$ obtained by replacing $df_{k,i}$ by $\Phi_{k,i}$ in~\eqref{eq:intermediary_frame}, we get as above a splitting of the corrugation matrix in two parts 
\[\mathcal{M}^\Phi_{k,i}=\mathcal{L}^\Phi_{k,i}.\mathcal{R}^\Phi_{k,i},
\]
where
\[\mathbf{F}^\Phi_{k,i-\frac{1}{2}}=\mathbf{F}^\Phi_{k,i-1}\cdot \mathcal{L}^\Phi_{k,i}\quad\mbox{and}\quad\mathbf{F}^\Phi_{k,i}=\mathbf{F}^\Phi_{k,i-\frac{1}{2}}\cdot \mathcal{R}^\Phi_{k,i}.
\]
\begin{lemme}
\label{lem:alpha-beta}
With $\mu_{k,i}^{\Phi}$ defined in Lemma~\ref{lem:sequence-Phi_ki-well-defined} and with the above notations for $w_i$, we have
\[\mathcal{R}^\Phi_{k,i}=\left(\begin{array}{ccc}
\cos\beta^\Phi_{k,i} & -\sin\beta^\Phi_{k,i} & 0\\
\sin\beta^\Phi_{k,i} & \cos\beta^\Phi_{k,i} & 0\\
  0 & 0 & 1
\end{array}\right)\]
with
\[\cos\beta^\Phi_{k,i}=\frac{\mu_{k,i}^{\Phi}(w_i,w_{i+1})}{\sqrt{\mu_{k,i}^{\Phi}(w_i,w_i)\mu_{k,i}^{\Phi}(w_{i+1},w_{i+1})}},\]
and
\[\mathcal{L}^\Phi_{k,i}=\left(\begin{array}{ccc}
\cos \theta^\Phi_{k,i} & 0 & -\sin\theta^\Phi_{k,i}\\
   0 & 1 & 0\\
\sin\theta^\Phi_{k,i} & 0 & \cos\theta^\Phi_{k,i}\\
\end{array}\right)\]
with $\theta^\Phi_{k,i}=\alpha^\Phi_{k,i}\cos(2\pi N_{k,i}\varpi_i)$, where 
\[\alpha^\Phi_{k,i} = J_0^{-1}\left(\sqrt{\frac{Z_{k,i}}{H_i(g_k-g_{k-1})+Z_{k,i}}}\right)
\]
and 
\[Z_{k,i}= \frac{\mu_{k,i-1}^{\Phi}(w_{i-1},w_{i-1})}{\ell_i(w_{i-1})^2}\left( \frac{\mu_{k,i-1}^{\Phi}(w_{i-1},w_{i-1})\mu_{k,i-1}^{\Phi}(w_i,w_{i})}{\mu_{k,i-1}^{\Phi}(w_{i-1},w_i)^2}-1\right)
\]
In particular, $\beta^\Phi_{k,i}$ and $\alpha^\Phi_{k,i}$ do not depend on the corrugation sequence $N_*=(N_{k,i})_{k,i}$.
\end{lemme}

\begin{proof} 
  By definition of the rotation matrix $\mathcal{R}^\Phi_{k,i}$, its angle $\beta^\Phi_{k,i}$ is the angle between $\Phi_{k,i}(w_i)$ and  $\Phi_{k,i}(w_{i+1})$. Since $\Phi_{k,i}$ is $\mu_{k,i}^{\Phi}$-isometric, we compute
 \[\cos\beta^\Phi_{k,i}=\frac{\langle \Phi_{k,i}(w_i),\Phi_{k,i}(w_{i+1})\rangle}{\|\Phi_{k,i}(w_i)\|\|\Phi_{k,i}(w_{i+1})\|}=\frac{\mu_{k,i}^{\Phi}(w_i,w_{i+1})}{\sqrt{\mu_{k,i}^{\Phi}(w_i,w_i)\mu_{k,i}^{\Phi}(w_{i+1},w_{i+1})}}
\]
as claimed in the lemma. For convenience we denote by $u_{k,i}$ the unique vector field such that $\Phi_{k,i-1}(u_{k,i})$ is collinear with $\mathbf{t}^\Phi_{k,i-1}$ and $\ell_i(u_{k,i})=1$, see Definition~\ref{def:FCP}.
From~\eqref{eq:Phi-c},
we get $\Phi_{k,i}(u_{k,i})=\mathbf{z}^\Phi_{k,i}$, where $\mathbf{z}^\Phi_{k,i}= r^\Phi_i(\cos \theta^\Phi_{k,i}\, \mathbf{t}^\Phi_{k,i-1} +\sin \theta^\Phi_{k,i}\, \mathbf{n}^\Phi_{k,i-1})$. 
We also get from~\eqref{eq:Phi-c} that $\Phi_{k,i}(w_i)=\Phi_{k,i-1}(w_i)$, implying $\mathbf{w}^\Phi_{k,i}=\mathbf{w}^\Phi_{k,i-1}$.
Now, $(u_{k,i}, w_i)$ being a direct frame, we compute
\[\mathbf{n}^\Phi_{k,i} =\frac{\Phi_{k,i}(u_{k,i})\wedge \Phi_{k,i}(w_i)}{\|\Phi_{k,i}(u_{k,i})\wedge \Phi_{k,i}(w_i)\|} = \frac{\mathbf{z}^\Phi_{k,i} \wedge \mathbf{w}^\Phi_{k,i-1}}{r^\Phi_i} = \cos \theta^\Phi_{k,i} \mathbf{n}^\Phi_{k,i-1} - \sin \theta^\Phi_{k,i} \mathbf{t}^\Phi_{k,i-1}.
\]
It follows that the rotation matrix
$\mathcal{L}^\Phi_{k,i}$ has the form given in the lemma with 
$\theta^\Phi_{k,i}=\alpha^\Phi_{k,i}\cos(2\pi N_{k,i}\varpi_i)$.
By~\eqref{eq:def-Phi^c}, we have
\[\alpha^\Phi_{k,i} = J_0^{-1}\left(\frac{\|\Phi_{k,i-1}(u_{k,i})\|}{\sqrt{\eta_{k,i}+\|\Phi_{k,i-1}(u_{k,i})\|^2}  }\right)\]
From Lemma~\ref{lem:sequence-Phi_ki-well-defined} and since  $\Phi_{k,i-1}$ is $\mu_{k,i-1}^{\Phi}$-isometric we obtain
\[\alpha^\Phi_{k,i} = J_0^{-1}\left(\sqrt{\frac{\mu_{k,i-1}^{\Phi}(u_{k,i},u_{k,i})}{H_i(g_k-g_{k-1})+\mu_{k,i-1}^{\Phi}(u_{k,i},u_{k,i})}}\right).\]
We have $u_{k,i} = x w_{i-1} + y w_i$ for some real coefficients $x, y$. Applying $\ell_i$ on both sides  of the decomposition we get $x = 1/\ell_i(w_{i-1})$. 
Using the fact that 
$\Phi_{k,i-1}(u_{k,i})$ is perpendicular to $\mathbf{w}_{k,i-1}$ and that $\Phi_{k,i-1}$ is $\mu_{k,i-1}^{\Phi}$-isometric, we deduce
\[y= - \frac{\mu_{k,i-1}^{\Phi}(w_{i-1},w_{i-1})}{\mu_{k,i-1}^{\Phi}(w_{i-1},w_i)\ell_i(w_{i-1})}\]
We then have
\begin{align*}
    \mu_{k,i-1}^{\Phi}(u_{k,i},u_{k,i}) &=  y\mu_{k,i-1}^{\Phi}(u_{k,i},w_i)
                                        &= y\left(\frac{\mu_{k,i-1}^{\Phi}(w_{i-1},w_i)}{\ell_i(w_{i-1})} + y\mu_{k,i-1}^{\Phi}(w_{i},w_i)\right).
  \end{align*}
  Replacing $y$ by its above value, we get the expression for  
  $Z_{k,i}= \mu_{k,i-1}^{\Phi}(u_{k,i},u_{k,i})$ as in the lemma.
\end{proof}

\subsection{Regularity of the formal analogue}

The map $\mathcal{M}^\Phi_{\infty}:C^*\to SO(3)$ has a natural factorization as a product of two maps that we now describe.  
In Formula~\eqref{eq:def-Phi^c} defining the formal corrugation process $\Phi^c=FCP(\Phi,\eta,\varpi,N)$, the affine projection $\varpi$ appears in the expression of the angle $\theta=\alpha\cos(2\pi N\varpi)$ and in the definition of $\ell:=d\varpi$.
We can derive from $FCP$ a new process $\widetilde{FCP}$ by decoupling $\varpi$ and $\ell$. Precisely, the linear form $\ell$ replaces $\varpi$ as a parameter of the process and the projection $\varpi$ is replaced by a variable $t$. In particular, $\theta$ is now considered as a function of two variables $(p,t)\mapsto \theta(p,t)=\alpha(p) \cos (2\pi Nt)$. Similarly, the vector $\mathbf{z}$ in Definition~\ref{def:FCP} is now given by
\[
    \mathbf{z}(p,t) = r(p)\Big(\cos(\theta(p,t))\mathbf{t}(p) + \sin(\theta(p,t))\mathbf{n}(p) \Big) .
\]
Consequently, the maps  $\widetilde{\Phi^c}:C\times\R\to Mono(\R^2,\E^3)$ defined by this new process also have two variables
\[\widetilde{\Phi^c}(p,t):=\Phi(p)+ (\mathbf{z}(p,t) - \Phi(u(p)))\otimes \ell.\]
We denote by $\widetilde{FCP}(\Phi,\eta,\ell,N)$ the formal corrugated map $\widetilde{\Phi^c}$. Of course, if $d\varpi=\ell$ then
\[FCP(\Phi,\eta,\varpi,N)=\widetilde{FCP}(\Phi,\eta,\ell,N)\circ (Id,\varpi).\]
Starting with $\Phi_0$, the extended formal corrugation process produces a sequence of maps $(\widetilde{\Phi}_{k,i})_{k,i}$ such that $\Phi_{k,i}=\widetilde{\Phi}_{k,i}\circ (Id,\varpi_i)$
and a sequence of corrugation matrices such that $\mathcal{M}^\Phi_{k,i}=\mathcal{M}(\widetilde{\Phi}_{k,i})\circ (Id,\varpi_i)$ for all $(k,i).$ We thus have 
\[\mathcal{M}^\Phi_{\infty}=\prod_{k=1}^{\infty}\left(\prod_{i=1}^3\mathcal{M}(\widetilde{\Phi}_{k,i})\circ (Id,\varpi_i)\right).\]
This motivates the introduction of the following corrugation matrix
\[\begin{array}{lcll}
\widetilde{\mathcal{M}}^\Phi_{\infty} : & C^*\times\R^3 & \longrightarrow & SO(3)\\
 & (p,t_1,t_2,t_3) & \longmapsto & \prod_{k=1}^{\infty}\left(\prod_{i=1}^3 \mathcal{M}(\widetilde{\Phi}_{k,i})(p,t_i)\right)
  \end{array}\]
which is defined over $C^*\times\R^3$ since the formal corrugation process converges over $C^*$. We obviously have
\[\mathcal{M}^\Phi_{\infty}=\widetilde{\mathcal{M}}^\Phi_{\infty}\circ (Id,\boldsymbol{\varpi})\]
where $\boldsymbol{\varpi}:C^*\to \R^3$ is the affine map defined by
$\boldsymbol{\varpi}(p):=(\varpi_1(p),\varpi_2(p),\varpi_3(p)).$

\begin{defn}
We call $\widetilde{\mathcal{M}}^\Phi_{\infty}$ the \emph{decoupled} corrugation matrix of~$\Phi_{\infty}.$
\end{defn}

\noi
By ignoring the affine projections $\varpi_i$, the decoupled corrugation matrix
$\widetilde{\mathcal{M}}^\Phi_{\infty}$ makes apparent some possible symmetries of the limit map~$\Phi_{\infty}$. 

\begin{lemme}
\label{lem:H2-G-invariant}
Let $\Phi_{\infty}$ be the formal analogue of $df_{\infty}$. The  decoupled corrugation matrix $\widetilde{\mathcal{M}}^\Phi_{\infty}$ does not depend on the angular parameter $\varphi$.
\end{lemme} 

\begin{proof} The chosen initial map $f_0$ is rotationaly invariant and its pull-back metric $g_0=\pullback{f_0}$ only depends on $\rho$, see \ref{subsec:initial-map}. The sequence of metrics $(g_k)_k$ also only depends on $\rho$, see~\ref{subsec:sequence_metrics_maps}. From the analytical expression~\eqref{eq:def_mu_ki_Phi}, the metrics $\mu_{k,i}^{\Phi}$ also depends only on $\rho.$ Obviously the angle $\widetilde{\beta^\Phi_{k,i}}(p,t_i)$ of the rotation matrix $\widetilde{\mathcal{L}^\Phi_{k,i}}(p,t_i)$ is equal to $\beta^\Phi_{k,i}(p)$ and the amplitude $\widetilde{\alpha^\Phi_{k,i}}(p,t_i)$ is equal to $\alpha^\Phi_{k,i}(p)$. By Lemma~\ref{lem:alpha-beta}, the two functions $\beta^\Phi_{k,i}$ and $\alpha^\Phi_{k,i}.$ can be expressed in terms of the metrics $\mu_{k,i}^{\Phi}$ and consequently, only depends on~$\rho.$ 
\end{proof}
\noi
{\bf Remark.} The matrix $\mathcal{M}^\Phi_{\infty}$ does depend on  $\varphi$ because of the presence of the projections $\varpi_i$ whose values depend on both $\rho$ and $\varphi$.  
\begin{coro}\label{cor:regularity-n}
  The H\"older regularity of $\mathbf{n}^\Phi_{\infty}$ at a point $(\rho,\varphi)$ only depends on~$\rho$. 
\end{coro}
\begin{proof}
  Since $\mathcal{M}^\Phi_{\infty}$ and $\widetilde{\mathcal{M}}^\Phi_{\infty}$ differ by an affine map, they share the same regularity. 
\end{proof}
\noi

\noi
The following proposition enlights the link between $\mathbf{n}^\Phi_{\infty}$ and the Weierstrass-like function defined by
\[(\rho,\varphi)\longmapsto \sum_{k=1}^{\infty}\left(\sum_{i=1}^{3}\alpha^{\Phi}_{k,i}(\rho)\cos(2\pi N_{k,i}\varpi_i(\rho,\varphi))\right).\] 
\begin{prop}
  \label{prop:formal-Weierstrass}
  Let $p_1=(\rho,\varphi_1)$ and $p_2=(\rho,\varphi_2)$ then
\[\|\mathbf{n}^\Phi_{\infty}(p_2)-\mathbf{n}^\Phi_{\infty}(p_1)\|\leq \sqrt{2}\sum_{k=1,i\in\{1,2,3\}}^{\infty}\alpha_{k,i}^{\Phi}(\rho)\left|\Delta\cos(2\pi N_{k,i}\varpi_i)\right|+\|\Delta\mathbf{F}_0\|_F.\]
where $\Delta X$ denotes the difference $X(p_2)-X(p_1).$
\end{prop}

\noi
In this proposition we have used the Frobenius norm:
\[ \|\mathbf{F}\|_F  =  \di \sqrt{\|\mathbf{F}.\mathbf{e}_1\|^2+\| \mathbf{F}.\mathbf{e}_2\|^2+\| \mathbf{F}.\mathbf{e}_3\|^2}.\]
Note that this norm is invariant under the action of the orthogonal group.

\begin{proof}
 We write the difference $\Delta \mathbf{F}^\Phi_{k,i}$ as
\[\begin{array}{lll}
 \di \Delta \mathbf{F}^\Phi_{k,i} & = & \di \mathbf{F}^\Phi_{k,i-1}(p_2).\mathcal{M}^\Phi_{k,i}(p_2)-\mathbf{F}^\Phi_{k,i-1}(p_1).\mathcal{M}^\Phi_{k,i}(p_1)\vspace*{1mm}\\
   & = & \di \mathbf{F}^\Phi_{k,i-1}(p_2).\Delta \mathcal{M}^\Phi_{k,i}+ \Delta\mathbf{F}^\Phi_{k,i-1}.\mathcal{M}^\Phi_{k,i}(p_1).\vspace*{1mm}\\
  \end{array}\]
Since $\mathcal{M}^\Phi_{k,i}$ is an orthogonal matrix we deduce
\[\| \Delta\mathbf{F}^\Phi_{k,i-1}.\mathcal{M}^\Phi_{k,i}(p_1)\|_F= \| \Delta\mathbf{F}^\Phi_{k,i-1}\|_F.\]
From the fact that $\mathcal{M}^\Phi_{k,i}=\mathcal{L}^\Phi_{k,i}\mathcal{R}^\Phi_{k,i}$ we deduce
\[\Delta \mathcal{M}^\Phi_{k,i}=\left(\Delta\mathcal{L}^\Phi_{k,i}\right)\mathcal{R}^\Phi_{k,i}+\mathcal{L}^\Phi_{k,i}\left(\Delta\mathcal{R}^\Phi_{k,i}\right).\]
Since $\beta^\Phi_{k,i}$ does not depend on $\varphi$ we have $\Delta\mathcal{R}^\Phi_{k,i}=0.$ Computing the difference $\Delta\mathcal{L}^\Phi_{k,i}$ and taking the Frobenius norm we obtain 
\[\begin{array}{ccl}
\di\|\Delta\mathcal{L}^\Phi_{k,i}\|_F^2 & = &   \di 8\sin^2\left(\frac{\theta^\Phi_{k,i}(p_2)-\theta^\Phi_{k,i}(p_1)}{2}\right).
  \end{array}\]
Thus
\[\|\Delta \mathcal{L}^\Phi_{k,i}\|_F\leq\sqrt{2}\left|\theta^\Phi_{k,i}(p_2)-\theta^\Phi_{k,i}(p_1)\right|\]
and
\[\|\Delta \mathbf{F}^\Phi_{k,i}\|_F\leq \|\Delta \mathbf{F}^\Phi_{k,i-1}\|_F+\sqrt{2}\left|\theta^\Phi_{k,i}(p_2)-\theta^\Phi_{k,i}(p_1)\right|.\]
The proposition then easily follows.
\end{proof}

\subsection{Periodicity of the formal analogue}

\begin{lemme}\label{lem:periodicity_formal_pattern} 
 Let $\rho \in ]0,1[$, the map
\[\varphi\longmapsto\boldsymbol{\nu}^{\Phi}_{\infty}[N_*]\left(\rho,\varphi\right) \]
is $\frac{2\pi}{7\Omega[N_*]}$-periodic.
\end{lemme}
\begin{proof}
  From~\eqref{eq:F0_M_e3} and ~\eqref{eq:F0_M_e3_Phi}   the pattern map  $\boldsymbol{\nu}^{\Phi}_{\infty}$ of the formal analogue $\Phi_{\infty}$ satisfies
\[ \boldsymbol{\nu}^{\Phi}_{\infty}=\mathcal{M}^\Phi_{\infty}\cdot\mathbf{e}_3.\]
The proof of Lemma~\ref{lem:H2-G-invariant} shows that $\alpha^\Phi_{k,i}$ is not only independent of $\varphi$ but also of the sequence $N_*$ of corrugation numbers. We thus write $\alpha^\Phi_{k,i}(\rho)$ instead of $\alpha^\Phi_{k,i}(N^*)(\rho,\varphi)$. 
From the definition of the wavefront forms~\eqref{eq:def-varpi} we have for every $\varphi_0\in\R$ and $i\in\{1,2,3\}$
\begin{align}
  \varpi_i(\rho,\varphi+\varphi_0)=\varpi_i(\rho,\varphi)+\zeta a\varphi_0 \label{eq:varpi-period}
\end{align}
where $a=\frac{7}{2\pi}$ and $\zeta=0,-1$ or $1$ depending on whether $i$ is 1, 2 or 3. It follows that
\[N_{k,i}\varpi_i\left(\rho,\varphi+\frac{2\pi}{7\Omega[N_*]}\right)=N_{k,i}\varpi_i(\rho,\varphi)+\zeta\frac{N_{k,i}}{\Omega[N_*]}=N_{k,i}\varpi_i(\rho,\varphi)\quad \mod 1\]
since by definition of $\Omega[N_*]$ the quotient $\frac{N_{k,i}}{\Omega[N_*]}$ is an integer if $i=2,3$ and since $\zeta=0$ if $i=1$. From the fact that 
\[\theta^\Phi_{k,i}[N_*]\left(\rho,\varphi\right)=\alpha^\Phi_{k,i}(\rho)\cos\left(2\pi N_{k,i}\varpi_i\left(\rho,\varphi\right)\right)\]
we deduce that for all $(k,i)$
\[\theta^\Phi_{k,i}[N_*]\left(\rho,\varphi+\frac{2\pi}{7\Omega[N_*]}\right)=\theta^\Phi_{k,i}[N_*]\left(\rho,\varphi\right).\]
Since $\beta^\Phi_{k,i}$ is independent of $\varphi$ and of the sequence $N_*$, it easily follows that
\begin{equation}
\label{eq:M_n_phi}
\mathcal{M}_{\infty}^{\Phi}[N_*]\left(\rho,\varphi+\frac{2\pi}{7\Omega[N_*]}\right)=\mathcal{M}_{\infty}^{\Phi}[N_*]\left(\rho,\varphi\right).
\end{equation}
Hence, the $\frac{2\pi}{7\Omega[N_*]}$-periodicity of $\boldsymbol{\nu}^{\Phi}_{\infty}.$
\end{proof}

\subsection{Proof of  Proposition~\ref{prop:fractal-normal} and Corollary~\ref{coro:fractal-normal2}}
\label{sec:self-similarity-prop-and-coro}
Recall that $\Gamma_\ell^j=\{(\rho,\varphi)\,|\, \frac{2j\pi}{7\Omega_\ell}\leq\varphi\leq\frac{2(j+1)\pi}{7\Omega_\ell}\}$ and that $\mathbf{F}^{\Phi}_{\ell}:=\mathbf{F}^{\Phi}_{\ell,3}= (\mathbf{t}^{\Phi}_{\ell},\mathbf{w}^\Phi_{\ell},\mathbf{n}^{\Phi}_{\ell})$ is obtained by considering $\Phi_{\ell,3}$. Define the formal pattern map $\boldsymbol{\nu}^{\Phi}_{\ell,\infty}$ at scale $\ell$ by the relation
\[
\mathbf{n}^\Phi_{\infty}(\rho,\varphi) = \mathbf{F}_{\ell-1}^{\Phi}(\rho,\varphi)\cdot\boldsymbol{\nu}^{\Phi}_{\ell,\infty}(\rho,\varphi).
\]
Hence, for any $(\rho,\varphi)\in \Gamma_\ell^j$ we have
\[
|\mathbf{n}^\Phi_{\infty}(\rho,\varphi)  - \mathbf{F}_{\ell-1}^{\Phi}(\rho,\frac{2j\pi}{7\Omega_\ell})\cdot\boldsymbol{\nu}^{\Phi}_{\ell,\infty}(\rho,\varphi)|
\leq \mathrm{Lip}_{\rho}(\mathbf{F}_{\ell-1}^{\Phi}) \frac{2\pi}{7\Omega_\ell}
\]
and, denoting by  $\mathrm{Rot}_\ell^j$ the rotation mapping $(\mathbf{e}_1,\mathbf{e}_2,\mathbf{e}_3)$ to $\mathbf{F}_{\ell-1}^{\Phi}(\rho,\frac{2j\pi}{7\Omega_{\ell}})$,
\[\d_{Hauss}\left( 
\mathbf{n}^{\Phi}_{\infty}(\Gamma^j_\ell)\ ,\mathrm{Rot}_\ell^j \circ \boldsymbol{\nu}^{\Phi}_{\ell, \infty}(\Gamma_\ell^j)
\right) 
\leq \frac{2\pi \mathrm{Lip}_\rho(\mathbf{F}_{\ell-1}^{\Phi})}{7\Omega_\ell}.\] 
Moreover, the map $\boldsymbol{\nu}^{\Phi}_{\ell,\infty}$ is $\frac{2\pi}{7\Omega_\ell}$ periodic as easily deduced from the proof of Lemma~\ref{lem:invariance_absolute_normal}. In the above inequality, we can thus replace $\boldsymbol{\nu}^{\Phi}_{\ell, \infty}(\Gamma_\ell^j)$ by $\boldsymbol{\nu}^{\Phi}_{\ell, \infty}(\Gamma_\ell^0)$ and obtain the following inequality. 
\begin{equation*}
  \d_{Hauss}\left( 
(\mathrm{Rot}_{\ell}^j)^{-1} \circ \mathbf{n}^{\Phi}_{\infty}(  \Gamma^j_{\ell}) ,\boldsymbol{\nu}^{\Phi}_{\ell, \infty}(\Gamma_{\ell}^0)
\right) 
\leq \frac{2\pi \mathrm{Lip}_\rho(\mathbf{F}_{\ell-1}^{\Phi})}{7\Omega_{\ell}}.  
\end{equation*}

Combining this equation for any $j$ with the one with $j=0$, we have by the triangle inequality
\begin{equation*}
  \d_{Hauss}\left( 
(\mathrm{Rot}_{\ell}^j)^{-1} \circ \mathbf{n}^{\Phi}_{\infty}(  \Gamma^j_{\ell}) ,
(\mathrm{Rot}_{\ell}^0)^{-1} \circ \mathbf{n}^{\Phi}_{\infty}(  \Gamma^0_{\ell}) 
\right) 
\leq \frac{4\pi \mathrm{Lip}_\rho(\mathbf{F}_{\ell-1}^{\Phi})}{7\Omega_{\ell}}, 
\end{equation*}
from which we deduce
\begin{equation}
\label{eq:dist_hausdorff-bis}
  \d_{Hauss}\left( 
\mathbf{n}^{\Phi}_{\infty}(  \Gamma^j_{\ell}) ,
\mathrm{Rot}_{\ell}^j \circ (\mathrm{Rot}_{\ell}^0)^{-1} \circ \mathbf{n}^{\Phi}_{\infty}(  \Gamma^0_{\ell}) 
\right) 
\leq \frac{4\pi \mathrm{Lip}_\rho(\mathbf{F}_{\ell-1}^{\Phi})}{7\Omega_{\ell}}, 
\end{equation}
which implies Inequality~\eqref{eq:dist_hausdorff}.
Decomposing $\Gamma$ as the union $\cup_{j=0}^{7\Omega_\ell-1}\Gamma_\ell^j$ we obtain  
\[\d_{Hauss}\left( 
\mathbf{n}^{\Phi}_{\infty}(\Gamma)\ ,\bigcup_{j=0}^{7\Omega_\ell-1}\mathrm{Rot}_\ell^j \circ (\mathrm{Rot}_{\ell}^0)^{-1} \circ \mathbf{n}^{\Phi}_{\infty}(\Gamma_\ell^0)
\right) 
\leq \frac{4\pi \mathrm{Lip}_\rho(\mathbf{F}_{\ell-1}^{\Phi})}{7\Omega_\ell}\]
which concludes the proof of Proposition~\ref{prop:fractal-normal}. 
Rewriting~\eqref{eq:dist_hausdorff-bis} at scale $\ell+1$, we have for any $j'$:
\begin{equation*}
  \d_{Hauss}\left( 
\mathbf{n}^{\Phi}_{\infty}(  \Gamma^{j'}_{\ell+1}) ,
\mathrm{Rot}_{\ell+1}^{j'} \circ (\mathrm{Rot}_{\ell+1}^0)^{-1} \circ \mathbf{n}^{\Phi}_{\infty}(  \Gamma^0_{\ell+1}) 
\right) 
\leq \frac{4\pi \mathrm{Lip}_\rho(\mathbf{F}_{\ell}^{\Phi})}{7\Omega_{\ell+1}}, 
\end{equation*}
Corollary~\ref{coro:fractal-normal2} is then obtained by applying this inequality to each term of the decomposition of $\Gamma_{\ell}^j$ into 
$$\Gamma_{\ell}^j=\bigcup_{j'=j\frac{\Omega_{\ell+1}}{\Omega_\ell}}^{(j+1)\frac{\Omega_{\ell+1}}{\Omega_\ell}-1}\Gamma_{\ell+1}^{j'}.$$
 
\subsection{Self-similarity at any scale}\label{sec:similarity-scale}
Proposition~\ref{prop:fractal-normal} and Corollary~\ref{coro:fractal-normal2} still hold if we replace $\infty$ by $k > \ell$ which is useful in practice since the numerical experiments are done for a fixed number of iterations, namely for $k<\infty$. More precisely, if we define 
\[
\boldsymbol{\nu}_{\ell,k}^{\Phi} := \prod_{q=\ell}^k \left(\prod_{i=1}^3  
\mathcal{M}_{q,i}^\Phi
\right)\cdot \mathbf{e}_3
\]
we have
\[
\mathbf{n}^\Phi_{k}(\rho,\varphi) = \mathbf{F}_{\ell-1}^{\Phi}(\rho,\varphi)\cdot\boldsymbol{\nu}^{\Phi}_{\ell,k}(\rho,\varphi).
\]
Still following the proof of Lemma~\ref{lem:periodicity_formal_pattern}, we deduce that the map $\boldsymbol{\nu}^{\Phi}_{\ell,k}$ is $\frac{2\pi}{7\Omega_\ell}$-periodic. This allows to conclude that 
\[\d_{Hauss}\left( 
\mathbf{n}^{\Phi}_{k}(\Gamma^j_\ell)\ ,\mathrm{Rot}_{\ell}^j \circ \boldsymbol{\nu}^{\Phi}_{\ell,k}(\Gamma_\ell^0)
\right) 
\leq \frac{2\pi \mathrm{Lip}_\rho(\mathbf{F}_{\ell-1}^{\Phi})}{7\Omega_\ell}.\] 
Hence all the bounds of Proposition~\ref{prop:fractal-normal}
and Corollary~\ref{coro:fractal-normal2} hold when $\mathbf{n}^{\Phi}_{\infty}$ and $\boldsymbol{\nu}^{\Phi}_{\ell,\infty}$ are replaced by $\mathbf{n}^{\Phi}_{k}$ and  $\boldsymbol{\nu}^{\Phi}_{\ell,k}$.
These results also hold at an even smaller scale by considering the following pattern maps
\[\boldsymbol{\nu}_{\ell,\sigma,k,i}^{\Phi} := \left(\prod_{(q,s)=(\ell,\sigma)}^{(k,i)}\mathcal{M}_{q,s}^\Phi\right)\cdot \mathbf{e}_3.\]
In the above formula, $(\ell,\sigma)$, $(k,i)$ and $(q,s)$ are elements of $\N^*\times\{1,2,3\}$ ordered lexicographically:
\[\cdots (q-1,3)<(q,1)<(q,2)<(q,3)<(q+1,1)<\cdots.\]
Observe that $\boldsymbol{\nu}_{\ell,k}^{\Phi}=\boldsymbol{\nu}_{\ell,3,k,3}^{\Phi}$. We have
\[
\mathbf{n}^\Phi_{k,i}(\rho,\varphi) = \mathbf{F}_{\ell,\sigma-1}^{\Phi}(\rho,\varphi)\cdot\boldsymbol{\nu}^{\Phi}_{\ell,\sigma,k,i}(\rho,\varphi)
\]
with the convention $(q-1,3)=(q,0)$ and $(q-1,4)=(q,1).$
Let
$\Omega_{\ell,\sigma}=\Omega_{\ell,\sigma}[N_*]$ be the greatest common divisor of the sequence $(N_{q,s})_{q,s}$ where $s=2$ or $3$ and $(q,s)\geq(\ell,\sigma).$
Note that $\Omega_{\ell,3}=\Omega_\ell$ and  $\Omega_{\ell,1}=\Omega_{\ell,2}$.
Proposition~\ref{prop:similarity-any-scale} below shows that the image by $\mathbf{n}_{k,i}^{\Phi}$ of the arc
\[\Gamma_{\ell,\sigma}^j:=\{(\rho,\varphi)\,|\, \frac{2j\pi}{7\Omega_{\ell,\sigma}}\leq\varphi\leq\frac{2(j+1)\pi}{7\Omega_{\ell,\sigma}}\},\]
with $j\in\{0,\cdots,7\Omega_{\ell,\sigma}-1\}$, approximately decomposes into $\Omega_{\ell,\sigma+1}/\Omega_{\ell,\sigma}$ rotated copies of $\mathbf{n}^{\Phi}_{k,i}(\Gamma_{\ell,\sigma+1}^0)$.
 \begin{prop}\label{prop:similarity-any-scale} Let $\rho \in ]0,1[$. Let $(\ell,\sigma)$ and $(k,i)$  be two elements of $\N^*\times\{1,2,3\}$ such  that $(\ell,\sigma)<(k,i)$. For every $j\in \{0, \cdots,7\Omega_{\ell,\sigma}-1\}$, we have 
 \[
\d_{Hauss}\left( 
\mathbf{n}^{\Phi}_{k,i}(\Gamma^j_{\ell,\sigma})\ ,\bigcup_{j'=j\frac{\Omega_{\ell,\sigma+1}}{\Omega_{\ell,\sigma}}}^{(j+1)\frac{\Omega_{\ell,\sigma+1}}{\Omega_{\ell,\sigma}}-1}
\mathrm{rot}_{\ell,\sigma+1}^{j'}
\circ\mathbf{n}^{\Phi}_{k,i}(\Gamma_{\ell,\sigma+1}^0)
\right) 
\leq \frac{2\pi \mathrm{Lip}_\rho(\mathbf{F}_{\ell,\sigma}^{\Phi})}{7\Omega_{\ell,\sigma+1}}.
\]
where $\mathrm{rot}_{\ell,\sigma+1}^{j}$  denotes the rotation mapping $\mathbf{F}_{\ell,\sigma}^{\Phi}(\rho,0)$ to $\mathbf{F}_{\ell,\sigma}^{\Phi}(\rho,\frac{2j\pi}{7\Omega_{\ell,\sigma+1}})$.
\end{prop}
\noi
The proof of this proposition is an easy adaptation of the proofs of Proposition~\ref{prop:fractal-normal} and Corollary~\ref{coro:fractal-normal2} and is left to the reader. 

\subsection{Proof of Theorem~\ref{thm:image_absolute_normal_circle}}
\label{sec:normal-pattern}

 \begin{lemme}
\label{lem:invariance_absolute_normal} 
Let $\rho=\frac{p}{q}\in\Q\,\cap\, ]0,1[$, we have for all $n\geq q$ and all $\varphi\in\R/(2\pi\Z)$
\begin{equation}
\label{eq:invariance_absolute_normal} \boldsymbol{\nu}^{\Phi}_{\infty}[(1+n!)N_*]\left(\rho,\varphi\right)=\boldsymbol{\nu}^{\Phi}_{\infty}[N_*]\left(\rho,(1+n!)\varphi\right) 
\end{equation}
where $(1+n!)N_*$ is the sequence $((1+n!)N_{k,i})_{k,i}.$ 
\end{lemme}

\begin{proof}
Since $\rho=\frac{p}{q}\in\Q$, for all $n\geq q$ we have $n!\rho\in\N$ and consequently $n!N_{k,i}\rho\in\N$ for any $(k,i)$. It follows that
\[(1+n!)N_{k,i}\rho=N_{k,i}\rho\mod 1\]
and 
\[\begin{array}{lll}
\di \theta^\Phi_{k,i}[(1+n!)N_*]\left(\rho,\varphi\right) & = & \alpha^\Phi_{k,i}(\rho)\cos\left(2\pi (1+n!)N_{k,i}\varpi_i(\rho,\varphi)\right)\vspace*{1mm}\\
& = & \di  \alpha^\Phi_{k,i}(\rho)\cos\left(2\pi (1+n!)N_{k,i}(\rho+\zeta a\varphi)\right)\vspace*{1mm}\\
& = & \di  \alpha^\Phi_{k,i}(\rho)\cos\left(2\pi N_{k,i}\rho+2\pi N_{k,i}\zeta a(1+n!)\varphi\right)\vspace*{1mm}\\
& = & \di \alpha^\Phi_{k,i}(\rho)\cos\left(2\pi N_{k,i}\varpi_i(\rho,(1+n!)\varphi)\right)\vspace*{1mm}\\
& = & \di \theta^\Phi_{k,i}[N_*]\left(\rho,(1+n!)\varphi\right).
  \end{array}.\]
With  arguments analogous to the proof of Lemma~\ref{lem:periodicity_formal_pattern} we deduce
\[\forall n\geq q,\qquad 
\boldsymbol{\nu}^{\Phi}_{\infty}[(1+n!)N_*](\Gamma)=\boldsymbol{\nu}^\Phi_{\infty}[N_*](\Gamma).
\]
\end{proof}

\noi
The proof of Theorem~\ref{thm:image_absolute_normal_circle} follows from Lemma~\ref{lem:invariance_absolute_normal} and Theorem~\ref{thm:C1-density}.

\bibliographystyle{plain}
\bibliography{H2}
\end{document}